\documentclass[10pt,a4paper,reqno]{amsart}
\usepackage[T1]{fontenc}
\usepackage{lmodern}
\usepackage{amsmath,amssymb,amsthm,mathtools,microtype,indentfirst,cite,etoolbox}
\usepackage{tikz,placeins}
\usetikzlibrary{arrows.meta,calc}
\definecolor{geoblue}{RGB}{31,91,127}
\definecolor{geoamber}{RGB}{190,109,32}
\usepackage[hidelinks]{hyperref}
\hypersetup{pdftitle={Scalar curvature integrals on geodesic balls},pdfauthor={Qixuan Hu}}
\allowdisplaybreaks[2]
\numberwithin{equation}{section}
\newtheorem{theorem}{Theorem}[section]
\newtheorem{proposition}[theorem]{Proposition}
\newtheorem{definition}[theorem]{Definition}
\newtheorem{lemma}[theorem]{Lemma}
\newtheorem{corollary}[theorem]{Corollary}
\newtheorem{conjecture}[theorem]{Conjecture}
\newtheorem*{quotedtheorem}{Theorem}
\theoremstyle{remark}
\newtheorem{remark}[theorem]{Remark}
\DeclareMathOperator{\Ric}{Ric}
\DeclareMathOperator{\Hess}{Hess}
\DeclareMathOperator{\Vol}{Vol}
\DeclareMathOperator{\Lip}{Lip}
\DeclareMathOperator{\diver}{div}
\DeclareMathOperator{\tr}{tr}
\newcommand{\R}{\mathbb R}
\newcommand{\Sph}{\mathbb S}
\newcommand{\dd}{\,\mathrm d}
\newcommand{\ip}[2]{\langle #1,#2\rangle}
\newcommand{\norm}[1]{\lvert #1\rvert}

\title[Scalar curvature integrals on geodesic balls]
{Scalar curvature integrals on geodesic balls}
\author{Qixuan Hu}
\address{Department of Mathematics and Computer Sciences,
Shantou University, China}
\email{qxhu@stu.edu.cn}
\thanks{This work was supported by SRIG under Grant NTF25026T.}
\date{\today}

\begin{document}
\begin{abstract}
We prove bounds for scalar curvature integrals over geodesic balls on
complete noncompact manifolds with nonnegative sectional curvature.
In every dimension $n\ge3$, when the asymptotic volume ratio is close
to one, we obtain upper and lower estimates with the sharp first-order
coefficient and a quadratic remainder. The leading terms
involve the asymptotic and local volume deficits, respectively.
Rotationally symmetric examples show that both first-order
coefficients are optimal and that a quadratic remainder is necessary
in the upper bound. For dimensions three,
four, and five we give estimates for any positive asymptotic
volume ratio. When the center is a pole, we
give an estimate in dimension three under nonnegative Ricci
curvature and in every dimension $n\ge4$ under nonnegative sectional
curvature. This estimate is sharp when the asymptotic volume ratio
is at most $2^{-(n-1)}$.
\end{abstract}
\maketitle
\setcounter{tocdepth}{1}
\tableofcontents

\section{Introduction}
\label{sec:introduction}

Our main results concern smooth, connected, complete, noncompact
Riemannian manifolds $(M^n,g)$ without boundary. We write $K$,
$\Ric$, and $S$ for the sectional curvature, Ricci tensor, and
scalar curvature, respectively. 
Let $\omega_n$ denote the volume of the Euclidean unit ball in
$\R^n$. We use the notation
\begin{equation}\label{eq:notation}
 I_p(r)=\int_{B_p(r)}S\dd\mu_g,\qquad v_p(r)=\frac{\Vol_g(B_p(r))}{\omega_nr^n},\qquad  \rho(x)=d(p,x).
\end{equation}
Under $\Ric\ge0$, Bishop--Gromov comparison shows that
$v_p(r)$ is nonincreasing and that the asymptotic volume ratio
\[
v=V_{M^n}:=\lim_{r\to\infty}v_p(r)
\]
is independent of $p$. All balls are intrinsic geodesic balls. A point $p$ is a pole if
$\exp_p:T_pM\to M$ is a diffeomorphism.

The quantity $r^{2-n}I_p(r)$ is invariant under simultaneous
rescaling of the metric and radius. Naber
\cite[Conjecture~3.5]{Nab20} asked the following question:
on a complete manifold, does
\[
 \Ric\ge-\Lambda g
 \quad\Longrightarrow\quad
 \int_{B_p(1)}S\dd\mu_g\le C(n,\Lambda).
\]
The noncollapsing version
\cite[Conjecture~2.18]{Nab20} assumes
$\Vol(B_p(1))\ge\nu>0$ and allows dependence on $\nu$. This question is the local version of Yau's
scalar-curvature problem: $\limsup_{r\to\infty}r^{2-n}I_p(r)<\infty$ arising
from \cite[Problem~9]{Yau92} and the classical Cohn--Vossen
inequality \cite{CV35}.

Nonnegative Ricci curvature alone does not suffice, even in
dimension three. Hao and Zhu
\cite[Theorem~1.1 and Proposition~1.3]{HZ26} constructed a
complete metric on $\R^3$ with $\Ric>0$, $V_{M^3}=0$, and
$\limsup_{r\to\infty}\frac{I_p(r)}{r}=+\infty$ at every center.
Independently, Xu \cite[Theorem~1.1]{XuCollapse26} obtained such
an example with the full limit tending to infinity. Rescaling
gives unbounded unit-ball integrals. Xu also constructed
Ricci-positive metrics on $\Sph^3$ with diameter less than
$\frac{1}{2}$, volume tending to zero, and total scalar curvature tending
to infinity; each unit ball is then the whole manifold.
These examples disprove the local assertion without
noncollapsing, but do not disprove its noncollapsing version.

For fixed-radius integrals under a lower Ricci bound, an estimate
of Wenshuai Jiang and Naber, announced in
\cite[Theorem~2.17]{Nab20} and proved in the appendix of
Cucinotta--Mondino \cite[Theorem~5.15]{CM26}, states that
\[
 \frac{1}{\Vol(B_p(1))}\int_{B_p(1)}|\Ric|^q\dd\mu_g
 \le C(n,\Lambda,\nu,q),\qquad 0<q<1,
\]
if $\Ric\ge-\Lambda g$ and $\Vol(B_p(1))\ge\nu>0$.
Without noncollapsing, Smith \cite[Theorem~1.1]{Smith25} proved
\[
 \int_{B_p(1)}|\Ric|^q\dd\mu_g
 \le C(n,q)\Vol(B_p(1))^{1-2q},\qquad 0<q<\frac{1}{2},
\]
assuming $\Ric\ge-g$ on $B_p(5)$. These estimates do not give
the endpoint $L^1$ bound considered here.

The $L^1$ bound is known in several structured settings.
Liu--Sz\'ekelyhidi \cite{LS22} proved a unit-ball scalar-curvature
integral bound on polarized K\"ahler manifolds under a lower
Ricci bound and noncollapsing, and obtained small integrals on
half-balls when the unit ball is sufficiently close to Euclidean
in the Gromov--Hausdorff distance. In their recent preprint, Jiang and Liu \cite[Theorem~1.1]{JiangLiu26}
removed the polarization assumption: on complete K\"ahler
manifolds of real dimension $n=2m$, the conditions
$\Ric\ge-(n-1)g$ and $\Vol(B_p(1))\ge\nu>0$ imply a bound
for $\int_{B_p(1)}|\Ric|\dd\mu_g$ depending only on $n$ and
$\nu$. 

For asymptotic integrals in dimension three under $\Ric\ge0$,
Bo Zhu \cite[Theorem~1.7]{Zhu22} obtained an upper bound in the
presence of a pole, and Xu \cite[Theorem~1.4]{Xu24} gave the exact limit
$\lim_{r\to\infty}\frac{I_q(r)}{r}=8\pi(1-v)$. Green-function-weighted
analogues appear in \cite{Xu20,CXZ26}. Other asymptotic estimates
assume positive two-sided scalar-curvature bounds and one end
\cite{MW25}, quadratic scalar-curvature decay \cite{Deng26}, or
local conformal flatness \cite{Ma26}.

For arbitrary centers under a sectional-curvature lower bound,
Petrunin \cite[Theorem~1.1]{Pet09} proved the
unit-ball estimate $\int_{B_p(1)}S\dd\mu_g\le C(n)$ when
$K\ge-1$. Thus $K\ge0$ gives $I_p(r)\le C(n)r^{n-2}$ by
rescaling. Nan Li \cite{Li25} gave an alternative proof using a
quantitative covering argument. Lebedeva--Petrunin
\cite[Corollary~1.2]{LP24} established weak continuity of
scalar-curvature measures under noncollapsed Gromov--Hausdorff
convergence with a lower sectional-curvature bound.

Jiang--Li \cite[Theorem~1.2]{JL26} refined the fixed-radius
estimate under $K\ge0$ to
\[
 I_p(r)\le C_1(n)(1-v)r^{n-2}\qquad(p\in M,\ r>0).
\]
Their main argument treats $v$ sufficiently close to one, using
quantitative Alexandrov geometry and induction on dimension through
convex level hypersurfaces. The remaining
range follows from Petrunin's bound.
The resulting coefficient $C_1(n)$ is not given explicitly.
In dimensions two and three, they also proved the local
volume-deficit estimate \cite[Theorem~1.5]{JL26}
\[
 \int_{B_p(\frac{r}{2})}S\dd\mu_g
 \le C r^{n-2}
       \left(1-\frac{\Vol(B_p(r))}{\omega_n r^n}\right).
\]

Our first result concerns manifolds $(M^n,g)$ with $K\ge0$ and
asymptotic volume ratio close to one. We give upper and lower bounds
for the scalar curvature integral over finite-radius balls, with optimal
first-order coefficients and quadratic errors. This sharpens the
near-maximal-volume upper estimate and gives a lower estimate whose
leading term is determined by the local volume deficit. Rotational
examples show that both first-order coefficients are optimal and that
a quadratic remainder is necessary in the upper estimate.

\begin{theorem}\label{thm:near-maximal}
For every integer $n\ge3$, there are constants $C_n,\,\delta_n>0$, depending only on $n$, such that if $(M^n,g)$ has
$K\ge0$ and $0\le1-v\le\delta_n$, then for any $p\in M$ and $r>0$,
\begin{equation}\label{eq:near-maximal}
 \begin{aligned}
 I_p(r)&\le\bigl[2n\omega_n(1-v)+C_n(1-v)^2\bigr]r^{n-2},\\
 I_p(r)&\ge\bigl[2n\omega_n\bigl(1-v_p(r)\bigr)
                    -C_n(1-v)^2\bigr]r^{n-2}.
 \end{aligned}
\end{equation}
Moreover, the first-order coefficient $2n\omega_n$ in the upper and lower estimates is optimal and a quadratic remainder in the upper estimate is necessary.
\end{theorem}

We first construct coordinates on a sufficiently large ball in which
$g=e+q$ is uniformly close to the Euclidean metric $e$.
We then compare the scalar curvature integral over $B_p(1)$ with
the Euclidean integral of the linearized scalar curvature of a
corrected perturbation $\widehat q$, which differs from $q$ by
 a quadratic error. Nonnegative sectional curvature allows
this correction to be chosen so that the linearized Ricci tensor
of $\widehat q$ is nonnegative. This yields a monotonicity formula
for the normalized integrals of its linearized scalar curvature.
A positive-kernel identity relates these integrals to the
linearized volume deficit, which agrees with the geometric volume
deficit up to a quadratic error. Applying the monotonicity formula on a
sufficiently large ball gives the upper bound with leading term
$2n\omega_n(1-v)$. On fixed comparison scales, the reverse comparison
and an integral bound for the correction give the lower bound with
leading term $2n\omega_n(1-v_p(1))$. Scaling then gives both estimates
at every radius.

In the rotationally symmetric case, we obtain a sharp upper bound
in terms of the asymptotic volume ratio and a sharp lower bound in
terms of the local volume ratio. 

\begin{theorem}\label{thm:rotational}
Let $(M^n,g)$, $n\ge3$, be rotationally symmetric about $p$,
with $\Ric\ge0$ and metric
$g=\mathrm dt^2+f(t)^2g_{\Sph^{n-1}}$ for $0\le t<\infty$.
Then $M$ has nonnegative sectional curvature. For every $r>0$,
writing $v=V_{M^n}$ and $\nu=v_p(r)$, we have
\begin{align}
 I_p(r)&\le2n(n-1)\omega_n
       \bigl(1-v^{\frac{1}{n-1}}\bigr)r^{n-2},\label{eq:rot-sharp}\\
 I_p(r)&\ge n(n-1)\omega_n
       \bigl(\nu^{\frac{n-3}{n-1}}-\nu\bigr)r^{n-2}.
       \label{eq:rot-local-lower}
\end{align}
For every fixed $n\ge3$ and $r>0$, the upper bound is sharp among
metrics with each prescribed asymptotic volume ratio $v\in[0,1]$,
and the lower bound is sharp among metrics with each prescribed
local volume ratio $\nu\in(0,1]$.
\end{theorem}

\begin{corollary}\label{cor:all-dimensional-sharpness}
For $n\ge3$ and $v\in[0,1)$, define
\begin{equation}\label{eq:Cn-definition}
 \mathcal C_n(v)=\sup\{\frac{I_p(1)}{1-v}:\dim M=n,\ K\ge0,\ V_{M^n}=v\},
\end{equation}
where the supremum is over pointed smooth, connected, complete,
noncompact manifolds $(M,g,p)$ without boundary. Then
\begin{equation}\label{eq:all-dimensional-sharp-limit}
 \lim_{v\uparrow1}\mathcal C_n(v)=2n\omega_n.
\end{equation}
More precisely, for $0<\delta=1-v\le\delta_n$,
\begin{equation}\label{eq:all-dimensional-sandwich-C}
 2n\omega_n+
   \frac{n(n-2)}{n-1}\omega_n\delta
 \le\mathcal C_n(1-\delta)
 \le2n\omega_n+C_n\delta.
\end{equation}
\end{corollary}

For a general positive asymptotic volume ratio, we give
estimates at arbitrary centers in dimensions three, four, and five.

\begin{theorem}\label{thm:main}
Let $(M^3,g)$ have $K\ge0$ and $v=V_{M^3}>0$.
Then, for every $p\in M$ and $r>0$,
\begin{equation}\label{eq:main}
 I_p(r)\le \frac{12\pi(1-v)}{v^3}r.
\end{equation}
\end{theorem}

\begin{theorem}\label{thm:explicit45}
	Let $(M^n,g)$ have $K\ge0$ and $0<v=V_{M^n}\le1$.
	For every $p\in M$ and $r>0$,
	\begin{align}
		I_p(r)&\le17280\frac{1-v}{v^3}r^2,
		&&n=4,\label{eq:explicit4}\\
		I_p(r)&\le116640\pi^2\frac{1-v}{v^{\frac{14}{3}}}r^3,
		&&n=5.\label{eq:explicit5}
	\end{align}
\end{theorem}

The three-dimensional proof compares the integral on a ball with
its asymptotic growth. The asymptotic value in dimension three is determined by Hu
\cite[Theorem~1.1 and equation~(1.4)]{Hu26} and Xu
\cite[Theorem~4.2]{Xu26},
\begin{equation}\label{eq:asymptotic-input}
	\lim_{t\to\infty}\frac{I_p(t)}{t}=8\pi(1-v),
	\qquad K\ge0,\quad v>0.
\end{equation}
This also gives a simpler proof of the upper estimate in
Theorem~\ref{thm:near-maximal} in dimension three. For that upper
estimate, one may take $C_3=200\pi$ and $\delta_3=\frac{1}{100}$.

In dimensions four and five we use strictly
convex level sets and the four-dimensional Gauss--Bonnet formulas.

When the center is a pole, we obtain sharper explicit bounds,
including the case $v=0$. 
\begin{theorem}\label{thm:pole}
\label{thm:piecewise}\label{thm:pole-higher}
Let $(M^n,g)$, $n\ge3$, have a pole $p$. Assume either that
$n=3$ and $\Ric\ge0$, or that $n\ge4$ and $K\ge0$.
Write $v=V_{M^n}\in[0,1]$. For every $r>0$, if
$0\le v\le2^{-(n-1)}$, then
\begin{equation}\label{eq:piecewise}
 I_p(r)\le2n(n-1)\omega_n
          \bigl(1-v^{\frac{1}{n-1}}\bigr)r^{n-2}.
\end{equation}
If $2^{-(n-1)}<v\le1$, then
\begin{equation}\label{eq:pole-higher}
 I_p(r)\le n(n-1)\omega_n
          \bigl(1+v^{\frac{n-2}{n-1}}-2v\bigr)r^{n-2}.
\end{equation}
The first bound is sharp for every exact value of $v$ in its
stated interval. In particular,
\begin{equation}\label{eq:pole-higher-linear}
 I_p(r)\le2n(n-1)\omega_nr^{n-2}.
\end{equation}
This is a sharp uniform bound without a positive volume-growth assumption.
\end{theorem}

For the pole estimate under $\Ric\ge0$ alone, the restriction to
dimension three is essential. Cheng
\cite[Theorem~1.1 and Proposition~5.2]{Ch26} constructed, for every
$n\ge4$, a complete metric on $\R^n$ with $\Ric\ge0$, a pole $o$,
and $V_{M^n}=0$, such that
\[
 \lim_{L\to\infty}L^{2-n}\int_{B_o(L)}S\dd\mu_g=+\infty.
\]
Rescaling preserves the pole and gives unbounded unit-ball integrals.
Thus, for $n\ge4$, the sectional-curvature hypothesis in
Theorem~\ref{thm:pole} cannot in general be replaced by $\Ric\ge0$.

Sharpness of \eqref{eq:pole-higher} is not asserted when
$2^{-(n-1)}<v<1$. The rotational bound \eqref{eq:rot-sharp} suggests that the first branch of
Theorem~\ref{thm:pole} should extend to every volume ratio.

\begin{conjecture}\label{conj:sharp-pole}
Let $(M^n,g)$, $n\ge3$, be a smooth, connected, complete,
noncompact Riemannian manifold without boundary, with a pole $p$.
Assume either that $n=3$ and $\Ric\ge0$, or that $n\ge4$ and
$K\ge0$. Then, for every $v=V_{M^n}\in[0,1]$ and every $r>0$,
\begin{equation}\label{eq:conjectured-pole}
 I_p(r)\le2n(n-1)\omega_n
          \bigl(1-v^{\frac{1}{n-1}}\bigr)r^{n-2}.
\end{equation}
\end{conjecture}

The pole assumption on the center cannot be omitted from the uniform bound $I_p(r)\le2n(n-1)\omega_nr^{n-2}$, even for convex
hypersurfaces. Appendix~\ref{sec:centers} constructs smooth
complete convex hypersurfaces in $\R^4$ with $K\ge0$, $V_{M^3}=0$, and unit-ball integrals
approaching
\[
 \mathcal E(\theta)=16\pi+8\pi\theta\cos\theta,
 \qquad0<\theta<\frac{\pi}{2}.
\]
The construction joins two flat copies of a smoothed
cap--cylinder domain by a convex band. Its largest limiting value is
$\mathcal E_0\approx64.367$, which exceeds $16\pi$.

We first establish the common Busemann preliminaries in
Section~\ref{sec:busemann-common}, and prove
Theorem~\ref{thm:near-maximal} in Section~\ref{sec:near-maximal}.
Section~\ref{sec:rotational} proves the rotational theorem and
Corollary~\ref{cor:all-dimensional-sharpness}. Section~\ref{sec:three-dimensional}
uses convex functions to give another proof of the upper bound in
Theorem~\ref{thm:near-maximal} in dimension three, as well as an estimate
for every positive asymptotic volume ratio.
Section~\ref{sec:sectional} treats general positive volume ratios
in dimensions four and five.
Section~\ref{sec:pole} proves both cases of Theorem~\ref{thm:pole}
by a common argument, after estimating the curvature integrals of
distance spheres separately in dimension three and in higher dimensions.
Appendix~\ref{sec:centers} gives a convex-hypersurface construction with large scalar curvature integral.

\section{Busemann functions and transfer of directions}
\label{sec:busemann-preliminaries}
\label{sec:busemann-common}

In this section $(M^n,g)$ has $K\ge0$, $n\ge3$, and $v>0$.
Let $s_n=n\omega_n$ denote the area of the unit sphere $\Sph^{n-1}$,
and write $\dd\theta_z$ for the area element
on the unit tangent sphere $S_zM$.

We first describe the asymptotic density and the Busemann functions
used in the construction. For $z\in M$, let $S_zM$ be the unit
tangent sphere, $c_z(u)$ the cut time, and $J_z(t,u)$ the polar
volume Jacobian, extended by zero for $t\ge c_z(u)$.
Polar volume comparison gives
\[
 0\le \frac{J_z(t,u)}{t^{n-1}}\le1,
 \qquad t\longmapsto \frac{J_z(t,u)}{t^{n-1}}
       \text{ is nonincreasing}.
\]
Define
\begin{equation}\label{eq:asymptotic-density}
 \mathcal R_z=\{u\in S_zM:c_z(u)=\infty\},
 \qquad j_z(u)=\lim_{t\to\infty}\frac{J_z(t,u)}{t^{n-1}}.
\end{equation}
Then $0\le j_z\le1$ and $j_z=0$ outside $\mathcal R_z$.
Changing variables $t=Rs$ and applying dominated convergence gives
\[
 \omega_nv
 =\lim_{R\to\infty}\int_{S_zM}\int_0^1
       \frac{J_z(Rs,u)}{R^{n-1}}\dd s\dd\theta_z(u)
 =\frac{1}{n}\int_{S_zM}j_z\dd\theta_z.
\]
In particular,
\begin{equation}\label{eq:density-mass}
 \int_{S_zM}j_z\dd\theta_z=s_nv.
\end{equation}

Fix $p\in M$ and define
\begin{equation}\label{gv:weighted-measure}
 \dd\nu_p(u)=\frac{j_p(u)}{s_nv}\dd\theta_p(u).
\end{equation}
It is a probability measure on $S_pM$, supported on $\mathcal R_p$.
All averages of the Busemann functions based at $p$ below are taken
with respect to this fixed measure.

For $u\in\mathcal R_p$, put $\gamma_u(t)=\exp_p(tu)$ and set
\[
 b_{u,t}(x)=t-d(x,\gamma_u(t)).
\]
The triangle inequality shows that $b_{u,t}$ is nondecreasing in
$t$, vanishes at $p$, and satisfies $-\rho\le b_{u,t}\le\rho$.
Since every $b_{u,t}$ is $1$-Lipschitz, the limit
\begin{equation}\label{eq:busemann-definition}
 b_u(x)=\lim_{t\to\infty}b_{u,t}(x)
\end{equation}
exists, is $1$-Lipschitz, and vanishes at $p$. The uniform
Lipschitz bound also makes this convergence uniform on compact
sets. On $B_p(L)$, for $t>L$, Hessian comparison gives
\[
 \Hess b_{u,t}\ge-\frac{1}{t-L}g
\]
distributionally. This includes the cut locus, since the singular
part of the Hessian of distance is nonpositive; see
\cite[Theorem~2.10 and Corollary~2.12]{MMU14}. Passing to the
limit gives $\Hess b_u\ge0$, and hence $b_u$ is geodesically
convex.

Fix $x\in M$ at which $b_u$ is differentiable. Given any
sequence $t_j\to\infty$, choose a unit-speed minimizing
geodesic $\sigma_j:[0,\ell_j]\to M$ from $x$ to
$\gamma_u(t_j)$, where
\[
 \ell_j=d(x,\gamma_u(t_j))\ge t_j-\rho(x)\longrightarrow\infty.
\]
By compactness of $S_xM$, after passing to a subsequence,
$\dot\sigma_j(0)\to\xi\in S_xM$. Smooth dependence of geodesics
on their initial data gives
\[
 \sigma_j(s)=\exp_x(s\dot\sigma_j(0))
 \longrightarrow\sigma(s):=\exp_x(s\xi)
\]
uniformly on every bounded interval of $s\ge0$. Since each
$\sigma_j$ is minimizing, passing to the limit gives
$d(\sigma(s_1),\sigma(s_2))=s_2-s_1$ for $0\le s_1<s_2$.
Thus $\sigma$ is a ray from $x$.
Moreover, for $0\le s\le\ell_j$,
\[
 t_j-d(\sigma_j(s),\gamma_u(t_j))
 =t_j-(\ell_j-s)=t_j-d(x,\gamma_u(t_j))+s.
\]
The local uniform convergence of $b_{u,t_j}$ to $b_u$,
together with the convergence of $\sigma_j$, implies
$b_u(\sigma(s))=b_u(x)+s$ for every $s\ge0$.
Differentiating at $s=0$
and using its Lipschitz bound yields
\[
 1=\ip{\nabla b_u(x)}{\xi}\le|\nabla b_u(x)|\le1.
\]
Consequently $|\nabla b_u(x)|=1$ and $\xi=\nabla b_u(x)$.
Every convergent subsequence has this same limit, so compactness
of $S_xM$ gives
\begin{equation}\label{eq:busemann-directions}
 u_x(\gamma_u(t))\longrightarrow\nabla b_u(x),
 \qquad |\nabla b_u(x)|=1,
\end{equation}
where $u_x(y)$ denotes any minimizing initial direction from
$x$ to $y$. This convergence holds for arbitrary
choices of $u_x(y)$ when $y=\gamma_u(t)$ is a cut point of $x$.

\begin{lemma}[Direction transfer]\label{lem:direction-transport}
For almost every $x\in M$, the gradient
$\nabla b_u(x)$ exists and has unit length for
$\nu_p$-almost every $u\in\mathcal R_p$, and
\begin{equation}\label{eq:direction-transport}
 \int_{\mathcal R_p}\Psi(\nabla b_u(x))\dd\nu_p(u)
 =\frac{1}{s_nv}\int_{S_xM}j_x(\xi)\Psi(\xi)\dd\theta_x(\xi)
\end{equation}
for every continuous function $\Psi:S_xM\to\R$.
The exceptional set of $x$ is independent of $\Psi$.
Equivalently, the distribution of $\nabla b_u(x)$ under $\nu_p$
has density $\frac{j_x}{v}\le\frac{1}{v}$ relative to normalized area on $S_xM$.
\end{lemma}

\begin{proof}
The Lipschitz bound and Fubini's theorem imply that, for almost
every $x\in M$, the function $b_u$ is differentiable at $x$
for $\dd\theta_p$-almost every $u\in\mathcal R_p$, hence also
for $\nu_p$-almost every $u$.
Fix such an $x$; the unit-length assertion follows from
\eqref{eq:busemann-directions}. Set $L=d(p,x)+1$. For $T>2L$, let
\[
 E_T=B_p(2T)\setminus\overline{B_p(T)},\qquad
 W_{T,L}=\frac{(2T+L)^n-(T-L)^n}{n}.
\]
In polar coordinates at $x$, define
\[
 f_{T,x}(\xi)=\frac{1}{W_{T,L}}
      \int_0^\infty\mathbf1_{E_T}(\exp_x(s\xi))J_x(s,\xi)\dd s.
\]
The radial integral is bounded above and below by the integrals
of $J_x(s,\xi)$ over $[T-L,2T+L]$ and $[T+L,2T-L]$,
respectively. From \eqref{eq:asymptotic-density}, we get $0\le f_{T,x}\le1$ and
$f_{T,x}(\xi)\to j_x(\xi)$ for every $\xi$.
Polar integration and dominated convergence therefore yield
\[
 \lim_{T\to\infty}\frac{1}{s_nvW_{T,L}}
      \int_{E_T}\Psi(u_x(y))\dd\mu_g(y)
 =\frac{1}{s_nv}\int_{S_xM}j_x\Psi\dd\theta_x.
\]
Compute the same limit in polar coordinates at $p$, with
$s=Tt$, $1\le t\le2$. For $u\in\mathcal R_p$ use
\eqref{eq:busemann-directions}; on other directions the
extended Jacobian is eventually zero. Since
$\frac{J_p(Tt,u)}{T^{n-1}}\le t^{n-1}$ and $\frac{W_{T,L}}{T^n}\to\frac{2^n-1}{n}$, dominated
convergence and \eqref{gv:weighted-measure} give the left-hand side of
\eqref{eq:direction-transport}. The choice of $x$ depended only
on the differentiability of the family $b_u$, so its exceptional
set is independent of $\Psi$.
\end{proof}

\section{Sharp first-order bounds near maximal volume growth}\label{hd:sec:first}
\label{sec:near-maximal}

Throughout this section $n\ge3$ and $K\ge0$. Put
\[
 \begin{gathered}
 B_n=40n^2,\qquad c_n=3\sqrt2\,n^{\frac{5}{2}},\qquad H_n=260n^2,
 \end{gathered}
\]
and choose
\[
 \delta_n=\min\left\{\frac{1}{2c_n},\frac{1}{2H_n}\right\}.
\]
We assume $0<\delta=1-v\le\delta_n$, and set
$s_n=n\omega_n$, $\alpha=\frac{\delta}{v}$, and $\beta=B_n\delta$.
In particular,
\[
 \alpha\le2\delta\le1,\qquad \beta\le\frac{1}{8},
 \qquad c_n\delta\le\frac{1}{2},\qquad H_n\delta\le\frac{1}{2}.
\]
The constants $B_n,c_n,H_n$ remain fixed. The notation $C_n$
denotes a positive constant depending only on $n$, whose value may
increase between estimates. Repeated indices in tensor contractions
are summed unless otherwise indicated; explicit summation signs are
retained for indexed families and restricted index ranges.
We use the Busemann functions based at $p$ from
Section~\ref{sec:busemann-common}, with the probability measure
$\nu_p$ of \eqref{gv:weighted-measure}. At differentiability points put
$\xi_u(x)=\nabla b_u(x)\in T_xM$; by
\eqref{eq:busemann-directions}, $|\xi_u(x)|=1$. Define
\[
 h(x)=\int_{\mathcal R_p}b_u(x)\dd\nu_p(u).
\]
Define the self-adjoint endomorphism $M_p$ of $T_pM$ by
\begin{equation}\label{hd:eq:basepoint-tensor}
 M_p=\int_{\mathcal R_p}u\otimes u\dd\nu_p(u).
\end{equation}
Here $u\otimes u$ acts by $w\mapsto\langle u,w\rangle u$.

\subsection{Coordinates with a common Hessian bound}

\begin{lemma}\label{hd:lem:busemann-average}
The function $h$ is Lipschitz and geodesically convex, with $\Lip(h)\le\alpha$.
The endomorphism $M_p$ satisfies
\begin{equation}\label{hd:eq:basepoint-tensor-bounds}
 \frac{1-n\delta}{nv}\operatorname{Id}
 \le M_p\le\frac{1}{nv}\operatorname{Id},
\end{equation}
where $\operatorname{Id}$ is the identity on $T_pM$.
\end{lemma}

\begin{proof}
The probability measure $\nu_p$ is independent of $x$. Since each $b_u$
is $1$-Lipschitz and vanishes at $p$,
\[
 |h(x)-h(y)|\le\int|b_u(x)-b_u(y)|\dd\nu_p(u)\le d(x,y),
 \qquad h(p)=0.
\]
Hence $h$ is globally Lipschitz. It is geodesically convex because
it is an average of the convex functions
$b_u$. Fubini's theorem permits weak differentiation under the
integral. At almost every $x$, for each $w\in S_xM$,
Lemma~\ref{lem:direction-transport} therefore gives
\[
 \begin{aligned}
 \langle\nabla h(x),w\rangle
 &=\int_{\mathcal R_p}\langle\xi_u(x),w\rangle\dd\nu_p(u)\\
 &=\frac{1}{s_nv}\int_{S_xM}
       j_x(\xi)\langle\xi,w\rangle\dd\theta_x(\xi).
 \end{aligned}
\]
The spherical average satisfies
$\int_{S_xM}\langle\xi,w\rangle\dd\theta_x=0$. Moreover,
$0\le j_x\le1$ and \eqref{eq:density-mass} give
$\int_{S_xM}(1-j_x)\dd\theta_x=s_n\delta$. Consequently,
\[
 \begin{aligned}
 \left|\int_{S_xM}j_x(\xi)\langle\xi,w\rangle\dd\theta_x(\xi)\right|
 &=\left|\int_{S_xM}(1-j_x(\xi))
                   \langle\xi,w\rangle\dd\theta_x(\xi)\right|\\
 &\le s_n\delta.
 \end{aligned}
\]
Taking the supremum over unit $w$ proves
$|\nabla h|\le\frac{\delta}{v}=\alpha$ almost everywhere. The bound on the
weak gradient of the Lipschitz function $h$ implies
$\Lip(h)\le\alpha$. In particular,
$|h(x)|\le\alpha d(p,x)$.

For the bounds on $M_p$, let $w\in S_pM$. Rotational symmetry gives
\[
 \int_{S_pM}\langle u,w\rangle^2\dd\theta_p(u)=\frac{s_n}{n}.
\]
Using \eqref{gv:weighted-measure}, we obtain
\[
 \begin{aligned}
 g_p(M_pw,w)
 &=\int_{\mathcal R_p}\langle u,w\rangle^2\dd\nu_p(u)\\
 &=\frac{1}{nv}-\frac{1}{s_nv}\int_{S_pM}(1-j_p(u))
                      \langle u,w\rangle^2\dd\theta_p(u).
 \end{aligned}
\]
Since $0\le\langle u,w\rangle^2\le1$ and
$\int_{S_pM}(1-j_p)\dd\theta_p=s_n\delta$, the last integral
lies between $0$ and $s_n\delta$. This proves
\eqref{hd:eq:basepoint-tensor-bounds}.
\end{proof}

Identify $T_pM$ isometrically with $\R^n$ and define the
vector-valued Busemann average
\begin{equation}\label{hd:eq:weighted-coords}
 F=n\int_{\mathcal R_p}u\,b_u\dd\nu_p(u)=(f_1,\ldots,f_n).
\end{equation}

\begin{lemma}[Busemann coordinates]\label{hd:lem:coords-rough}
The functions $nh\pm f_a$ are geodesically convex for
$a=1,\ldots,n$. The map $F:M\to\R^n$ is a global homeomorphism and
\begin{equation}\label{hd:eq:bilip}
 (1-20n^2\delta)d(x,y)\le\norm{F(x)-F(y)}
       \le(1+20n^2\delta)d(x,y)
 \qquad(x,y\in M).
\end{equation}
\end{lemma}

\begin{proof}
Since every $b_u$ is $1$-Lipschitz and convex, and $|u|=1$, the definition gives
$\Lip(F)\le n$ and $F(p)=0$.
Writing $u=(u_1,\ldots,u_n)$, we have
$nh\pm f_a=n\int(1\pm u_a)b_u\dd\nu_p$. The weights are
nonnegative, so these functions are convex.

We first establish an initial-velocity estimate using convexity.
Suppose $F$ is differentiable at $x$, and let
$\gamma:[0,\ell]\to M$ be any unit-speed geodesic starting at $x$,
with $\ell>0$.
For each $u$, convexity of $\phi_u=b_u\circ\gamma$ gives a right
derivative $\phi'_{u,+}(0)$ and a nonnegative remainder
\[
 E_u=\phi_u(\ell)-\phi_u(0)-\ell\phi'_{u,+}(0)\ge0.
\]
Since $|\phi'_{u,+}(0)|\le1$, dominated convergence permits taking
right derivatives under both Busemann averages. Thus
\[
 \begin{aligned}
 F(\gamma(\ell))-F(x)-\ell\,dF_x(\dot\gamma(0))
     &=n\int uE_u\dd\nu_p(u),\\
 \int E_u\dd\nu_p(u)
     &=h(\gamma(\ell))-h(x)-\ell(h\circ\gamma)'_+(0)
       \le2\alpha\ell.
 \end{aligned}
\]
The last inequality uses $\Lip(h)\le\alpha$,
$d(x,\gamma(\ell))\le\ell$, and
$|(h\circ\gamma)'_+(0)|\le\alpha$. Since $|u|=1$, we obtain
\begin{equation}\label{hd:eq:initial-velocity}
 \norm{F(\gamma(\ell))-F(x)-\ell\,dF_x(\dot\gamma(0))}
 \le n\int E_u\dd\nu_p(u)\le2n\alpha\ell.
\end{equation}
Since $\int_{S_pM}u\dd\theta_p(u)=0$, we get
\[
 \begin{aligned}
 \left\lVert\int_{\mathcal R_p}u\dd\nu_p(u)\right\rVert
 &=\frac{1}{s_nv}\left\lVert\int_{S_pM}u(1-j_p(u))\dd\theta_p(u)
                    \right\rVert\\
 &\le\frac{1}{s_nv}\int_{S_pM}(1-j_p(u))\dd\theta_p(u)=\alpha.
 \end{aligned}
\]
Suppose a minimizing segment from $p$ to $z$ has length $r$ and
initial direction $\theta$. The hinge form of Toponogov comparison
\cite[Theorem~12.2.2]{Petersen16}, applied to this segment and
$\gamma_u|_{[0,t]}$, gives
\[
 d(z,\gamma_u(t))^2\le r^2+t^2-2rt\langle u,\theta\rangle.
\]
Subtracting from $t$ after taking square roots and letting
$t\to\infty$ yields $b_u(z)\ge r\langle u,\theta\rangle$.
The average of this
nonnegative difference satisfies
\[
 \int\bigl(b_u(z)-r\langle u,\theta\rangle\bigr)\dd\nu_p(u)
 =h(z)-r\left\langle\int u\dd\nu_p(u),\theta\right\rangle
 \le2\alpha r.
\]
Multiplication by $nu$ and integration therefore yield
\begin{equation}\label{hd:eq:radial-F}
 \norm{F(z)-nrM_p\theta}\le2n\alpha r.
\end{equation}
The upper bound on $M_p$ therefore gives the global growth estimate
\[
 \norm{F(z)}\le\left(\frac{1}{v}+2n\alpha\right)d(p,z)
 \qquad(z\in M).
\]

Fix $x$ in the full-measure set where $F$ is differentiable,
$b_u$ is differentiable for $\nu_p$-almost every $u$, and
differentiation under the integral gives
$dF_x(w)=n\int u\langle\xi_u,w\rangle\dd\nu_p(u)$.
For a unit vector $w\in T_xM$, completeness allows us to take
$\gamma(t)=\exp_x(tw)$ for every $t\ge0$. Apply
\eqref{hd:eq:initial-velocity} on $[0,T]$ and use the preceding
growth estimate together with $d(p,\gamma(T))\le d(p,x)+T$:
\[
\begin{aligned}
 \norm{dF_x(w)}
 &\le\frac{\norm{F(\gamma(T))-F(x)}}{T}+2n\alpha\\
 &\le\left(\frac{1}{v}+2n\alpha\right)
       \frac{d(p,x)+T}{T}+\frac{\norm{F(x)}}{T}+2n\alpha.
\end{aligned}
\]
Letting $T\to\infty$ yields
\begin{equation}\label{eq:df-upper-bound}
 \norm{dF_x(w)}\le\frac{1}{v}+4n\alpha=1+(4n+1)\alpha
 \qquad(\norm w_g=1).
\end{equation}

To obtain a lower bound, first consider the directions $\xi_u$.
For almost every ray direction $u$, apply
\eqref{hd:eq:initial-velocity} to a minimizing segment from $x$
to $\gamma_u(T)$. Its length divided by $T$ tends to one, and
its initial direction tends to $\xi_u$ by
\eqref{eq:busemann-directions}.
Divide the estimate by $T$, use \eqref{hd:eq:radial-F} at
$\gamma_u(T)$, and let $T\to\infty$. This gives
\[
|dF_x(\xi_u)-nM_pu|\le4n\alpha.
\]
Taking the inner product with $u$ and using
\eqref{hd:eq:basepoint-tensor-bounds} gives
\[
\begin{aligned}
	\langle dF_x(\xi_u),u\rangle
	&\ge n\int\langle u,z\rangle^2\dd\nu_p(z)-4n\alpha\\
	&\ge\frac{1-n\delta}{v}-4n\alpha
	=1-(5n-1)\alpha\ge1-5n\alpha.
\end{aligned}
\]

Now fix an arbitrary unit vector $w\in T_xM$ and extend it
to a $g$-orthonormal basis $e_1=w,e_2,\ldots,e_n$.
Differentiating \eqref{hd:eq:weighted-coords} gives
\[
\frac{1}{n} dF_x(e_i)=\int u\langle\xi_u,e_i\rangle\dd\nu_p(u).
\]
Linearity of $dF_x$ gives
\[
dF_x(\xi_u)
=\langle\xi_u,e_i\rangle\,dF_x(e_i).
\]
Therefore
\[
\begin{aligned}
	\int\langle dF_x(\xi_u),u\rangle\dd\nu_p(u)
	&=
	\left\langle dF_x(e_i),
	\int u\langle\xi_u,e_i\rangle\dd\nu_p(u)\right\rangle\\
	&=\frac{1}{n}\langle dF_x(e_i),dF_x(e_i)\rangle.
\end{aligned}
\]
Since $\langle dF_x(\xi_u),u\rangle\ge1-5n\alpha$ and $\nu_p$ is a
probability measure, we obtain
\[
\langle dF_x(e_i),dF_x(e_i)\rangle
\ge n(1-5n\alpha).
\]
By \eqref{eq:df-upper-bound}, each of the remaining $n-1$ terms is
at most $\bigl(1+(4n+1)\alpha\bigr)^2$.
Our fixed threshold gives $\delta\le\frac{1}{2H_n}=\frac{1}{520n^2}$;
with $\alpha\le2\delta$, this implies $(4n+1)\alpha\le \frac{1}{2}$.
Using $(1+s)^2\le1+\frac{5}{2}s$ for $0\le s\le \frac{1}{2}$, we obtain
\[
\begin{aligned}
	\norm{dF_x(w)}^2
	&\ge  n(1-5n\alpha)
	-(n-1)\bigl(1+(4n+1)\alpha\bigr)^2\\
	&\ge1-\left[5n^2+\frac{5}{2}(n-1)(4n+1)\right]\alpha\\
	&\ge1-15n^2\alpha\ge1-30n^2\delta.
\end{aligned}
\]
The same threshold gives $1-16n^2\delta>0$ and
\[
 (1-16n^2\delta)^2
 =1-32n^2\delta+256n^4\delta^2
 \le1-30n^2\delta.
\]
Taking square roots, and using \eqref{eq:df-upper-bound} and
$\alpha\le2\delta$ for the upper bound, yields
\begin{equation}\label{eq:df-bound}
 1-16n^2\delta\le\norm{dF_x(w)}\le1+(8n+2)\delta
 \qquad(\norm w_g=1).
\end{equation}

Let $y\in M$, and let $\gamma:[0,\ell]\to M$ be a unit-speed
minimizing geodesic from $x$ to $y$, where $\ell=d(x,y)$.
The estimate \eqref{hd:eq:initial-velocity}
gives
\[
\left|\norm{F(y)-F(x)}
-\ell\,\norm{dF_x(\dot\gamma(0))}\right|
\le2n\alpha\ell\le4n\delta\ell.
\]
Combining this with \eqref{eq:df-bound}, and using
$16n^2+4n\le20n^2$ and $12n+2\le20n^2$, proves
\eqref{hd:eq:bilip} for every $y$ and every
$x$ in the full-measure set chosen above. This set is dense, and
continuity of $F$ extends the bounds to all $x,y\in M$.

Since $20n^2\delta\le \frac{1}{8}$, the lower distance bound
is positive and $F$ is injective. Invariance of domain makes
$F(M)$ open in $\R^n$. It is also closed: if $F(x_j)$ converges,
the lower distance bound makes $(x_j)$ Cauchy, and completeness
gives a limit $x\in M$ with $F(x_j)\to F(x)$. Thus
$F(M)=\R^n$, and the same lower bound makes the inverse continuous.
\end{proof}

We now prepare the coordinates for curvature calculations.
The map $F$ is close to an isometry, but it may not be smooth.
Lemma~\ref{hd:lem:coords-smooth} gives smooth coordinates with the
same distance and velocity estimates. Lemmas~\ref{hd:lem:majorant}
and \ref{hd:lem:quadratic-algebra} then control the quadratic terms
in the curvature formula by the Hessian of a single convex function.
This will allow Proposition~\ref{hd:prop:correction} to replace the
true curvature by linear curvature expressions with a
$C_n\delta^2$ error in the underlying tensor.

\begin{lemma}[Smooth coordinates on finite scales]\label{hd:lem:coords-smooth}
For each $L>1$, there are smooth functions $h_*$ and
$y=(y^1,\ldots,y^n)$ on $M$, with $y(p)=0$, such that $y$
restricts to a coordinate map on $B_p(L)$ and
\begin{align}
 \norm{\nabla_g h_*}_g&\le3\delta,
      \label{hd:eq:smooth-gradient}\\
 Q:=\Hess_g h_*+\tau g&\ge0,\qquad \tau=\frac{\delta}{L},\qquad
       -n^2Q\le\Hess_g y^a\le n^2Q,\label{hd:eq:common-Hess}\\
 g_{ij}(y)&=\delta_{ij}+q_{ij}(y),\qquad \norm q\le4\beta,\qquad \beta=B_n\delta=40n^2\delta.
       \label{hd:eq:q-small}
\end{align}
The first two lines hold on $M$, and the last holds in the chart
$y(B_p(L))$, with the Euclidean operator norm. For
$x,z\in\overline{B_p(L)}$,
\[
 (1-\beta)d(x,z)\le\norm{y(z)-y(x)}
       \le(1+\beta)d(x,z).
\]
The coordinate image satisfies
\[
 B_{\frac{7L}{8}}^E\subset y(B_p(L))\subset B_{\frac{9L}{8}}^E.
\]
Moreover, every unit-speed geodesic $\gamma:[0,d]\to M$ of length
$d\le2L$ satisfies
\[
 \int_0^d\norm{(y\circ\gamma)''(s)}\dd s\le10n^2\delta.
\]
The geodesic need not stay inside $B_p(L)$.
\end{lemma}

\begin{proof}
We specify the approximation parameters before constructing the
smooth functions.
Fix
\[
 i_L=\inf_{x\in\overline{B_p(L)}}\operatorname{inj}_g(x)>0,
 \qquad \ell=\frac{1}{2}\min\{L,i_L\}.
\]
Let $\varepsilon=\min\left\{\frac{n\delta}{4},\frac{\delta\ell}{4\sqrt n}\right\}$.

\smallskip\noindent
\emph{Step 1: the global approximation.}
The functions $nh\pm f_a$ are convex on $M$.
Apply Greene--Wu \cite[Proposition~2.2]{GW79} with the strictly
negative Hessian lower bound $-n\tau g$ and the constant error
tolerance $\varepsilon$. Convex functions satisfy the corresponding
strict local lower-bound condition because $\tau>0$. The fine
approximation is on all of $M$ and gives
$G_a^\pm\in C^\infty(M)$ satisfying
\[
 \norm{G_a^\pm-(nh\pm f_a)}_{\infty,M}\le\varepsilon,\qquad
 P_a^\pm:=\Hess_gG_a^\pm+n\tau g\ge0.
\]
Define $y,h_*$ on $M$,
\[
 y^a=\frac{1}{2}(G_a^+-G_a^-)-\frac{1}{2}(G_a^+(p)-G_a^-(p)),\qquad
 h_* =\frac{1}{2n^2}\sum_a(G_a^++G_a^-).
\]
Since $\sum_a[(nh+f_a)+(nh-f_a)]=2n^2h$ and $F(p)=0$,
\[
 \norm{h_*-h}_{\infty,M}\le\frac{\varepsilon}{n},
 \qquad
 \norm{y-F}_{\infty,M}\le2\sqrt n\,\varepsilon,\qquad y(p)=0.
\]
Moreover,
\[
 Q=\frac{1}{2n^2}\sum_a(P_a^++P_a^-)\ge0,\qquad
 \Hess_g y^a=\frac{1}{2}(P_a^+-P_a^-).
\]
These identities immediately give \eqref{hd:eq:common-Hess}.

\smallskip\noindent
\emph{Step 2: the gradient bound.}
For $x\in M$ and $\norm w_g=1$, let
$\gamma(t)=\exp_x(tw)$ for $0\le t\le1$. Put
$\phi=h_*\circ\gamma$. Since $\gamma$ is a unit-speed geodesic,
$\phi''=\Hess_g h_*(\dot\gamma,\dot\gamma)\ge-\tau$.
Integrating twice gives
\[
 dh_*(w)=\phi'(0)
 =\phi(1)-\phi(0)-\int_0^1(1-t)\phi''(t)\dd t
 \le h_*(\gamma(1))-h_*(x)+\frac{\tau}{2}.
\]
Using $\Lip(h)\le\alpha$ and the uniform error
$\norm{h_*-h}_\infty\le\frac{\varepsilon}{n}$, we obtain
\[
 \begin{aligned}
 dh_*(w)
 &\le h_*(\gamma(1))-h_*(x)+\frac{\tau}{2}\\
 &\le\alpha+\frac{2\varepsilon}{n}+\frac{\tau}{2}
 \le2\delta+\frac{\delta}{2}+\frac{\delta}{2}=3\delta.
 \end{aligned}
\]
Taking the supremum over $g$-unit vectors $w$ proves
\eqref{hd:eq:smooth-gradient} globally. 

\smallskip\noindent
\emph{Step 3: control coordinate velocities and distances.}
For any unit-speed geodesic $\gamma:[0,d]\to M$, positivity of
$P_a^\pm$ gives
\[
 \begin{aligned}
 \norm{(y\circ\gamma)''}
 &\le\sum_a\left|(y^a\circ\gamma)''\right|\\
 &\le\frac{1}{2}\sum_a(P_a^++P_a^-)(\dot\gamma,\dot\gamma)
 =n^2\bigl((h_*\circ\gamma)''+\tau\bigr).
 \end{aligned}
\]
Hence, if $d\le2L$, the total variation is bounded by
\[
 \int_0^d\norm{(y\circ\gamma)''}\dd t
 \le n^2\bigl(2\norm{\nabla_g h_*}_{g,\infty}+\tau d\bigr)
 \le8n^2\delta.
\]

Now let $x\in\overline{B_p(L)}$ and $\norm w_g=1$.
The segment $\gamma(s)=\exp_x(sw)$, $0\le s\le\ell$, is minimizing.
Write $Y=y\circ\gamma$, so that $Y'(0)=dy_x(w)$. The triangle inequality gives
\[
 \begin{aligned}
 \left\lVert dy_x(w)
       -\frac{F(\exp_x(\ell w))-F(x)}{\ell}\right\rVert
 &\le\left\lVert Y'(0)-\frac{Y(\ell)-Y(0)}{\ell}\right\rVert\\
 &\quad+\left\lVert\frac{Y(\ell)-Y(0)}{\ell}
       -\frac{F(\gamma(\ell))-F(x)}{\ell}\right\rVert\\
 &\le\frac{1}{\ell}\int_0^\ell\norm{Y'(0)-Y'(s)}\dd s
       +\frac{2\norm{y-F}_\infty}{\ell}\\
 &\le8n^2\delta+\frac{4\sqrt n\,\varepsilon}{\ell}
 \le10n^2\delta.
 \end{aligned}
\]
By \eqref{hd:eq:bilip}, this bounds $\norm{dy_x(w)}$ between
$1-30n^2\delta$ and $1+30n^2\delta$.
For distinct $x,z\in\overline{B_p(L)}$, a minimizing segment $\gamma$
between them has length $0<d\le2L$. The same calculation as above gives
$\norm{dy_x(\dot\gamma(0))-\frac{y(z)-y(x)}{d}}\le10n^2\delta$.
Consequently, the fixed choices $B_n=40n^2$ and $\beta=B_n\delta$ give
\[
 (1-\beta)d(x,z)\le\norm{y(z)-y(x)}
       \le(1+\beta)d(x,z).
\]
The bounds for $dy$ and the distance, together with
$\beta\le \frac{1}{8}$, show that $y$ is injective and locally invertible
on $B_p(L)$.
Put $e=\dd y^i\otimes\dd y^i$. Since $30n^2\delta\le\beta$, the preceding estimate for
$dy_x(w)$ extends by homogeneity to
\[
(1-\beta)\norm V_g
\le\norm{dy_x(V)}_E
\le(1+\beta)\norm V_g
\qquad(x\in B_p(L),\ V\in T_xM).
\]
Since $\norm{dy_x(V)}_E^2=e_x(V,V)$, these inequalities imply
\[
(1-\beta)^2g\le e\le(1+\beta)^2g,
\qquad\text{hence}\qquad
(1+\beta)^{-2}e\le g\le(1-\beta)^{-2}e.
\]
From $0\le \beta\le \frac{1}{8}$, we have
\[
1-(1+\beta)^{-2}\le4\beta,
\qquad
(1-\beta)^{-2}-1\le4\beta.
\]
Consequently,
\[
-4\beta\,e\le g-e\le4\beta\,e.
\]
With $q=g-e$, this proves the Euclidean operator-norm bound
\[
\norm q\le4\beta.
\]

Since $y(p)=0$, the upper distance bound gives
$\norm{y(x)}\le\frac{9}{8}d(p,x)<\frac{9L}{8}$ for $x\in B_p(L)$.
For $z\in\R^n$ with $\norm z<\frac{7L}{8}$, continue the local inverse of
$y$ along $s\mapsto sz$, starting at $p$. Wherever the lift $x_s$
is defined, the lower distance bound gives
\[
 d(p,x_s)\le\frac{s\norm z}{1-\beta}\le\frac{\norm z}{1-\beta}=:r_z<L.
\]
Thus the lift remains in the compact set
$\overline{B_p(r_z)}\subset B_p(L)$. The lower distance bound also gives
$d(x_s,x_{s'})\le |s-s'|\frac{\norm z}{1-\beta}$, so the lift is Cauchy at
any finite endpoint and has a limit in this compact set. Local invertibility of $y$ at that
limit continues the lift. It therefore reaches $s=1$, proving
$B_{\frac{7L}{8}}^E\subset y(B_p(L))$.
\end{proof}

\subsection{A quadratic correction of the coordinate tensor}

Write $e=\dd y^i\otimes\dd y^i$ for the Euclidean background
metric in the smooth coordinates $y$. All derivatives and contractions
with a subscript $E$ below are taken with respect to $e$; matrix
norms in these coordinates are Euclidean operator norms.
The tensor $Q$ is nonnegative, but it is defined using the
Riemannian Hessian. The next lemma replaces it by a Euclidean
Hessian, which is the form needed in the curvature calculation.

\begin{lemma}[Convex control of the Christoffel symbols]\label{hd:lem:majorant}
There is a smooth function $\varphi$ on $B_p(L)$ such that, with
$P=\Hess_E \varphi$,
\begin{equation}\label{hd:eq:majorant}
 P\ge\frac{1}{2}Q\ge0,\qquad \norm{D_E\varphi}\le 10\delta,\qquad
       -2n^2P\le\Gamma^a\le2n^2P.
\end{equation}
Here $\Gamma^a=(\Gamma^a_{ij})$ is the symmetric matrix of
Christoffel symbols with upper index $a$ in the fixed coordinates
$y$. In these coordinates $(\Hess_E \varphi)_{ij}=\partial_i\partial_j\varphi$.
These inequalities hold throughout the coordinate image $y(B_p(L))$;
in particular, the coordinate expression of $\varphi$ is Euclidean convex
on $B_{\frac{7L}{8}}^E$.
\end{lemma}

\begin{proof}
Write $\partial_i=\frac{\partial}{\partial y^i}$.
Since
$\nabla^g_{\partial_i}\partial_j=\Gamma^a_{ij}\partial_a$
and the Euclidean connection has zero Christoffel symbols in $y$,
the definition of the Hessian gives
\[
 (\Hess_E h_*)_{ij}=\partial_i\partial_jh_*,\qquad
 (\Hess_g h_*)_{ij}
 =\partial_i\partial_jh_*-\Gamma^a_{ij}\partial_a h_*.
\]
Applying the same formula to $y^a$ gives
$(\Hess_g y^a)_{ij}=-\Gamma^a_{ij}$, so
\eqref{hd:eq:common-Hess} implies
$-n^2Q\le\Gamma^a\le n^2Q$. Subtracting the two Hessian formulas yields
\[
 \Hess_E h_*=\Hess_g h_*+(\partial_a h_*)\Gamma^a
        =Q-\tau g+(\partial_a h_*)\Gamma^a.
\]
The global gradient bound in \eqref{hd:eq:smooth-gradient} uses $g$.
To convert it to the Euclidean norm, note that
\eqref{hd:eq:q-small} and $4\beta\le \frac{1}{2}$ give $g\le2e$,
hence $g^{-1}\ge\frac{1}{2}e^{-1}$. Therefore, on $B_p(L)$,
\[
 \norm{D_Eh_*}^2=(\partial_i h_*)(\partial_i h_*)
 \le2g^{ij}(\partial_i h_*)(\partial_j h_*)
 =2\norm{\nabla_g h_*}_g^2\le18\delta^2.
\]
For every tangent vector $V$, we consequently have
\[
 \begin{aligned}
 \left|(\partial_a h_*)\Gamma^a(V,V)\right|
 &\le n^2\sum_a|\partial_a h_*|\,Q(V,V)\\
 &\le n^{\frac{5}{2}}\norm{D_Eh_*}\,Q(V,V)
 \le c_n\delta\,Q(V,V),
 \qquad c_n=3\sqrt2\,n^{\frac{5}{2}}.
 \end{aligned}
\]
Thus, as inequalities of symmetric bilinear forms,
\[
 (1-c_n\delta)Q-\tau g\le \Hess_E h_*\le(1+c_n\delta)Q-\tau g.
\]
The only negative term that is not a small multiple of $Q$ is
$-\tau g$. Remove it by setting
\[
 \varphi=h_*+\tau\norm y^2,\qquad
 P=\Hess_E \varphi\ge(1-c_n\delta)Q+\tau(2e-g).
\]
The fixed threshold gives $c_n\delta\le \frac{1}{2}$ and $g\le2e$.
Consequently, $P\ge\frac{Q}{2}\ge0$ and
$-2n^2P\le\Gamma^a\le2n^2P$ throughout $B_p(L)$.
Finally, the outer inclusion in Lemma~\ref{hd:lem:coords-smooth}
gives $\norm y\le\frac{9L}{8}$ on this ball. Since $\tau L=\delta$,
\[
 \norm{D_E\varphi}\le\norm{D_Eh_*}+2\tau\norm y
 \le3\sqrt2\,\delta+\frac{9}{4}\tau L
 \le\left(3\sqrt2+\frac{9}{4}\right)\delta\le10\delta.
\]
This proves \eqref{hd:eq:majorant} with the fixed threshold and
constants independent of $L$.
\end{proof}

We now introduce the linearized curvature operators. The formulas
below are the first variations of the curvature tensors at the
Euclidean metric; see \cite[Theorem~1.174(c)--(e), p.~63]{Besse87}.
We use $\Delta_E=\partial_i\partial_i$. 
Let $R_g$ be the covariant curvature tensor, with convention
\[
(R_g)_{ijkl}
=g\bigl(R^g(\partial_i,\partial_j)\partial_l,\partial_k\bigr).
\]
For a smooth symmetric tensor $a$, the first variation of $R_g$ at $e$
in the direction $a$ is denoted by
$R^{\mathrm{lin}}(a):=DR_e[a]=\left.\frac{d}{d\varepsilon}\right|_{\varepsilon=0}
R_{e+\varepsilon a}$. In coordinates,
\[
R^{\mathrm{lin}}_{ijkl}(a)=\frac{1}{2}\left(
\partial_i\partial_l a_{jk}+\partial_j\partial_k a_{il}
-\partial_i\partial_k a_{jl}-\partial_j\partial_l a_{ik}\right).
\]
Define the linearized Ricci curvature and linearized scalar curvature by
\begin{equation}\label{hd:eq:linear-Ric}
	\Ric^{\mathrm{lin}}(a):=D\Ric_e[a],\qquad
	S^{\mathrm{lin}}(a):=\tr_E\Ric^{\mathrm{lin}}(a).
\end{equation}
Since $R_e=0$, the Euclidean contraction of
$R^{\mathrm{lin}}(a)$ is $\Ric^{\mathrm{lin}}(a)$.
In coordinates,
\[
 \Ric^{\mathrm{lin}}(a)_{ij}
 =\frac{1}{2}\left(
   \partial_k\partial_i a_{jk}+\partial_k\partial_j a_{ik}
   -\Delta_Ea_{ij}-\partial_i\partial_j\tr_Ea\right),
\]
and taking the Euclidean trace gives
\begin{equation}\label{hd:eq:linear-S}
 S^{\mathrm{lin}}(a)
 =\partial_i\partial_j a_{ij}-\Delta_E\tr_Ea=\diver_E(\diver_Ea-d\tr_Ea).
\end{equation}
These are linear operators on tensors;
they are not the actual Ricci and scalar curvatures of $e+a$.

For the true curvature tensor, write
\[
 \Ric_E(R_g)(w,z):=R_g(e_i,w,e_i,z),
\]
where $(e_i)$ is an $e$-orthonormal basis. This contraction uses
$e^{-1}$, whereas the intrinsic Ricci tensor $\Ric_g$ contracts
$R_g$ with $g^{-1}$. We retain the subscript $E$ to distinguish them.

For the nonnegative Euclidean Hessian $P=\Hess_E \varphi$, define
\[
 A_P=(\tr_E P)P-P^2.
\]
If $\lambda$ is an eigenvalue of $P$, the corresponding eigenvalue
of $A_P$ is $\lambda(\tr_E P-\lambda)\ge0$. Thus $A_P\ge0$.
This tensor measures products of the Hessian in different directions.

\begin{lemma}[Ricci comparison with its linearization]\label{hd:lem:quadratic-algebra}
In the coordinates of Lemma~\ref{hd:lem:coords-smooth}, let $q=g-e$
and let $P=\Hess_E \varphi$ be as in Lemma~\ref{hd:lem:majorant}.
For every tangent vector $w$,
\begin{equation}\label{hd:eq:quadratic-algebra}
 \left|\bigl(\Ric_E(R_g)-\Ric^{\mathrm{lin}}(q)\bigr)(w,w)\right|
 \le8n^6 A_P(w,w).
\end{equation}
Equivalently, as inequalities of symmetric bilinear forms,
\[
 -8n^6A_P\le\Ric_E(R_g)-\Ric^{\mathrm{lin}}(q)\le8n^6A_P.
\]
\end{lemma}

\begin{proof}
\emph{Step 1: express the Ricci difference by Christoffel products.}
We first establish the exact identity
\[
 (R_g)_{ijkl}=R^{\mathrm{lin}}_{ijkl}(q)
       +g_{ab}(\Gamma^a_{il}\Gamma^b_{jk}-\Gamma^a_{ik}\Gamma^b_{jl}).
\]
Define the Christoffel symbols with the upper index lowered:
\[
 C_{ij,k}:=g_{ka}\Gamma^a_{ij}
 =\frac{1}{2}(\partial_i g_{jk}+\partial_j g_{ik}-\partial_k g_{ij}).
\]
The coordinate formula for curvature is
\[
 (R_g)_{ijkl}
 =g_{ka}\bigl(\partial_i\Gamma^a_{jl}-\partial_j\Gamma^a_{il}
       +\Gamma^a_{ib}\Gamma^b_{jl}-\Gamma^a_{jb}\Gamma^b_{il}\bigr).
\]
Using metric compatibility,
$\partial_i g_{ka}=g_{ab}\Gamma^b_{ik}+g_{kb}\Gamma^b_{ia}$,
we obtain
\[
 g_{ka}\bigl(\partial_i\Gamma^a_{jl}
                  +\Gamma^a_{ib}\Gamma^b_{jl}\bigr)
 =\partial_i C_{jl,k}-g_{ab}\Gamma^a_{ik}\Gamma^b_{jl}.
\]
Subtract the same expression with $i,j$ exchanged. The terms
with second derivatives satisfy
\[
 \partial_iC_{jl,k}-\partial_jC_{il,k}
 =R^{\mathrm{lin}}_{ijkl}(q),
\]
because $g_{ij}=\delta_{ij}+q_{ij}$ and $\delta_{ij}$ is constant.
This proves the curvature identity.

Contracting with $e$ gives, for every Euclidean orthonormal basis
$(e_i)$,
\begin{equation}\label{hd:eq:ricci-remainder}
 \begin{aligned}
 &\bigl(\Ric_E(R_g)-\Ric^{\mathrm{lin}}(q)\bigr)(w,w)\\
 &\quad=g_{ab}
 \left[
 \Gamma^a(e_i,w)\Gamma^b(e_i,w)
 -\Gamma^a(e_i,e_i)\Gamma^b(w,w)
 \right].
 \end{aligned}
\end{equation}
Here $\Gamma^a(V,W)=\Gamma^a_{jk}V^jW^k$ in the fixed
$y$-coordinates.

\smallskip\noindent
\emph{Step 2: bound the Christoffel products by $A_P$.}
By \eqref{hd:eq:q-small} and the fixed
smallness assumption, $\norm q\le4\beta\le \frac{1}{2}$. Hence, for each fixed pair $a,b$,
$|g_{ab}|\le\frac{1}{2}+\delta_{ab}$. Summing the squares gives
\[
 g_{ab}g_{ab}\le\frac{9}{4}n+\frac{1}{4}n(n-1)
 =\frac{n(n+8)}{4}\le n^2,
\]
where the last inequality uses $n\ge3$.

Fix Euclidean orthonormal vectors $w,z$. For a symmetric
bilinear form $A$, write
\[
\mathcal D(A)
=A(w,w)A(z,z)-A(w,z)^2.
\]
Since $P\ge0$, its restriction to the two-plane is positive semidefinite, so $\mathcal D(P)\ge0$. Lemma~\ref{hd:lem:majorant} gives
\[
-4n^2P\le\Gamma^a\pm\Gamma^b\le4n^2P.
\]
Consequently,
\[
\left|\mathcal D(\Gamma^a\pm\Gamma^b)\right|
\le16n^4\mathcal D(P).
\]
Indeed, let $H$ be the matrix of $P$ restricted to
$\operatorname{span}\{w,z\}$. If $H$ is positive definite,
the eigenvalues of
\[
H^{-\frac{1}{2}}
\bigl((\Gamma^a\pm\Gamma^b)|_{\operatorname{span}\{w,z\}}\bigr)
H^{-\frac{1}{2}}
\]
lie in $[-4n^2,4n^2]$, so taking determinants proves the
estimate. If $H$ is singular, replace it by
$H+\varepsilon\operatorname{Id}$ and let
$\varepsilon\downarrow0$.

For each fixed pair $a,b$, the polarization identity gives
\[
\begin{aligned}
&\left|
\Gamma^a(w,z)\Gamma^b(w,z)
-\frac{1}{2}\left[
\Gamma^a(w,w)\Gamma^b(z,z)
+\Gamma^b(w,w)\Gamma^a(z,z)
\right]\right|\\
&\qquad=
\frac{1}{4}\left|
\mathcal D(\Gamma^a+\Gamma^b)
-\mathcal D(\Gamma^a-\Gamma^b)
\right|\\
&\qquad\le8n^4\mathcal D(P).
\end{aligned}
\]
Since $g_{ab}=g_{ba}$, symmetrizing in $a,b$ does not change
the contracted expression. Cauchy--Schwarz over the $n^2$ pairs
of indices and the preceding bound therefore give
\[
\begin{aligned}
&\left|g_{ab}
\left[
\Gamma^a(z,w)\Gamma^b(z,w)
-\Gamma^a(z,z)\Gamma^b(w,w)
\right]\right|\\
&\qquad\le
8n^5(g_{ab}g_{ab})^{\frac{1}{2}}
\left[P(z,z)P(w,w)-P(z,w)^2\right]\\
&\qquad\le8n^6\left[P(z,z)P(w,w)-P(z,w)^2\right].
\end{aligned}
\]

The contraction over $i$ is independent of the Euclidean
orthonormal basis, so choose $e_1=w$. The term corresponding to
$e_1$ vanishes. Apply the preceding estimate separately to the
remaining basis vectors and contract over the full basis. Since
\[
 P(e_i,e_i)P(w,w)-P(e_i,w)P(e_i,w)
 =(\tr_E P)P(w,w)-P^2(w,w)=A_P(w,w),
\]
the identity \eqref{hd:eq:ricci-remainder} gives
\[
 \left|\bigl(\Ric_E(R_g)-\Ric^{\mathrm{lin}}(q)\bigr)(w,w)\right|
 \le8n^6A_P(w,w).
\]
This proves \eqref{hd:eq:quadratic-algebra} for Euclidean unit vectors.
The assertion for arbitrary $w$ follows by homogeneity.

\end{proof}

Lemma~\ref{hd:lem:quadratic-algebra} bounds the
nonlinear error in the Ricci comparison. The next proposition
compensates this error by adding a tensor of size $C_n\delta^2$.

\begin{proposition}[Quadratic correction]\label{hd:prop:correction}
For $q=g-e$ and the function $\varphi$ of Lemma~\ref{hd:lem:majorant}, set
\begin{equation}\label{hd:eq:qhat}
 \widehat q=q+8n^6\,d\varphi\otimes d\varphi.
\end{equation}
Then
\begin{equation}\label{hd:eq:correction-bounds}
 \begin{gathered}
 \norm{\widehat q-q}_\infty\le C_n\delta^2,\qquad
 \Ric^{\mathrm{lin}}(\widehat q)\ge0,\\
 S_g\sqrt{\det g}\le(1+C_n\delta)S^{\mathrm{lin}}(\widehat q).
 \end{gathered}
\end{equation}
In the other direction,
\begin{equation}\label{hd:eq:reverse-density}
 S^{\mathrm{lin}}(\widehat q)
 \le(1+C_n\delta)S_g\sqrt{\det g}+16n^6\tr_EA_P.
\end{equation}
\end{proposition}

\begin{proof}
\emph{Step 1: compensate the Ricci error.}
Lemma~\ref{hd:lem:quadratic-algebra} gives
\[
 \Ric^{\mathrm{lin}}(q)\ge\Ric_E(R_g)-8n^6A_P.
\]
Moreover, nonnegative sectional curvature implies
\[
 \Ric_E(R_g)(w,w)=R_g(e_i,w,e_i,w)\ge0,
\]
even though the basis $(e_i)$ is orthonormal for $e$, rather than $g$.
To compensate the error, substitute
$a_{ij}=\partial_i \varphi\,\partial_j \varphi$ into the formula for
$R^{\mathrm{lin}}$. The third derivatives cancel in pairs, giving
\[
 R^{\mathrm{lin}}_{ijkl}(d\varphi\otimes d\varphi)
 =P_{ik}P_{jl}-P_{il}P_{jk},\qquad
 \Ric^{\mathrm{lin}}(d\varphi\otimes d\varphi)=A_P.
\]
By linearity,
\[
 \Ric^{\mathrm{lin}}(\widehat q)
 =\Ric^{\mathrm{lin}}(q)+8n^6A_P
 \ge\Ric_E(R_g)\ge0.
\]
Also, Lemma~\ref{hd:lem:majorant} yields
\[
 \norm{\widehat q-q}_\infty
 \le8n^6\norm{D_E\varphi}_\infty^2
 \le800n^6\delta^2.
\]

\smallskip\noindent
\emph{Step 2: compare the scalar curvature densities.}
Taking the Euclidean trace of the preceding Ricci inequality gives
\[
 S^{\mathrm{lin}}(\widehat q)
 \ge R_g(e_i,e_j,e_i,e_j).
\]
At a given point, choose the $e$-orthonormal basis to diagonalize
$g$, and recall $\beta=B_n\delta\le \frac{1}{8}$. The metric comparison in
Lemma~\ref{hd:lem:coords-smooth} gives
\[
 (1+\beta)^{-2}\le\lambda_i\le(1-\beta)^{-2}
\]
for the eigenvalues of $g$. For each fixed $i$, set
$v_i=\frac{e_i}{\sqrt{\lambda_i}}$. This frame is $g$-orthonormal, and
nonnegative sectional curvature gives
\[
 \begin{aligned}
 R_g(e_i,e_j,e_i,e_j)
 &=2\sum_{i<j}\lambda_i\lambda_jK_g(v_i\wedge v_j)\\
 &\ge(1+\beta)^{-4}S_g.
 \end{aligned}
\]
Also $\sqrt{\det g}\le(1-\beta)^{-n}$. Consequently,
\[
 S_g\sqrt{\det g}
 \le(1+\beta)^4(1-\beta)^{-n}S^{\mathrm{lin}}(\widehat q)
 \le(1+C_n\delta)S^{\mathrm{lin}}(\widehat q).
\]
The last estimate is uniform for $0\le\beta\le \frac{1}{8}$, so it requires
no further smallness assumption. This proves
\eqref{hd:eq:correction-bounds}.

For the reverse comparison, take the Euclidean trace of both sides
of the error estimate in Lemma~\ref{hd:lem:quadratic-algebra}.
Since $S^{\mathrm{lin}}(d\varphi\otimes d\varphi)=\tr_EA_P$, we obtain
\[
 0\le S^{\mathrm{lin}}(\widehat q)
          -R_g(e_i,e_j,e_i,e_j)
   \le16n^6\tr_EA_P.
\]
The same eigenvalue bounds and $K\ge0$ give
\[
 \begin{aligned}
 R_g(e_i,e_j,e_i,e_j)
 &\le(1-\beta)^{-4}S_g\\
 &\le(1-\beta)^{-4}(1+\beta)^nS_g\sqrt{\det g}\\
 &\le(1+C_n\delta)S_g\sqrt{\det g}.
 \end{aligned}
\]
This proves \eqref{hd:eq:reverse-density}.
\end{proof}

\subsection{Linearized volume comparison}

We now compare the scalar-curvature integral with the volume deficit.
The two effects on volume are the change of volume density and the
movement of the boundary of a geodesic ball in the coordinate image.

We continue to use the fixed coordinates $y=(y^1,\ldots,y^n)$ and
the Euclidean background metric $e=\dd y^i\otimes\dd y^i$.
The coordinate vectors $\partial_i=\frac{\partial}{\partial y^i}$ are
$e$-orthonormal; they need not be $g$-orthonormal. All components
$g_{ij}$, $q_{ij}$, and $a_{ij}$ below are taken in this same frame.
Write $y=t\theta$, with $t=|y|$ and $\theta\in\Sph^{n-1}$.

At the point $t\theta$, the Euclidean unit radial vector is
$\partial_r=\theta_i\partial_i$. Thus, for a symmetric tensor $a$,
\[
a(t\theta)(\partial_r,\partial_r)= a_{ij}(t\theta)\theta_i\theta_j
,\qquad
 \tr_Ea(t\theta)=a_{ii}(t\theta).
\]
For radii for which the Euclidean sphere lies in the chart, define
\[
 \begin{aligned}
 A_a(t)&=\int_{\Sph^{n-1}}
             a_{ij}(t\theta)\theta_i\theta_j\dd\sigma(\theta),\\
 B_a(t)&=\int_{\Sph^{n-1}}\tr_Ea(t\theta)\dd\sigma(\theta).
 \end{aligned}
\]
The measure $\dd\sigma$ is the area measure on the unit sphere,
of total mass $s_n$.
For $a=q$, $A_q$ measures the radial length change, whereas $B_q$
measures the change of volume density. The Euclidean polar integration
formula is $\dd y=t^{n-1}\dd t\dd\sigma$.

Let $\rho$ also denote the coordinate expression of the distance from
$p$, so that $\rho(y(x))=d_g(p,x)$ for $x\in B_p(L)$. Define
\begin{equation}\label{hd:eq:volume-functional}
 D_a(R)=\frac{R^{n-1}\int_0^R A_a(t)\dd t
                  -\int_0^R B_a(t)t^{n-1}\dd t}{2\omega_nR^n}.
\end{equation}
The next lemma shows that $D_q(R)$ is the
first-order deficit of the normalized volume.

\begin{lemma}[Distance and volume expansion]\label{hd:lem:volume-expansion}
For $q=g-e$, $0<r<\frac{7L}{8}$, $\theta\in\Sph^{n-1}$, and
$0<R\le \frac{L}{2}$, we have
\begin{align}
 \rho(r\theta)&=r+\frac{1}{2}\int_0^r
                    q_{ij}(t\theta)\theta_i\theta_j\dd t+O_n(\delta^2r),
                 \label{hd:eq:distance-expansion}\\
 1-\frac{\Vol_gB_p(R)}{\omega_nR^n}
     &=D_q(R)+O_n(\delta^2),\qquad
 D_q(R)\le\delta+C_n\delta^2.\label{hd:eq:volume-expansion}
\end{align}
The constants are independent of $L$ and the comparison radius.
No derivative bound for $q$ is required.
\end{lemma}

\begin{proof}
Recall that $\beta=B_n\delta\le \frac{1}{8}$. Lemma~\ref{hd:lem:coords-smooth}
provides the chart on $B_{\frac{7L}{8}}^E$ and gives
\[
 \frac{|y|}{1+\beta}\le\rho(y)\le\frac{|y|}{1-\beta},
 \qquad \norm q\le4\beta,\qquad \norm g\le2.
\]
We first compute the distance, then solve for the boundary radius,
and finally integrate the volume density up to that boundary.

\smallskip\noindent
\emph{Step 1: expand the distance along a Euclidean ray.}
Fix $r\theta\in B_{\frac{7L}{8}}^E\setminus\{0\}\subset y(B_p(L))$ where $\rho$ is
differentiable. Let $\gamma:[0,d]\to M$ be a unit-speed minimizing
geodesic from $p$ to the corresponding point, with
$d=\rho(r\theta)<L$. Put $Y=y\circ\gamma$, so
$Y(0)=0$ and $Y(d)=r\theta$. The vector $\frac{r\theta}{d}$ is the
average coordinate velocity. Hence
\[
 \begin{aligned}
 Y'(d)-\frac{r\theta}{d}
 &=\frac{1}{d}\int_0^d\bigl(Y'(d)-Y'(s)\bigr)\dd s
   =\frac{1}{d}\int_0^d sY''(s)\dd s,\\
 \left|Y'(d)-\frac{r\theta}{d}\right|
 &\le\int_0^d|Y''(s)|\dd s\le10n^2\delta.
 \end{aligned}
\]
The last inequality is the velocity-variation estimate in
Lemma~\ref{hd:lem:coords-smooth}. The coordinate distance bound also gives
$1-\beta\le \frac{r}{d}\le1+\beta$. Consequently,
\[
 |Y'(d)-\theta|
 \le\left|Y'(d)-\frac{r\theta}{d}\right|
       +\left|\frac{r}{d}-1\right|
 \le10n^2\delta+\beta.
\]

At a differentiability point of the distance,
$\nabla_g\rho=\dot\gamma(d)$, whose coordinate components are
$Y'^j(d)$. Thus
\[
  D_E\rho(r\theta)=g(r\theta)Y'(d).
\]
Here $D_E\rho$ is the column vector $(\partial_i\rho)$ in the
fixed Euclidean coordinates. It follows that
\[
 \begin{aligned}
 |D_E\rho(r\theta)-\theta|
 &\le\norm g\,|Y'(d)-\theta|+\norm{g-e}\,|\theta|\\
 &\le2(10n^2\delta+\beta)+4\beta
   =H_n\delta\le\frac{1}{2}.
 \end{aligned}
\]

Set $u(y)=\rho(y)-|y|$. Since $D_E|y|=\theta$ at $y=r\theta$,
the preceding estimate says
\[
 D_E u=D_E\rho-\theta,\qquad |D_Eu|\le H_n\delta.
\]
The distance has unit $g$-gradient almost everywhere. In coordinates
this gives $g^{ij}\partial_i\rho\,\partial_j\rho=1$.
Also, $\norm q\le4B_n\delta\le \frac{1}{2}$ gives
$g^{-1}=e-q+O_n(\delta^2)$. Substituting
$\partial_i\rho=\theta_i+\partial_i u$ at the point $r\theta$ gives
\[
 \begin{aligned}
 1
 &=g^{ij}\partial_i\rho\,\partial_j\rho\\
 &=(\delta_{ij}-q_{ij})(\theta_i+\partial_i u)
                         (\theta_j+\partial_j u)+O_n(\delta^2)\\
 &=1+2\theta_i\partial_i u-q_{ij}\theta_i\theta_j
       +|D_Eu|^2-2q_{ij}\theta_i\partial_j u
       -q_{ij}\partial_i u\partial_j u+O_n(\delta^2).
 \end{aligned}
\]
From $\partial_r=\theta_i\partial_i$, we get
$2\theta_i\partial_i u=2\partial_r u$.
The bound $|D_Eu|\le H_n\delta$, together with $\norm q=O_n(\delta)$,
shows that the last three terms involving $D_Eu$ are
$O_n(\delta^2)$. Thus
\[
 2\partial_r u(r\theta)
 =q_{ij}(r\theta)\theta_i\theta_j+O_n(\delta^2).
\]
Integrating from $0$ to $r$, with $u(0)=0$, proves
\eqref{hd:eq:distance-expansion}. The distance is locally Lipschitz,
so this calculation applies first on almost every ray by absolute
continuity; continuity then extends the formula to every endpoint.

\smallskip\noindent
\emph{Step 2: solve for the Euclidean radius of the geodesic sphere.}
The gradient bound from Step 1 gives
$\partial_r\rho=1+\partial_r u\ge1-H_n\delta\ge \frac{1}{2}$
almost everywhere. Integration on almost every ray, followed by
continuity in the direction, shows that $\rho(r\theta)$ is strictly
increasing in $r$ on every ray in $B_{\frac{7L}{8}}^E$.

Fix $0<R\le \frac{L}{2}$. The points $(1\pm\beta)R\theta$ lie in this
Euclidean ball because $(1+\beta)R\le\frac{9L}{16}<\frac{7L}{8}$.
The distance bounds imply
\[
 \rho((1-\beta)R\theta)\le R,
 \qquad \rho((1+\beta)R\theta)\ge R.
\]
So there exists a unique Euclidean radius $r_R(\theta)$,
continuous in $\theta$, such that
\[
 \rho(r_R(\theta)\theta)=R,
 \qquad (1-\beta)R\le r_R(\theta)\le(1+\beta)R.
\]
Here $R$ is the prescribed intrinsic radius, while $r_R(\theta)$
is the coordinate radius of its boundary in direction $\theta$.
Since $y(B_p(R))\subset B_{(1+\beta)R}^E$, monotonicity gives
\[
 y(B_p(R))
 =\{r\theta:\theta\in\Sph^{n-1},\ 0\le r<r_R(\theta)\}.
\]

For a fixed direction, write $r_R=r_R(\theta)$. Substituting this
radius into \eqref{hd:eq:distance-expansion} gives the equation
\[
 R=\rho(r_R\theta)
   =r_R+\frac{1}{2}\int_0^{r_R}q_{ij}(t\theta)\theta_i\theta_j\dd t
        +O_n(\delta^2R).
\]
We may replace the upper integration limit $r_R$ by $R$, because
\[
 \left|\int_R^{r_R}q_{ij}(t\theta)\theta_i\theta_j\dd t\right|
 \le\norm q_\infty|r_R-R|
 \le4\beta^2R=O_n(\delta^2R).
\]
This yields
\begin{equation}\label{hd:eq:boundary-displacement}
 r_R(\theta)-R
 =-\frac{1}{2}\int_0^R q_{ij}(t\theta)\theta_i\theta_j\dd t
       +O_n(\delta^2R).
\end{equation}

\smallskip\noindent
\emph{Step 3: add the volume-density change and the boundary movement.}
The density expansion is
$\sqrt{\det g}=1+\frac{1}{2}\tr_Eq+O_n(\delta^2)$.
Integrating over the coordinate domain $y(B_p(R))$ gives
\[
 \begin{aligned}
 \Vol_gB_p(R)
 &=\frac{1}{n}\int_{\Sph^{n-1}}r_R(\theta)^n\dd\sigma(\theta)\\
 &\quad+\frac{1}{2}\int_{\Sph^{n-1}}\int_0^{r_R(\theta)}
              \tr_Eq(t\theta)t^{n-1}\dd t\dd\sigma(\theta)
       +O_n(\delta^2R^n).
 \end{aligned}
\]
The first term is the Euclidean volume of the coordinate domain.
Since $|r_R-R|\le\beta R=O_n(\delta R)$, expanding
$r_R^n=(R+(r_R-R))^n$ gives
\[
 \frac{1}{n}\int_{\Sph^{n-1}}r_R(\theta)^n\dd\sigma
 =\omega_nR^n+R^{n-1}\int_{\Sph^{n-1}}(r_R(\theta)-R)\dd\sigma
       +O_n(\delta^2R^n).
\]
Integrating \eqref{hd:eq:boundary-displacement} over $\Sph^{n-1}$
identifies the boundary contribution:
\[
 \begin{aligned}
 &\int_{\Sph^{n-1}}(r_R(\theta)-R)\dd\sigma\\
 &\quad=-\frac{1}{2}\int_0^R
       \left(\int_{\Sph^{n-1}}q_{ij}(t\theta)\theta_i\theta_j\dd\sigma\right)
       \dd t+O_n(\delta^2R)\\
 &\quad=-\frac{1}{2}\int_0^R A_q(t)\dd t+O_n(\delta^2R).
 \end{aligned}
\]

For the density contribution, the interval between $r_R$ and $R$
has length at most $\beta R$, and $|\tr_Eq|\le4n\beta$.
Changing the upper limit from $r_R(\theta)$ to $R$ therefore costs only
$O_n(\delta^2R^n)$. Fubini's theorem then gives
\[
 \begin{aligned}
 &\int_{\Sph^{n-1}}\int_0^{r_R(\theta)}
             \tr_Eq(t\theta)t^{n-1}\dd t\dd\sigma\\
 &\quad=\int_0^R
       \left(\int_{\Sph^{n-1}}\tr_Eq(t\theta)\dd\sigma\right)
       t^{n-1}\dd t+O_n(\delta^2R^n)\\
 &\quad=\int_0^RB_q(t)t^{n-1}\dd t+O_n(\delta^2R^n).
 \end{aligned}
\]
Combining the boundary and density contributions, we obtain
\[
 \Vol_gB_p(R)
 =\omega_nR^n-\frac{1}{2}R^{n-1}\int_0^R A_q(t)\dd t
       +\frac{1}{2}\int_0^R B_q(t)t^{n-1}\dd t
       +O_n(\delta^2R^n).
\]
Division by $\omega_nR^n$ proves the volume expansion.
Bishop--Gromov comparison gives
$v\le\frac{\Vol_gB_p(R)}{\omega_nR^n}\le1$, so
$D_q(R)\le\delta+C_n\delta^2$.
\end{proof}

For any smooth symmetric tensor $a$ on the coordinate region, put
\[
 J_a(t)=\int_{B_t^E}S^{\mathrm{lin}}(a)\dd y,\qquad
 \Theta_a(t)=t^{2-n}J_a(t).
\]
Thus $\Theta_a(t)$ is the normalized linear scalar-curvature integral. The next lemma identifies $2s_nD_a(R)$ as a weighted average
of $\Theta_a(t)$ over the radii $0<t<R$. Nonnegative linearized Ricci
curvature makes these normalized integrals nonnegative and
nondecreasing, allowing the volume deficit to control a smaller ball.

\begin{lemma}[Linearized volume comparison]\label{hd:lem:kernel}
Let $a$ be a smooth symmetric tensor on a neighborhood of
$\overline{B_R^E}$. For $0<t<R$ and $\theta=\frac{y}{|y|}$,
\begin{equation}\label{hd:eq:linear-monotonicity}
 \Theta_a'(t)=2t^{2-n}
       \int_{\partial B_t^E}\Ric^{\mathrm{lin}}(a)(\theta,\theta)\dd A_E.
\end{equation}
Moreover,
\begin{equation}\label{hd:eq:kernel}
 D_a(R)=\frac{1}{2s_n}\int_0^R w_R(t)\Theta_a(t)\dd t,
 \qquad
 w_R(t)=\frac{n}{(n-1)R}
                 \left[1-\left(\frac{t}{R}\right)^{n-1}\right].
\end{equation}
The kernel is nonnegative and has integral one, and
\begin{equation}\label{hd:eq:inner-mass}
 \int_0^{t_0}w_R(t)\dd t\le\frac{nt_0}{(n-1)R}
 \qquad(0<t_0\le R).
\end{equation}
If $\Ric^{\mathrm{lin}}(a)\ge0$ on $B_R^E$, then $\Theta_a$ is nonnegative
and nondecreasing, and for every $0<r<R$,
\begin{equation}\label{hd:eq:linear-volume-comparison}
 2s_nD_a(r)\le\Theta_a(r)
 =r^{2-n}\int_{B_r^E}S^{\mathrm{lin}}(a)\dd y
 \le\frac{2s_nD_a(R)}{\displaystyle\int_r^R w_R(t)\dd t}.
\end{equation}
\end{lemma}

\begin{proof}
\emph{Step 1: relate the curvature integral to $A_a$ and $B_a$.}
From \eqref{hd:eq:linear-S}, the formula for linearized scalar curvature is
$S^{\mathrm{lin}}(a)=\diver_E(\diver_Ea-d\tr_Ea)$. The divergence
theorem and $\dd A_E=t^{n-1}d\sigma$ give
\[
 J_a(t)=t^{n-1}\left[
   \int_{\Sph^{n-1}}(\diver_Ea)(t\theta)\cdot\theta\dd\sigma
   -B_a'(t)\right].
\]
To evaluate the remaining spherical integral, apply the divergence
theorem to the vector field with components $a_{ij}(y)y^j$ on $B_t^E$.
Its divergence is $\tr_Ea+(\diver_Ea)\cdot y$, so
\[
 \int_{B_t^E}\bigl(\tr_Ea+(\diver_Ea)\cdot y\bigr)\dd y=\int_{\partial B_t^E}a_{ij}(y)y^j\theta_i\dd A_E=t^nA_a(t)
 .
\]
Differentiating this identity in $t$ and dividing by $t^{n-1}$ yields
\[
 nA_a(t)+tA_a'(t)
 =B_a(t)+t\int_{\Sph^{n-1}}
                     (\diver_Ea)(t\theta)\cdot\theta\dd\sigma.
\]
Substitute into the formula for $J_a$ and multiply by $t^{2-n}$:
\begin{equation}\label{hd:eq:spherical-identities}
 \Theta_a(t)=t\bigl(A_a'(t)-B_a'(t)\bigr)+nA_a(t)-B_a(t).
\end{equation}
This expresses the curvature integral in terms of the same two
spherical integrals that define the volume deficit.

\smallskip\noindent
\emph{Step 2: integrate twice to recover the volume deficit.}
Denote the numerator in \eqref{hd:eq:volume-functional} by
\[
 N(t)=2\omega_nt^nD_a(t)
     =t^{n-1}\int_0^t A_a(s)\dd s
                  -\int_0^t B_a(s)s^{n-1}\dd s.
\]
Direct differentiation gives
\[
 t^{2-n}N'(t)
 =(n-1)\int_0^t A_a(s)\dd s+t\bigl(A_a(t)-B_a(t)\bigr).
\]
Differentiating once more and using
\eqref{hd:eq:spherical-identities}, we obtain the key identity
\[
 \frac{\dd}{\dd t}\bigl(t^{2-n}N'(t)\bigr)=\Theta_a(t).
\]
From their expressions and smoothness of $a$, both $N(t)$ and $t^{2-n}N'(t)$ tend to zero as $t\downarrow0$.
Thus integration twice, followed by a change of integration order, gives
\[
 \begin{aligned}
 N(R)
 &=\int_0^R s^{n-2}\int_0^s\Theta_a(t)\dd t\dd s\\
 &=\int_0^R\left(\int_t^R s^{n-2}\dd s\right)\Theta_a(t)\dd t\\
 &=\frac{1}{n-1}\int_0^R
             (R^{n-1}-t^{n-1})\Theta_a(t)\dd t.
 \end{aligned}
\]

Dividing the expression for $N(R)$ by $2\omega_nR^n$ and using
$s_n=n\omega_n$, we obtain
\[
\begin{aligned}
 D_a(R)&=\int_0^R \frac{R^{n-1}-t^{n-1}}{2(n-1)\omega_nR^n}\Theta_a(t)\dd t\\
 &=\frac{1}{2s_n}\int_0^R\frac{n}{(n-1)R}\left[1-(\frac{t}{R})^{n-1}\right]\Theta_a(t)\dd t.
\end{aligned}
\]
This is \eqref{hd:eq:kernel} with the stated kernel $w_R$.

The weight $w_R$ satisfies
\[
 w_R(t)\ge0,\qquad
 \int_0^R w_R(t)\dd t
 =\frac{n}{n-1}\left(1-\frac{1}{n}\right)=1.
\]
The bound $w_R(t)\le \frac{n}{(n-1)R}$ also proves
\eqref{hd:eq:inner-mass}.

\smallskip\noindent
\emph{Step 3: use the Bianchi identity to compare two radii.}
Differentiating the contracted Bianchi identity at the flat metric
gives
\[
 \partial_i\Ric^{\mathrm{lin}}(a)_{ij}
       =\frac{1}{2}\partial_jS^{\mathrm{lin}}(a).
\]
Consequently,
\[
 \partial_i\left[
   \left(\frac{1}{2}S^{\mathrm{lin}}(a)\delta_{ij}
                       -\Ric^{\mathrm{lin}}(a)_{ij}\right)y^j\right]
 =\frac{n}{2}S^{\mathrm{lin}}(a)-\text{tr}_E\Ric^{\mathrm{lin}}(a)=\frac{n-2}{2}S^{\mathrm{lin}}(a).
\]
Integrate over $B_t^E$. On its boundary $y=t\theta$, and
$J_a'(t)=\int_{\partial B_t^E}S^{\mathrm{lin}}(a)\dd A_E$, so
\[
 \frac{n-2}{2}J_a(t)
 =\frac{t}{2}J_a'(t)
       -t\int_{\partial B_t^E}
                      \Ric^{\mathrm{lin}}(a)(\theta,\theta)\dd A_E.
\]
Therefore
\[
 \begin{aligned}
 \Theta_a'(t)
 &=t^{1-n}\bigl(tJ_a'(t)-(n-2)J_a(t)\bigr)\\
 &=2t^{2-n}\int_{\partial B_t^E}
                       \Ric^{\mathrm{lin}}(a)(\theta,\theta)\dd A_E,
 \end{aligned}
\]
which proves \eqref{hd:eq:linear-monotonicity}.

If $\Ric^{\mathrm{lin}}(a)\ge0$, this derivative is nonnegative.
Equation \eqref{hd:eq:spherical-identities} and smoothness of $a$ give $\lim_{t\downarrow0}\Theta_a(t)=0$, so $\Theta_a$ is
also nonnegative. Equation \eqref{hd:eq:kernel} and the inequality $\Theta_a(t)\ge\Theta_a(r)$ for $t\ge r$ give
\[
 2s_nD_a(R)
 =\int_0^R w_R(t)\Theta_a(t)\dd t
 \ge\Theta_a(r)\int_r^R w_R(t)\dd t.
\]
This gives the upper
inequality in \eqref{hd:eq:linear-volume-comparison}. For the lower
inequality, apply the same identity at radius $r$ and use
$\Theta_a(t)\le\Theta_a(r)$ for $t\le r$:
\[
 2s_nD_a(r)=\int_0^rw_r(t)\Theta_a(t)\dd t
 \le\Theta_a(r)\int_0^rw_r(t)\dd t=\Theta_a(r).
\]
\end{proof}

\subsection{Completion of the estimates}

\begin{proof}[Proof of the upper bound in Theorem~\ref{thm:near-maximal}]
It suffices first to prove the estimate at radius one. The
preceding lemmas give three ingredients: a tensor $\widehat q$
with nonnegative linearized Ricci curvature, a bound for its
linearized volume deficit, and a comparison between that deficit
and its scalar-curvature integral. We now combine them.

\smallskip\noindent
\emph{Step 1: choose the comparison scale and bound the volume
functional.}
Choose $t_0=1+\beta$ and put
\[
 R=\frac{t_0}{\delta}.
\]
Apply Lemmas~\ref{hd:lem:coords-smooth}--\ref{hd:lem:majorant}
with $L=2R$, and construct $\widehat q$ in these coordinates.
Then $\overline{B_R^E}\subset B_{\frac{7L}{8}}^E\subset y(B_p(L))$.
For $x\in B_p(1)$, the coordinate distance estimate gives
$|y(x)|\le(1+\beta)d(p,x)<t_0$, hence
$y(B_p(1))\subset B_{t_0}^E$.
Since $R=\frac{L}{2}$, Lemma~\ref{hd:lem:volume-expansion} applies at
the comparison radius $R$. All constants are independent of $R$.

The functional $D_a(R)$ is linear in $a$. Its definition \eqref{hd:eq:volume-functional} and the
sphere area $s_n=n\omega_n$ give the uniform bound
\[
 |D_a(R)|\le n\norm a_\infty.
\]
Hence Proposition~\ref{hd:prop:correction} implies
$|D_{\widehat q}(R)-D_q(R)|\le C_n\delta^2$.
Combining this with Lemma~\ref{hd:lem:volume-expansion}, we get
\begin{equation}\label{hd:eq:corrected-D}
 D_{\widehat q}(R)\le\delta+C_n\delta^2.
\end{equation}

\smallskip\noindent
\emph{Step 2: control the linearized curvature on the smaller ball.}
By Proposition~\ref{hd:prop:correction},
$\Ric^{\mathrm{lin}}(\widehat q)\ge0$. We may therefore apply
Lemma~\ref{hd:lem:kernel}. The choice of $R$ ensures that the
weight below $t_0$ is small:
\[
 \int_{t_0}^Rw_R(t)\dd t
 =1-\int_0^{t_0}w_R(t)\dd t
 \ge1-\frac{n}{n-1}\frac{t_0}{R}
 =1-\frac{n}{n-1}\delta.
\]
Now \eqref{hd:eq:linear-volume-comparison} yields
\[
 \begin{aligned}
 J_{\widehat q}(t_0)
 &\le\frac{2s_nt_0^{n-2}D_{\widehat q}(R)}{1-\frac{n}{n-1}\delta}\\
 &\le2s_n\delta+C_n\delta^2,
 \end{aligned}
\]
because $t_0^{n-2}=1+O_n(\delta)$ and
\eqref{hd:eq:corrected-D} bounds the numerator.

\smallskip\noindent
\emph{Step 3: return to the scalar curvature of $g$.}
The density comparison \eqref{hd:eq:correction-bounds} and
$S^{\mathrm{lin}}(\widehat q)\ge0$ allow us to integrate over $y(B_p(1))$
and then enlarge the domain to $B_{t_0}^E$:
\[
 \begin{aligned}
 I_p(1)
 &=\int_{y(B_p(1))}S_g\sqrt{\det g}\dd y\\
 &\le(1+C_n\delta)\int_{B_{t_0}^E}S^{\mathrm{lin}}(\widehat q)\dd y\\
 &\le2n\omega_n\delta+C_n\delta^2.
 \end{aligned}
\]
For a general radius $r>0$, replace $g$ by $r^{-2}g$.
The asymptotic volume ratio is unchanged, while the unit-ball
scalar integral for the rescaled metric equals $r^{2-n}I_p(r)$.
Thus
\[
 I_p(r)\le\bigl(2n\omega_n\delta+C_n\delta^2\bigr)r^{n-2}.
\]
When $v=1$, Bishop--Gromov rigidity makes the metric Euclidean
and the integral is zero. Section~\ref{sec:rotational} proves
optimality of the first-order coefficient and necessity of
a quadratic remainder.
\end{proof}

\begin{proof}[Proof of the lower bound in Theorem~\ref{thm:near-maximal}]
We again work at radius one. The case $v=1$ follows from
Bishop--Gromov rigidity, so assume $0<\delta\le\delta_n$ and put
\[
 \eta=1-v_p(1),\qquad 0\le\eta\le\delta.
\]
Apply
Lemmas~\ref{hd:lem:coords-smooth}--\ref{hd:lem:majorant} with $L=4$
and form $\widehat q$ in the resulting chart. In particular,
$\overline{B_2^E}$ lies inside the coordinate region.

\smallskip\noindent
\emph{Step 1: bound the integral of the correction.}
The identity from Proposition~\ref{hd:prop:correction} gives
\[
 \Ric^{\mathrm{lin}}(d\varphi\otimes d\varphi)=A_P\ge0,\qquad
 S^{\mathrm{lin}}(d\varphi\otimes d\varphi)=\tr_EA_P.
\]
Also $\norm{D_E\varphi}\le10\delta$, hence
$\norm{d\varphi\otimes d\varphi}_\infty\le100\delta^2$.
The definition of $D_a$ yields $|D_a(R)|\le n\norm a_\infty$.
Apply Lemma~\ref{hd:lem:kernel} to $a=d\varphi\otimes d\varphi$, with
inner radius $1$ and outer radius $2$. By
\eqref{hd:eq:inner-mass},
\[
 \int_1^2w_2(t)\dd t
 \ge1-\frac{n}{2(n-1)}\ge\frac{1}{4}.
\]
Consequently,
\begin{equation}\label{hd:eq:correction-integral}
 \begin{aligned}
 \int_{B_1^E}\tr_EA_P\dd y
 &=J_{d\varphi\otimes d\varphi}(1)\\
 &\le8s_nD_{d\varphi\otimes d\varphi}(2)
 \le800ns_n\delta^2.
 \end{aligned}
\end{equation}
Thus the correction has a quadratic integral error.

\smallskip\noindent
\emph{Step 2: compare the volume functional at a slightly smaller radius.}
Set $r_-=1-\beta$. The coordinate distance bounds imply
\[
 B_{r_-}^E\subset y(B_p(1)).
\]
Lemma~\ref{hd:lem:volume-expansion} at radius one, linearity of $D_a$,
and $\norm{\widehat q-q}_\infty\le C_n\delta^2$ give
\[
 D_{\widehat q}(1)=\eta+O_n(\delta^2).
\]
To change the radius, differentiate the definition
\eqref{hd:eq:volume-functional}:
\[
 \begin{aligned}
 2\omega_n R D_a'(R)
 &=A_a(R)-B_a(R)-\frac{1}{R}\int_0^RA_a(t)\dd t\\
 &\quad+\frac{n}{R^n}\int_0^RB_a(t)t^{n-1}\dd t.
 \end{aligned}
\]
Since $|A_a|\le s_n\norm a_\infty$ and
$|B_a|\le ns_n\norm a_\infty$, this gives
\[
 |R D_a'(R)|\le n(n+1)\norm a_\infty.
\]
Using $\norm{\widehat q}_\infty\le C_n\delta$ and $\beta=B_n\delta\le \frac{1}{8}$,
we conclude that
\[
 |D_{\widehat q}(r_-)-D_{\widehat q}(1)|
 \le C_n\delta\int_{r_-}^1\frac{\dd R}{R}
 \le C_n\delta\beta\le C_n\delta^2.
\]
Therefore
\begin{equation}\label{hd:eq:local-corrected-D}
 D_{\widehat q}(r_-)=\eta+O_n(\delta^2).
\end{equation}

\smallskip\noindent
\emph{Step 3: recover the scalar integral on the geodesic ball.}
The lower inequality in \eqref{hd:eq:linear-volume-comparison}
and \eqref{hd:eq:local-corrected-D} imply
\[
 \begin{aligned}
 J_{\widehat q}(r_-)
 &\ge2s_nr_-^{n-2}D_{\widehat q}(r_-)\\
 &\ge2s_n\eta-C_n\delta^2.
 \end{aligned}
\]
Here $r_-^{n-2}=1+O_n(\delta)$ and $\eta\le\delta$, so the change
of the scaling factor costs only $O_n(\delta^2)$.
On the other hand, \eqref{hd:eq:reverse-density},
\eqref{hd:eq:correction-integral}, and $S_g\ge0$ give
\[
 \begin{aligned}
 J_{\widehat q}(r_-)
 &\le(1+C_n\delta)
     \int_{B_{r_-}^E}S_g\sqrt{\det g}\dd y
       +16n^6\int_{B_{r_-}^E}\tr_EA_P\dd y\\
 &\le(1+C_n\delta)I_p(1)+C_n\delta^2.
 \end{aligned}
\]
Combining these estimates and again using $\eta\le\delta$, we obtain
\[
 I_p(1)\ge\frac{2s_n\eta-C_n\delta^2}{1+C_n\delta}
          \ge2n\omega_n\eta-C_n\delta^2.
\]
Rescaling $g$ by $r^{-2}$ proves the lower estimate in
\eqref{eq:near-maximal}. Take the larger of the upper and lower
dimensional constants to obtain a common $C_n$ in that theorem.
\end{proof}

\section{Rotationally symmetric manifolds}
\label{sec:rotational}

We now prove both rotational estimates for $n\ge3$ and their sharpness
at each prescribed asymptotic or local volume ratio. We then establish optimality of both first-order coefficients in
Theorem~\ref{thm:near-maximal} and the necessity of a quadratic
remainder in its upper bound.

Consider a smooth rotational metric on $\R^n$,
\begin{equation}\label{eq:rot-metric}
 g=\mathrm dt^2+f(t)^2g_{\Sph^{n-1}},\qquad
 f(0)=0,\quad f'(0)=1,\quad f(t)>0\quad(t>0).
\end{equation}
We assume the usual smoothness condition at the origin: $f$ has
a smooth odd extension near zero. The origin is $p$ and
$t=d(p,\cdot)$.

\begin{lemma}\label{lem:rot-concavity}
If the metric \eqref{eq:rot-metric} has $\Ric\ge0$, then
\begin{equation}\label{eq:rot-concavity}
 f''\le0,\qquad 0\le f'\le1,\qquad f(t)\le t.
\end{equation}
Moreover,
\begin{equation}\label{eq:rot-avr}
 a:=\lim_{t\to\infty}f'(t)
   =\lim_{t\to\infty}\frac{f(t)}{t},
 \qquad V_{M^n}=a^{n-1}.
\end{equation}
The scalar curvature $S$ is a function of $t$ and
\begin{equation}\label{eq:rot-scalar}
	S(t)
	=-2(n-1)\frac{f''(t)}{f(t)}
	+(n-1)(n-2)\frac{1-f'(t)^2}{f(t)^2},
	\qquad t>0.
\end{equation}
\end{lemma}

\begin{proof}
Since
\[
 \Ric(\partial_t,\partial_t)=-(n-1)\frac{f''}{f},
\]
we have $f''\le0$. A negative value of $f'$ would force $f$ to
become negative at a finite distance, so $f'\ge0$. Also, $f'$ is nonincreasing
and $f'(0)=1$, so $f'\le1$. Hence $0\le f'\le1$ and
$f(t)\le t$. The nonincreasing function $f'$ has a limit
$a\in[0,1]$, and averaging $f'$ gives $\frac{f(t)}{t}\to a$.
Using $|\Sph^{n-1}|=n\omega_n$, dominated convergence yields
\[
 V_{M^n}=\lim_{r\to\infty}
 n\int_0^1\left(\frac{f(rs)}{r}\right)^{n-1}\dd s
 =n\int_0^1(as)^{n-1}\dd s=a^{n-1}.
\]
The radial sectional curvatures are $-\frac{f''}{f}$, and the tangential
sectional curvatures are $\frac{1-f'^2}{f^2}$.
From these, we get \eqref{eq:rot-scalar}.
\end{proof}

By scaling, we only need to work with $r=1$ in the proof of Theorem~\ref{thm:rotational}.

\begin{proof}[Proof of the upper bound in Theorem~\ref{thm:rotational}]
Put $u=f'$ and $m=-u'\ge0$.
The scalar curvature formula \eqref{eq:rot-scalar} gives
\begin{equation}\label{eq:rot-integral}
 \frac{I_p(1)}{n(n-1)\omega_n}
 =2\int_0^1 f^{n-2}m\dd t+(n-2)\int_0^1 f^{n-3}(1-u^2)\dd t.
\end{equation}
Since $1-u(t)=\int_0^t m(s)\dd s$, Fubini's theorem rewrites
the right-hand side as
\[
 \int_0^1 m(s)\left[
  2f(s)^{n-2}+(n-2)\int_s^1 f(t)^{n-3}(1+u(t))\dd t
                  \right]\dd s.
\]
By Lemma~\ref{lem:rot-concavity}, the expression in brackets
is at most
\[
 2s^{n-2}+2(n-2)\int_s^1 t^{n-3}\dd t=2.
\]
Since $u(1)\ge a=V_{M^n}^{\frac{1}{n-1}}$, it follows that
\[
 I_p(1)\le2n(n-1)\omega_n\int_0^1m(s)\dd s
       \le2n(n-1)\omega_n(1-a).
\]

To prove sharpness, fix $a\in[0,1)$. Choose a smooth
nondecreasing function $\chi:\R\to[0,1]$ with $\chi=0$ on
$(-\infty,0]$ and $\chi=1$ on $[1,\infty)$. For
$0<\varepsilon<\frac{1}{4}$, set
\begin{equation}\label{eq:rot-example}
 \begin{split}
  u_\varepsilon(t)
   &=1-(1-a)\chi\left(\frac{t-1+2\varepsilon}{\varepsilon}\right),\\
  f_\varepsilon(t)&=\int_0^t u_\varepsilon(s)\dd s,\qquad
  g_\varepsilon=\mathrm dt^2+
                         f_\varepsilon(t)^2g_{\Sph^{n-1}}.
 \end{split}
\end{equation}
The metric is Euclidean on $B_p(1-2\varepsilon)$. Outside
$B_p(1)$, its end is conical when $a>0$ and cylindrical when $a=0$.
It is complete, has a pole $p$, and has nonnegative sectional
curvature and asymptotic volume ratio exactly $a^{n-1}$.
The Euclidean ball $t\le1-2\varepsilon$ is joined to the end
through the thin annulus $1-2\varepsilon<t<1-\varepsilon$,
entirely inside $B_p(1)$. Figures~\ref{fig:rot-zero} and~\ref{fig:rot-quarter}
show the construction in dimension three.

The function $m_\varepsilon=-u_\varepsilon'$ is supported in
$[1-2\varepsilon,1-\varepsilon]$, with integral $1-a$.
On this interval, $f_\varepsilon\to1$ uniformly, so the first
term in \eqref{eq:rot-integral} tends to $2(1-a)$.
The second term lies between zero and $2(n-2)\varepsilon$.
Consequently,
\[
 \lim_{\varepsilon\downarrow0}I_p^{g_\varepsilon}(1)
 =2n(n-1)\omega_n(1-a).
\]
For $a=1$, the metric is Euclidean and equality holds.
\end{proof}

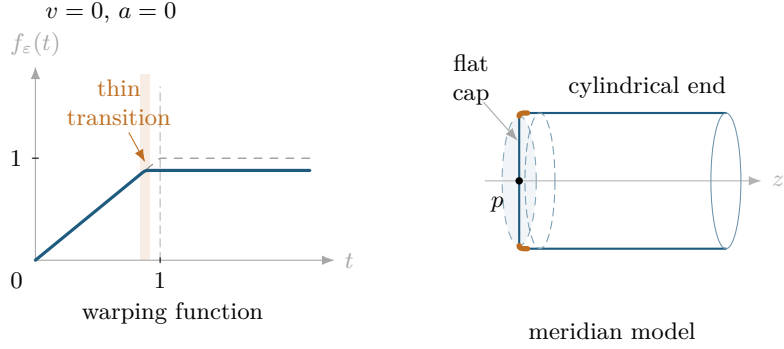
\begin{figure}[htbp]
\centering
\begin{tikzpicture}[font=\footnotesize,>=Latex,line cap=round,line join=round]
\begin{scope}[x=1.65cm,y=1.35cm]
\fill[geoamber!12] (.84,0) rectangle (.92,1.82);
\draw[->,gray!70] (0,0)--(2.4,0) node[right] {$t$};
\draw[->,gray!70] (0,0)--(0,1.9) node[above] {$f_\varepsilon(t)$};
\draw[gray,dashed] (0,0)--(1,1)--(2.2,1);
\draw[geoblue,very thick] plot coordinates {(0.0000,0.0000) (0.0200,0.0200) (0.0400,0.0400) (0.0600,0.0600) (0.0800,0.0800) (0.1000,0.1000) (0.1200,0.1200) (0.1400,0.1400) (0.1600,0.1600) (0.1800,0.1800) (0.2000,0.2000) (0.2200,0.2200) (0.2400,0.2400) (0.2600,0.2600) (0.2800,0.2800) (0.3000,0.3000) (0.3200,0.3200) (0.3400,0.3400) (0.3600,0.3600) (0.3800,0.3800) (0.4000,0.4000) (0.4200,0.4200) (0.4400,0.4400) (0.4600,0.4600) (0.4800,0.4800) (0.5000,0.5000) (0.5200,0.5200) (0.5400,0.5400) (0.5600,0.5600) (0.5800,0.5800) (0.6000,0.6000) (0.6200,0.6200) (0.6400,0.6400) (0.6600,0.6600) (0.6800,0.6800) (0.7000,0.7000) (0.7200,0.7200) (0.7400,0.7400) (0.7600,0.7600) (0.7800,0.7800) (0.8000,0.8000) (0.8200,0.8200) (0.8250,0.8250) (0.8275,0.8275) (0.8300,0.8300) (0.8325,0.8325) (0.8350,0.8350) (0.8375,0.8375) (0.8400,0.8400) (0.8425,0.8425) (0.8450,0.8450) (0.8475,0.8475) (0.8500,0.8500) (0.8525,0.8525) (0.8550,0.8550) (0.8575,0.8574) (0.8600,0.8598) (0.8625,0.8621) (0.8650,0.8643) (0.8675,0.8663) (0.8700,0.8683) (0.8725,0.8700) (0.8750,0.8717) (0.8775,0.8732) (0.8800,0.8745) (0.8825,0.8757) (0.8850,0.8767) (0.8875,0.8775) (0.8900,0.8783) (0.8925,0.8788) (0.8950,0.8793) (0.8975,0.8796) (0.9000,0.8798) (0.9025,0.8799) (0.9050,0.8800) (0.9075,0.8800) (0.9100,0.8800) (0.9125,0.8800) (0.9150,0.8800) (0.9175,0.8800) (0.9200,0.8800) (0.9225,0.8800) (0.9250,0.8800) (0.9275,0.8800) (0.9300,0.8800) (0.9325,0.8800) (0.9350,0.8800) (0.9375,0.8800) (0.9400,0.8800) (0.9425,0.8800) (0.9450,0.8800) (0.9475,0.8800) (0.9600,0.8800) (0.9800,0.8800) (1.0000,0.8800) (1.0200,0.8800) (1.0400,0.8800) (1.0600,0.8800) (1.0800,0.8800) (1.1000,0.8800) (1.1200,0.8800) (1.1400,0.8800) (1.1600,0.8800) (1.1800,0.8800) (1.2000,0.8800) (1.2200,0.8800) (1.2400,0.8800) (1.2600,0.8800) (1.2800,0.8800) (1.3000,0.8800) (1.3200,0.8800) (1.3400,0.8800) (1.3600,0.8800) (1.3800,0.8800) (1.4000,0.8800) (1.4200,0.8800) (1.4400,0.8800) (1.4600,0.8800) (1.4800,0.8800) (1.5000,0.8800) (1.5200,0.8800) (1.5400,0.8800) (1.5600,0.8800) (1.5800,0.8800) (1.6000,0.8800) (1.6200,0.8800) (1.6400,0.8800) (1.6600,0.8800) (1.6800,0.8800) (1.7000,0.8800) (1.7200,0.8800) (1.7400,0.8800) (1.7600,0.8800) (1.7800,0.8800) (1.8000,0.8800) (1.8200,0.8800) (1.8400,0.8800) (1.8600,0.8800) (1.8800,0.8800) (1.9000,0.8800) (1.9200,0.8800) (1.9400,0.8800) (1.9600,0.8800) (1.9800,0.8800) (2.0000,0.8800) (2.0200,0.8800) (2.0400,0.8800) (2.0600,0.8800) (2.0800,0.8800) (2.1000,0.8800) (2.1200,0.8800) (2.1400,0.8800) (2.1600,0.8800) (2.1800,0.8800) (2.2000,0.8800)};
\draw[gray!60,densely dashed] (1,0)--(1,1.7);
\draw (1,.025)--(1,-.025) node[below] {$1$};
\draw (.025,1)--(-.025,1) node[left] {$1$};
\node[anchor=north east] at (-.02,-.02) {$0$};
\node[geoamber,align=center] at (.67,1.55) {thin\\transition};
\draw[geoamber,->] (.75,1.25)--(.88,.96);
\node at (1.1,-.52) {warping function};
\end{scope}
\begin{scope}[shift={(6.4cm,1.05cm)},x=2.05cm,y=1.02cm]
\fill[geoblue!6] (0,0) ellipse[x radius=.11,y radius=.84];
\draw[geoblue!55,dashed] (0,0) ellipse[x radius=.11,y radius=.84];
\draw[gray!60,->] (-.22,0)--(1.5640,0) node[right] {$z$};
\draw[geoblue,thick] plot coordinates {(0.0000,0.0000) (0.0000,0.0200) (0.0000,0.0400) (0.0000,0.0600) (0.0000,0.0800) (0.0000,0.1000) (0.0000,0.1200) (0.0000,0.1400) (0.0000,0.1600) (0.0000,0.1800) (0.0000,0.2000) (0.0000,0.2200) (0.0000,0.2400) (0.0000,0.2600) (0.0000,0.2800) (0.0000,0.3000) (0.0000,0.3200) (0.0000,0.3400) (0.0000,0.3600) (0.0000,0.3800) (0.0000,0.4000) (0.0000,0.4200) (0.0000,0.4400) (0.0000,0.4600) (0.0000,0.4800) (0.0000,0.5000) (0.0000,0.5200) (0.0000,0.5400) (0.0000,0.5600) (0.0000,0.5800) (0.0000,0.6000) (0.0000,0.6200) (0.0000,0.6400) (0.0000,0.6600) (0.0000,0.6800) (0.0000,0.7000) (0.0000,0.7200) (0.0000,0.7400) (0.0000,0.7600) (0.0000,0.7800) (0.0000,0.8000) (0.0000,0.8200) (0.0000,0.8250) (0.0000,0.8275) (0.0000,0.8300) (0.0000,0.8325) (0.0000,0.8350) (0.0000,0.8375) (0.0000,0.8400) (0.0000,0.8425) (0.0000,0.8450) (0.0000,0.8475) (0.0001,0.8500) (0.0003,0.8525) (0.0006,0.8550) (0.0012,0.8574) (0.0019,0.8598) (0.0029,0.8621) (0.0041,0.8643) (0.0056,0.8663) (0.0071,0.8683) (0.0089,0.8700) (0.0108,0.8717) (0.0128,0.8732) (0.0149,0.8745) (0.0171,0.8757) (0.0194,0.8767) (0.0217,0.8775) (0.0241,0.8783) (0.0266,0.8788) (0.0290,0.8793) (0.0315,0.8796) (0.0340,0.8798) (0.0365,0.8799) (0.0390,0.8800) (0.0415,0.8800) (0.0440,0.8800) (0.0465,0.8800) (0.0490,0.8800) (0.0515,0.8800) (0.0540,0.8800) (0.0565,0.8800) (0.0590,0.8800) (0.0615,0.8800) (0.0640,0.8800) (0.0665,0.8800) (0.0690,0.8800) (0.0715,0.8800) (0.0740,0.8800) (0.0765,0.8800) (0.0790,0.8800) (0.0815,0.8800) (0.0940,0.8800) (0.1140,0.8800) (0.1340,0.8800) (0.1540,0.8800) (0.1740,0.8800) (0.1940,0.8800) (0.2140,0.8800) (0.2340,0.8800) (0.2540,0.8800) (0.2740,0.8800) (0.2940,0.8800) (0.3140,0.8800) (0.3340,0.8800) (0.3540,0.8800) (0.3740,0.8800) (0.3940,0.8800) (0.4140,0.8800) (0.4340,0.8800) (0.4540,0.8800) (0.4740,0.8800) (0.4940,0.8800) (0.5140,0.8800) (0.5340,0.8800) (0.5540,0.8800) (0.5740,0.8800) (0.5940,0.8800) (0.6140,0.8800) (0.6340,0.8800) (0.6540,0.8800) (0.6740,0.8800) (0.6940,0.8800) (0.7140,0.8800) (0.7340,0.8800) (0.7540,0.8800) (0.7740,0.8800) (0.7940,0.8800) (0.8140,0.8800) (0.8340,0.8800) (0.8540,0.8800) (0.8740,0.8800) (0.8940,0.8800) (0.9140,0.8800) (0.9340,0.8800) (0.9540,0.8800) (0.9740,0.8800) (0.9940,0.8800) (1.0140,0.8800) (1.0340,0.8800) (1.0540,0.8800) (1.0740,0.8800) (1.0940,0.8800) (1.1140,0.8800) (1.1340,0.8800) (1.1540,0.8800) (1.1740,0.8800) (1.1940,0.8800) (1.2140,0.8800) (1.2340,0.8800) (1.2540,0.8800) (1.2740,0.8800) (1.2940,0.8800) (1.3140,0.8800) (1.3340,0.8800)};
\draw[geoblue,thick] plot coordinates {(0.0000,-0.0000) (0.0000,-0.0200) (0.0000,-0.0400) (0.0000,-0.0600) (0.0000,-0.0800) (0.0000,-0.1000) (0.0000,-0.1200) (0.0000,-0.1400) (0.0000,-0.1600) (0.0000,-0.1800) (0.0000,-0.2000) (0.0000,-0.2200) (0.0000,-0.2400) (0.0000,-0.2600) (0.0000,-0.2800) (0.0000,-0.3000) (0.0000,-0.3200) (0.0000,-0.3400) (0.0000,-0.3600) (0.0000,-0.3800) (0.0000,-0.4000) (0.0000,-0.4200) (0.0000,-0.4400) (0.0000,-0.4600) (0.0000,-0.4800) (0.0000,-0.5000) (0.0000,-0.5200) (0.0000,-0.5400) (0.0000,-0.5600) (0.0000,-0.5800) (0.0000,-0.6000) (0.0000,-0.6200) (0.0000,-0.6400) (0.0000,-0.6600) (0.0000,-0.6800) (0.0000,-0.7000) (0.0000,-0.7200) (0.0000,-0.7400) (0.0000,-0.7600) (0.0000,-0.7800) (0.0000,-0.8000) (0.0000,-0.8200) (0.0000,-0.8250) (0.0000,-0.8275) (0.0000,-0.8300) (0.0000,-0.8325) (0.0000,-0.8350) (0.0000,-0.8375) (0.0000,-0.8400) (0.0000,-0.8425) (0.0000,-0.8450) (0.0000,-0.8475) (0.0001,-0.8500) (0.0003,-0.8525) (0.0006,-0.8550) (0.0012,-0.8574) (0.0019,-0.8598) (0.0029,-0.8621) (0.0041,-0.8643) (0.0056,-0.8663) (0.0071,-0.8683) (0.0089,-0.8700) (0.0108,-0.8717) (0.0128,-0.8732) (0.0149,-0.8745) (0.0171,-0.8757) (0.0194,-0.8767) (0.0217,-0.8775) (0.0241,-0.8783) (0.0266,-0.8788) (0.0290,-0.8793) (0.0315,-0.8796) (0.0340,-0.8798) (0.0365,-0.8799) (0.0390,-0.8800) (0.0415,-0.8800) (0.0440,-0.8800) (0.0465,-0.8800) (0.0490,-0.8800) (0.0515,-0.8800) (0.0540,-0.8800) (0.0565,-0.8800) (0.0590,-0.8800) (0.0615,-0.8800) (0.0640,-0.8800) (0.0665,-0.8800) (0.0690,-0.8800) (0.0715,-0.8800) (0.0740,-0.8800) (0.0765,-0.8800) (0.0790,-0.8800) (0.0815,-0.8800) (0.0940,-0.8800) (0.1140,-0.8800) (0.1340,-0.8800) (0.1540,-0.8800) (0.1740,-0.8800) (0.1940,-0.8800) (0.2140,-0.8800) (0.2340,-0.8800) (0.2540,-0.8800) (0.2740,-0.8800) (0.2940,-0.8800) (0.3140,-0.8800) (0.3340,-0.8800) (0.3540,-0.8800) (0.3740,-0.8800) (0.3940,-0.8800) (0.4140,-0.8800) (0.4340,-0.8800) (0.4540,-0.8800) (0.4740,-0.8800) (0.4940,-0.8800) (0.5140,-0.8800) (0.5340,-0.8800) (0.5540,-0.8800) (0.5740,-0.8800) (0.5940,-0.8800) (0.6140,-0.8800) (0.6340,-0.8800) (0.6540,-0.8800) (0.6740,-0.8800) (0.6940,-0.8800) (0.7140,-0.8800) (0.7340,-0.8800) (0.7540,-0.8800) (0.7740,-0.8800) (0.7940,-0.8800) (0.8140,-0.8800) (0.8340,-0.8800) (0.8540,-0.8800) (0.8740,-0.8800) (0.8940,-0.8800) (0.9140,-0.8800) (0.9340,-0.8800) (0.9540,-0.8800) (0.9740,-0.8800) (0.9940,-0.8800) (1.0140,-0.8800) (1.0340,-0.8800) (1.0540,-0.8800) (1.0740,-0.8800) (1.0940,-0.8800) (1.1140,-0.8800) (1.1340,-0.8800) (1.1540,-0.8800) (1.1740,-0.8800) (1.1940,-0.8800) (1.2140,-0.8800) (1.2340,-0.8800) (1.2540,-0.8800) (1.2740,-0.8800) (1.2940,-0.8800) (1.3140,-0.8800) (1.3340,-0.8800)};
\draw[geoamber,line width=2pt] plot coordinates {(0.0000,0.8400) (0.0000,0.8420) (0.0000,0.8440) (0.0000,0.8460) (0.0000,0.8480) (0.0001,0.8500) (0.0002,0.8520) (0.0005,0.8540) (0.0008,0.8559) (0.0013,0.8579) (0.0019,0.8598) (0.0027,0.8616) (0.0036,0.8634) (0.0047,0.8651) (0.0059,0.8667) (0.0071,0.8683) (0.0085,0.8697) (0.0100,0.8710) (0.0116,0.8723) (0.0132,0.8734) (0.0149,0.8745) (0.0167,0.8754) (0.0185,0.8763) (0.0203,0.8770) (0.0222,0.8777) (0.0241,0.8783) (0.0261,0.8787) (0.0280,0.8791) (0.0300,0.8794) (0.0320,0.8796) (0.0340,0.8798) (0.0360,0.8799) (0.0380,0.8799) (0.0400,0.8800) (0.0420,0.8800) (0.0440,0.8800) (0.0460,0.8800) (0.0480,0.8800) (0.0500,0.8800) (0.0520,0.8800) (0.0540,0.8800)};
\draw[geoamber,line width=2pt] plot coordinates {(0.0000,-0.8400) (0.0000,-0.8420) (0.0000,-0.8440) (0.0000,-0.8460) (0.0000,-0.8480) (0.0001,-0.8500) (0.0002,-0.8520) (0.0005,-0.8540) (0.0008,-0.8559) (0.0013,-0.8579) (0.0019,-0.8598) (0.0027,-0.8616) (0.0036,-0.8634) (0.0047,-0.8651) (0.0059,-0.8667) (0.0071,-0.8683) (0.0085,-0.8697) (0.0100,-0.8710) (0.0116,-0.8723) (0.0132,-0.8734) (0.0149,-0.8745) (0.0167,-0.8754) (0.0185,-0.8763) (0.0203,-0.8770) (0.0222,-0.8777) (0.0241,-0.8783) (0.0261,-0.8787) (0.0280,-0.8791) (0.0300,-0.8794) (0.0320,-0.8796) (0.0340,-0.8798) (0.0360,-0.8799) (0.0380,-0.8799) (0.0400,-0.8800) (0.0420,-0.8800) (0.0440,-0.8800) (0.0460,-0.8800) (0.0480,-0.8800) (0.0500,-0.8800) (0.0520,-0.8800) (0.0540,-0.8800)};
\draw[geoblue!60,densely dashed] (0.1340,0) ellipse[x radius=0.0968,y radius=0.8800];
\draw[geoblue!65] (1.3340,0) ellipse[x radius=0.0968,y radius=0.8800];
\fill (0,0) circle[radius=1.4pt] node[below left=2pt] {$p$};
\node[align=center] at (-.31,1.31) {flat\\cap};
\draw[gray!70,->] (-.23,1.02)--(-.01,.52);
\node at (0.8271,1.2200) {cylindrical end};
\end{scope}
\node[anchor=west] at (0,3.27) {$v=0$, $a=0$};
\end{tikzpicture}
\caption{The sharpness construction for $n=3$, $r=1$, and $v=0$.
The solid curve is a smooth approximating profile; the dashed
curve on the left is its limit $\min\{t,1\}$. On the right: a flat cap is rounded in
the orange annulus and continued by a cylinder. The annulus lies
inside $B_p(1)$, and $I_p(1)\to16\pi$.}
\label{fig:rot-zero}
\end{figure}

\begin{figure}[htbp]
\centering
\begin{tikzpicture}[font=\footnotesize,>=Latex,line cap=round,line join=round]
\begin{scope}[x=1.65cm,y=1.35cm]
\fill[geoamber!12] (.84,0) rectangle (.92,1.82);
\draw[->,gray!70] (0,0)--(2.4,0) node[right] {$t$};
\draw[->,gray!70] (0,0)--(0,1.9) node[above] {$f_\varepsilon(t)$};
\draw[gray,dashed] (0,0)--(1,1)--(2.2,1.6);
\draw[geoblue,very thick] plot coordinates {(0.0000,0.0000) (0.0200,0.0200) (0.0400,0.0400) (0.0600,0.0600) (0.0800,0.0800) (0.1000,0.1000) (0.1200,0.1200) (0.1400,0.1400) (0.1600,0.1600) (0.1800,0.1800) (0.2000,0.2000) (0.2200,0.2200) (0.2400,0.2400) (0.2600,0.2600) (0.2800,0.2800) (0.3000,0.3000) (0.3200,0.3200) (0.3400,0.3400) (0.3600,0.3600) (0.3800,0.3800) (0.4000,0.4000) (0.4200,0.4200) (0.4400,0.4400) (0.4600,0.4600) (0.4800,0.4800) (0.5000,0.5000) (0.5200,0.5200) (0.5400,0.5400) (0.5600,0.5600) (0.5800,0.5800) (0.6000,0.6000) (0.6200,0.6200) (0.6400,0.6400) (0.6600,0.6600) (0.6800,0.6800) (0.7000,0.7000) (0.7200,0.7200) (0.7400,0.7400) (0.7600,0.7600) (0.7800,0.7800) (0.8000,0.8000) (0.8200,0.8200) (0.8250,0.8250) (0.8275,0.8275) (0.8300,0.8300) (0.8325,0.8325) (0.8350,0.8350) (0.8375,0.8375) (0.8400,0.8400) (0.8425,0.8425) (0.8450,0.8450) (0.8475,0.8475) (0.8500,0.8500) (0.8525,0.8525) (0.8550,0.8550) (0.8575,0.8575) (0.8600,0.8599) (0.8625,0.8623) (0.8650,0.8646) (0.8675,0.8669) (0.8700,0.8691) (0.8725,0.8713) (0.8750,0.8733) (0.8775,0.8753) (0.8800,0.8772) (0.8825,0.8791) (0.8850,0.8808) (0.8875,0.8825) (0.8900,0.8841) (0.8925,0.8857) (0.8950,0.8871) (0.8975,0.8885) (0.9000,0.8899) (0.9025,0.8912) (0.9050,0.8925) (0.9075,0.8937) (0.9100,0.8950) (0.9125,0.8962) (0.9150,0.8975) (0.9175,0.8987) (0.9200,0.9000) (0.9225,0.9012) (0.9250,0.9025) (0.9275,0.9037) (0.9300,0.9050) (0.9325,0.9062) (0.9350,0.9075) (0.9375,0.9087) (0.9400,0.9100) (0.9425,0.9112) (0.9450,0.9125) (0.9475,0.9137) (0.9600,0.9200) (0.9800,0.9300) (1.0000,0.9400) (1.0200,0.9500) (1.0400,0.9600) (1.0600,0.9700) (1.0800,0.9800) (1.1000,0.9900) (1.1200,1.0000) (1.1400,1.0100) (1.1600,1.0200) (1.1800,1.0300) (1.2000,1.0400) (1.2200,1.0500) (1.2400,1.0600) (1.2600,1.0700) (1.2800,1.0800) (1.3000,1.0900) (1.3200,1.1000) (1.3400,1.1100) (1.3600,1.1200) (1.3800,1.1300) (1.4000,1.1400) (1.4200,1.1500) (1.4400,1.1600) (1.4600,1.1700) (1.4800,1.1800) (1.5000,1.1900) (1.5200,1.2000) (1.5400,1.2100) (1.5600,1.2200) (1.5800,1.2300) (1.6000,1.2400) (1.6200,1.2500) (1.6400,1.2600) (1.6600,1.2700) (1.6800,1.2800) (1.7000,1.2900) (1.7200,1.3000) (1.7400,1.3100) (1.7600,1.3200) (1.7800,1.3300) (1.8000,1.3400) (1.8200,1.3500) (1.8400,1.3600) (1.8600,1.3700) (1.8800,1.3800) (1.9000,1.3900) (1.9200,1.4000) (1.9400,1.4100) (1.9600,1.4200) (1.9800,1.4300) (2.0000,1.4400) (2.0200,1.4500) (2.0400,1.4600) (2.0600,1.4700) (2.0800,1.4800) (2.1000,1.4900) (2.1200,1.5000) (2.1400,1.5100) (2.1600,1.5200) (2.1800,1.5300) (2.2000,1.5400)};
\draw[gray!60,densely dashed] (1,0)--(1,1.7);
\draw (1,.025)--(1,-.025) node[below] {$1$};
\draw (.025,1)--(-.025,1) node[left] {$1$};
\node[anchor=north east] at (-.02,-.02) {$0$};
\node[geoamber,align=center] at (.67,1.55) {thin\\transition};
\draw[geoamber,->] (.75,1.25)--(.88,.96);
\node at (1.1,-.52) {warping function};
\end{scope}
\begin{scope}[shift={(6.4cm,1.05cm)},x=2.05cm,y=1.02cm]
\fill[geoblue!6] (0,0) ellipse[x radius=.11,y radius=.84];
\draw[geoblue!55,dashed] (0,0) ellipse[x radius=.11,y radius=.84];
\draw[gray!60,->] (-.22,0)--(1.3821,0) node[right] {$z$};
\draw[geoblue,thick] plot coordinates {(0.0000,0.0000) (0.0000,0.0200) (0.0000,0.0400) (0.0000,0.0600) (0.0000,0.0800) (0.0000,0.1000) (0.0000,0.1200) (0.0000,0.1400) (0.0000,0.1600) (0.0000,0.1800) (0.0000,0.2000) (0.0000,0.2200) (0.0000,0.2400) (0.0000,0.2600) (0.0000,0.2800) (0.0000,0.3000) (0.0000,0.3200) (0.0000,0.3400) (0.0000,0.3600) (0.0000,0.3800) (0.0000,0.4000) (0.0000,0.4200) (0.0000,0.4400) (0.0000,0.4600) (0.0000,0.4800) (0.0000,0.5000) (0.0000,0.5200) (0.0000,0.5400) (0.0000,0.5600) (0.0000,0.5800) (0.0000,0.6000) (0.0000,0.6200) (0.0000,0.6400) (0.0000,0.6600) (0.0000,0.6800) (0.0000,0.7000) (0.0000,0.7200) (0.0000,0.7400) (0.0000,0.7600) (0.0000,0.7800) (0.0000,0.8000) (0.0000,0.8200) (0.0000,0.8250) (0.0000,0.8275) (0.0000,0.8300) (0.0000,0.8325) (0.0000,0.8350) (0.0000,0.8375) (0.0000,0.8400) (0.0000,0.8425) (0.0000,0.8450) (0.0000,0.8475) (0.0001,0.8500) (0.0002,0.8525) (0.0004,0.8550) (0.0008,0.8575) (0.0014,0.8599) (0.0021,0.8623) (0.0030,0.8646) (0.0040,0.8669) (0.0051,0.8691) (0.0064,0.8713) (0.0078,0.8733) (0.0093,0.8753) (0.0109,0.8772) (0.0126,0.8791) (0.0144,0.8808) (0.0163,0.8825) (0.0182,0.8841) (0.0202,0.8857) (0.0222,0.8871) (0.0242,0.8885) (0.0263,0.8899) (0.0285,0.8912) (0.0306,0.8925) (0.0328,0.8937) (0.0349,0.8950) (0.0371,0.8962) (0.0393,0.8975) (0.0414,0.8987) (0.0436,0.9000) (0.0458,0.9012) (0.0479,0.9025) (0.0501,0.9037) (0.0523,0.9050) (0.0544,0.9062) (0.0566,0.9075) (0.0588,0.9087) (0.0609,0.9100) (0.0631,0.9112) (0.0653,0.9125) (0.0674,0.9137) (0.0782,0.9200) (0.0956,0.9300) (0.1129,0.9400) (0.1302,0.9500) (0.1475,0.9600) (0.1648,0.9700) (0.1822,0.9800) (0.1995,0.9900) (0.2168,1.0000) (0.2341,1.0100) (0.2515,1.0200) (0.2688,1.0300) (0.2861,1.0400) (0.3034,1.0500) (0.3207,1.0600) (0.3381,1.0700) (0.3554,1.0800) (0.3727,1.0900) (0.3900,1.1000) (0.4073,1.1100) (0.4247,1.1200) (0.4420,1.1300) (0.4593,1.1400) (0.4766,1.1500) (0.4939,1.1600) (0.5113,1.1700) (0.5286,1.1800) (0.5459,1.1900) (0.5632,1.2000) (0.5805,1.2100) (0.5979,1.2200) (0.6152,1.2300) (0.6325,1.2400) (0.6498,1.2500) (0.6671,1.2600) (0.6845,1.2700) (0.7018,1.2800) (0.7191,1.2900) (0.7364,1.3000) (0.7537,1.3100) (0.7711,1.3200) (0.7884,1.3300) (0.8057,1.3400) (0.8230,1.3500) (0.8403,1.3600) (0.8577,1.3700) (0.8750,1.3800) (0.8923,1.3900) (0.9096,1.4000) (0.9270,1.4100) (0.9443,1.4200) (0.9616,1.4300) (0.9789,1.4400) (0.9962,1.4500) (1.0136,1.4600) (1.0309,1.4700) (1.0482,1.4800) (1.0655,1.4900) (1.0828,1.5000) (1.1002,1.5100) (1.1175,1.5200) (1.1348,1.5300) (1.1521,1.5400)};
\draw[geoblue,thick] plot coordinates {(0.0000,-0.0000) (0.0000,-0.0200) (0.0000,-0.0400) (0.0000,-0.0600) (0.0000,-0.0800) (0.0000,-0.1000) (0.0000,-0.1200) (0.0000,-0.1400) (0.0000,-0.1600) (0.0000,-0.1800) (0.0000,-0.2000) (0.0000,-0.2200) (0.0000,-0.2400) (0.0000,-0.2600) (0.0000,-0.2800) (0.0000,-0.3000) (0.0000,-0.3200) (0.0000,-0.3400) (0.0000,-0.3600) (0.0000,-0.3800) (0.0000,-0.4000) (0.0000,-0.4200) (0.0000,-0.4400) (0.0000,-0.4600) (0.0000,-0.4800) (0.0000,-0.5000) (0.0000,-0.5200) (0.0000,-0.5400) (0.0000,-0.5600) (0.0000,-0.5800) (0.0000,-0.6000) (0.0000,-0.6200) (0.0000,-0.6400) (0.0000,-0.6600) (0.0000,-0.6800) (0.0000,-0.7000) (0.0000,-0.7200) (0.0000,-0.7400) (0.0000,-0.7600) (0.0000,-0.7800) (0.0000,-0.8000) (0.0000,-0.8200) (0.0000,-0.8250) (0.0000,-0.8275) (0.0000,-0.8300) (0.0000,-0.8325) (0.0000,-0.8350) (0.0000,-0.8375) (0.0000,-0.8400) (0.0000,-0.8425) (0.0000,-0.8450) (0.0000,-0.8475) (0.0001,-0.8500) (0.0002,-0.8525) (0.0004,-0.8550) (0.0008,-0.8575) (0.0014,-0.8599) (0.0021,-0.8623) (0.0030,-0.8646) (0.0040,-0.8669) (0.0051,-0.8691) (0.0064,-0.8713) (0.0078,-0.8733) (0.0093,-0.8753) (0.0109,-0.8772) (0.0126,-0.8791) (0.0144,-0.8808) (0.0163,-0.8825) (0.0182,-0.8841) (0.0202,-0.8857) (0.0222,-0.8871) (0.0242,-0.8885) (0.0263,-0.8899) (0.0285,-0.8912) (0.0306,-0.8925) (0.0328,-0.8937) (0.0349,-0.8950) (0.0371,-0.8962) (0.0393,-0.8975) (0.0414,-0.8987) (0.0436,-0.9000) (0.0458,-0.9012) (0.0479,-0.9025) (0.0501,-0.9037) (0.0523,-0.9050) (0.0544,-0.9062) (0.0566,-0.9075) (0.0588,-0.9087) (0.0609,-0.9100) (0.0631,-0.9112) (0.0653,-0.9125) (0.0674,-0.9137) (0.0782,-0.9200) (0.0956,-0.9300) (0.1129,-0.9400) (0.1302,-0.9500) (0.1475,-0.9600) (0.1648,-0.9700) (0.1822,-0.9800) (0.1995,-0.9900) (0.2168,-1.0000) (0.2341,-1.0100) (0.2515,-1.0200) (0.2688,-1.0300) (0.2861,-1.0400) (0.3034,-1.0500) (0.3207,-1.0600) (0.3381,-1.0700) (0.3554,-1.0800) (0.3727,-1.0900) (0.3900,-1.1000) (0.4073,-1.1100) (0.4247,-1.1200) (0.4420,-1.1300) (0.4593,-1.1400) (0.4766,-1.1500) (0.4939,-1.1600) (0.5113,-1.1700) (0.5286,-1.1800) (0.5459,-1.1900) (0.5632,-1.2000) (0.5805,-1.2100) (0.5979,-1.2200) (0.6152,-1.2300) (0.6325,-1.2400) (0.6498,-1.2500) (0.6671,-1.2600) (0.6845,-1.2700) (0.7018,-1.2800) (0.7191,-1.2900) (0.7364,-1.3000) (0.7537,-1.3100) (0.7711,-1.3200) (0.7884,-1.3300) (0.8057,-1.3400) (0.8230,-1.3500) (0.8403,-1.3600) (0.8577,-1.3700) (0.8750,-1.3800) (0.8923,-1.3900) (0.9096,-1.4000) (0.9270,-1.4100) (0.9443,-1.4200) (0.9616,-1.4300) (0.9789,-1.4400) (0.9962,-1.4500) (1.0136,-1.4600) (1.0309,-1.4700) (1.0482,-1.4800) (1.0655,-1.4900) (1.0828,-1.5000) (1.1002,-1.5100) (1.1175,-1.5200) (1.1348,-1.5300) (1.1521,-1.5400)};
\draw[geoamber,line width=2pt] plot coordinates {(0.0000,0.8400) (0.0000,0.8420) (0.0000,0.8440) (0.0000,0.8460) (0.0000,0.8480) (0.0001,0.8500) (0.0002,0.8520) (0.0003,0.8540) (0.0006,0.8560) (0.0009,0.8579) (0.0014,0.8599) (0.0019,0.8618) (0.0026,0.8637) (0.0034,0.8655) (0.0042,0.8674) (0.0051,0.8691) (0.0062,0.8708) (0.0073,0.8725) (0.0084,0.8741) (0.0097,0.8757) (0.0109,0.8772) (0.0123,0.8787) (0.0137,0.8801) (0.0152,0.8815) (0.0166,0.8828) (0.0182,0.8841) (0.0198,0.8854) (0.0214,0.8865) (0.0230,0.8877) (0.0247,0.8888) (0.0263,0.8899) (0.0281,0.8909) (0.0298,0.8920) (0.0315,0.8930) (0.0332,0.8940) (0.0349,0.8950) (0.0367,0.8960) (0.0384,0.8970) (0.0401,0.8980) (0.0419,0.8990) (0.0436,0.9000)};
\draw[geoamber,line width=2pt] plot coordinates {(0.0000,-0.8400) (0.0000,-0.8420) (0.0000,-0.8440) (0.0000,-0.8460) (0.0000,-0.8480) (0.0001,-0.8500) (0.0002,-0.8520) (0.0003,-0.8540) (0.0006,-0.8560) (0.0009,-0.8579) (0.0014,-0.8599) (0.0019,-0.8618) (0.0026,-0.8637) (0.0034,-0.8655) (0.0042,-0.8674) (0.0051,-0.8691) (0.0062,-0.8708) (0.0073,-0.8725) (0.0084,-0.8741) (0.0097,-0.8757) (0.0109,-0.8772) (0.0123,-0.8787) (0.0137,-0.8801) (0.0152,-0.8815) (0.0166,-0.8828) (0.0182,-0.8841) (0.0198,-0.8854) (0.0214,-0.8865) (0.0230,-0.8877) (0.0247,-0.8888) (0.0263,-0.8899) (0.0281,-0.8909) (0.0298,-0.8920) (0.0315,-0.8930) (0.0332,-0.8940) (0.0349,-0.8950) (0.0367,-0.8960) (0.0384,-0.8970) (0.0401,-0.8980) (0.0419,-0.8990) (0.0436,-0.9000)};
\draw[geoblue!60,densely dashed] (0.1129,0) ellipse[x radius=0.1034,y radius=0.9400];
\draw[geoblue!65] (1.1521,0) ellipse[x radius=0.1694,y radius=1.5400];
\fill (0,0) circle[radius=1.4pt] node[below left=2pt] {$p$};
\node[align=center] at (-.31,1.31) {flat\\cap};
\draw[gray!70,->] (-.23,1.02)--(-.01,.52);
\node at (0.7143,1.8800) {conical end};
\end{scope}
\node[anchor=west] at (0,3.27) {$v=\frac{1}{4}$, $a=\frac{1}{2}$};
\end{tikzpicture}
\caption{The same construction with $n=3$, $r=1$, and $v=\frac{1}{4}$.
Here $f_\varepsilon'=a=\frac{1}{2}$ beyond the transition, so the end
is conical. The dashed limit on the left is $t$ for $t\le1$
and $1+\frac{t-1}{2}$ for $t\ge1$. The transition stays inside the
unit ball and $I_p(1)\to8\pi$.}
\label{fig:rot-quarter}
\end{figure}
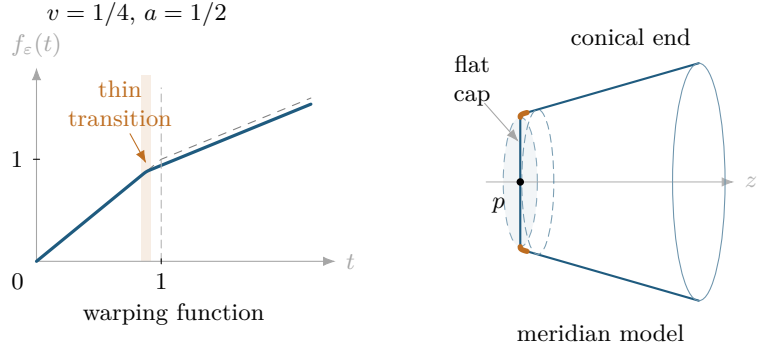

\FloatBarrier

\begin{proof}[Proof of the lower bound in Theorem~\ref{thm:rotational}]
Write
\[
 \nu=v_p(1)=n\int_0^1f(t)^{n-1}\dd t,
 \quad u=f',\quad m=-u',\quad b=f(1),\quad c=u(1).
\]

Integration by parts gives
\[
\int_0^1f^{n-2}m\dd t=(n-2)\int_0^1f^{n-3}u^2\dd t-cb^{n-2}.
\]

Substituting this into \eqref{eq:rot-integral} and using $m\ge0$ gives
\begin{equation}\label{eq:rot-lower-reduction}
 \frac{I_p(1)}{n(n-1)\omega_n}
 \ge(n-2)\int_0^1f^{n-3}\dd t-cb^{n-2}.
\end{equation}
Concavity and $f(0)=0$ imply $f(t)\ge bt$ and $c\le b$, so
\[
 cb^{n-2}\le b^{n-1}
 \le n\int_0^1f(t)^{n-1}\dd t=\nu.
\]
It remains to show
\begin{equation}\label{eq:rot-power-comparison}
 (n-2)\int_0^1f(t)^{n-3}\dd t
 \ge\left(n\int_0^1f(t)^{n-1}\dd t\right)^{\frac{n-3}{n-1}}=\nu^{\frac{n-3}{n-1}}.
\end{equation}
Put $a_0=\nu^{\frac{1}{n-1}}$. For $n=3$ both sides equal one.
For $n>3$, the definition of $a_0$ gives
$\int_0^1[f^{n-1}-(a_0t)^{n-1}]\dd t=0$.
If $f(t)=a_0t$ identically, equality holds. Otherwise,
concavity makes $\frac{f(t)}{t}$ nonincreasing, so $f(t)-a_0t$
changes sign from positive to negative at some $t_*\in(0,1)$.
Set $z=f(t_*)=a_0t_*>0$. Since $f$ is nondecreasing,
$a_0t\le f(t)\le z$ for $t\le t_*$ and
$a_0t\ge f(t)\ge z$ for $t\ge t_*$. The ratio of the derivatives of
$s^{n-3}$ and $s^{n-1}$ is $\frac{n-3}{(n-1)s^2}$.
Cauchy's mean value theorem, with the ordering reversed after $t_*$, gives 
\[
 f(t)^{n-3}-(a_0t)^{n-3}
 \ge\frac{n-3}{(n-1)z^2}
       \bigl[f(t)^{n-1}-(a_0t)^{n-1}\bigr],\qquad t\in[0,1].
\]
After integration over $[0,1]$, the right-hand side vanishes by
the choice of $a_0$. Since
$\int_0^1(a_0t)^{n-3}\dd t=\frac{a_0^{n-3}}{n-2}$, this proves
\eqref{eq:rot-power-comparison}. Together with
\eqref{eq:rot-lower-reduction} and $cb^{n-2}\le\nu$, this
proves the lower bound.

To prove sharpness, fix $\nu\in(0,1)$ and use the smooth nondecreasing cutoff $\chi$
from the upper-bound construction. For $0<\varepsilon<\frac{1}{4}$ and
$0\le a\le1$, define
\[
 u_{\varepsilon,a}(t)
 =1-(1-a)\chi\!\left(\frac{t-\varepsilon}{\varepsilon}\right),
 \qquad
 f_{\varepsilon,a}(t)=\int_0^tu_{\varepsilon,a}(s)\dd s.
\]
The functions $f_{\varepsilon,a}$ equal $t$ for $t\le\varepsilon$,
have constant slope $a$ for $t\ge2\varepsilon$, and define smooth
complete metrics with $K\ge0$. They satisfy
$0\le f_{\varepsilon,a}(t)-at\le2\varepsilon$.
Their unit-ball volume ratios depend continuously and strictly
increasingly on $a$, and converge uniformly on $[0,1]$ to
$a^{n-1}$. The ratio at $a=0$ tends to zero, while the ratio at
$a=1$ is one. Thus, for all sufficiently small $\varepsilon$, we may
choose $a_\varepsilon\to a_0$ so that
$n\int_0^1 f_{\varepsilon,a_\varepsilon}^{n-1}\dd t=\nu$ exactly.

Write $f_\varepsilon=f_{\varepsilon,a_\varepsilon}$ and let
$g_\varepsilon$ be the corresponding metrics. Then
$f_\varepsilon\to a_0t$ uniformly on $[0,1]$ and
$f_\varepsilon'\to a_0$ for every $t>0$.
To pass to the limit, use the exact integration-by-parts identity
\[
 \frac{I_p(1)}{n(n-1)\omega_n}
 =(n-2)\int_0^1 f^{n-3}(1+u^2)\dd t-2u(1)f(1)^{n-2}.
\]
Since $0\le f_\varepsilon'\le1$ and $f_\varepsilon\le t$, the
integrand is bounded by $2t^{n-3}$, which is integrable for $n\ge3$.
Dominated convergence and convergence of the boundary term give
\[
\lim_{\varepsilon\downarrow0}
\frac{I_p^{g_\varepsilon}(1)}{n(n-1)\omega_n}
=a_0^{n-3}-a_0^{n-1}
=\nu^{\frac{n-3}{n-1}}-\nu.
\]
This proves sharpness at each prescribed local volume ratio.
For $\nu=1$, the Euclidean metric attains equality.
\end{proof}

Figure~\ref{fig:rot-lower-quarter} illustrates the lower-bound
construction. The model is a cone smoothed near $p$.

\begin{figure}[htbp]
\centering
\begin{tikzpicture}[font=\footnotesize,>=Latex,line cap=round,line join=round]
\begin{scope}[x=1.85cm,y=2.65cm]
\fill[geoamber!12] (.08,0) rectangle (.16,.98);
\draw[->,gray!70] (0,0)--(2.03,0) node[right] {$t$};
\draw[->,gray!70] (0,0)--(0,1.04) node[above] {$f_\varepsilon(t)$};
\draw[gray,dashed] (0,0)--(1.8,.9);
\draw[gray!60,densely dashed] (1,0)--(1,.59);
\draw (1,.015)--(1,-.015) node[below] {$1$};
\node[anchor=north east] at (-.02,-.01) {$0$};
\draw[geoblue,very thick] plot coordinates {(0.00000,0.00000) (0.08000,0.08000) (0.08100,0.08100) (0.08200,0.08200) (0.08300,0.08300) (0.08400,0.08400) (0.08500,0.08500) (0.08600,0.08600) (0.08700,0.08700) (0.08800,0.08800) (0.08900,0.08900) (0.09000,0.09000) (0.09100,0.09100) (0.09200,0.09200) (0.09300,0.09299) (0.09400,0.09399) (0.09500,0.09498) (0.09600,0.09597) (0.09700,0.09695) (0.09800,0.09793) (0.09900,0.09890) (0.10000,0.09986) (0.10100,0.10082) (0.10200,0.10177) (0.10300,0.10270) (0.10400,0.10363) (0.10500,0.10455) (0.10600,0.10545) (0.10700,0.10634) (0.10800,0.10722) (0.10900,0.10808) (0.11000,0.10893) (0.11100,0.10977) (0.11200,0.11059) (0.11300,0.11139) (0.11400,0.11219) (0.11500,0.11296) (0.11600,0.11373) (0.11700,0.11447) (0.11800,0.11521) (0.11900,0.11592) (0.12000,0.11662) (0.12100,0.11731) (0.12200,0.11798) (0.12300,0.11864) (0.12400,0.11928) (0.12500,0.11990) (0.12600,0.12051) (0.12700,0.12111) (0.12800,0.12169) (0.12900,0.12225) (0.13000,0.12280) (0.13100,0.12334) (0.13200,0.12387) (0.13300,0.12438) (0.13400,0.12487) (0.13500,0.12536) (0.13600,0.12583) (0.13700,0.12629) (0.13800,0.12674) (0.13900,0.12718) (0.14000,0.12761) (0.14100,0.12804) (0.14200,0.12845) (0.14300,0.12886) (0.14400,0.12927) (0.14500,0.12967) (0.14600,0.13006) (0.14700,0.13045) (0.14800,0.13084) (0.14900,0.13123) (0.15000,0.13162) (0.15100,0.13201) (0.15200,0.13240) (0.15300,0.13278) (0.15400,0.13317) (0.15500,0.13356) (0.15600,0.13395) (0.15700,0.13433) (0.15800,0.13472) (0.15900,0.13511) (0.16000,0.13550) (1.00000,0.46093) (1.80000,0.77087)};
\draw[geoamber,line width=2pt] plot coordinates {(0.08000,0.08000) (0.08100,0.08100) (0.08200,0.08200) (0.08300,0.08300) (0.08400,0.08400) (0.08500,0.08500) (0.08600,0.08600) (0.08700,0.08700) (0.08800,0.08800) (0.08900,0.08900) (0.09000,0.09000) (0.09100,0.09100) (0.09200,0.09200) (0.09300,0.09299) (0.09400,0.09399) (0.09500,0.09498) (0.09600,0.09597) (0.09700,0.09695) (0.09800,0.09793) (0.09900,0.09890) (0.10000,0.09986) (0.10100,0.10082) (0.10200,0.10177) (0.10300,0.10270) (0.10400,0.10363) (0.10500,0.10455) (0.10600,0.10545) (0.10700,0.10634) (0.10800,0.10722) (0.10900,0.10808) (0.11000,0.10893) (0.11100,0.10977) (0.11200,0.11059) (0.11300,0.11139) (0.11400,0.11219) (0.11500,0.11296) (0.11600,0.11373) (0.11700,0.11447) (0.11800,0.11521) (0.11900,0.11592) (0.12000,0.11662) (0.12100,0.11731) (0.12200,0.11798) (0.12300,0.11864) (0.12400,0.11928) (0.12500,0.11990) (0.12600,0.12051) (0.12700,0.12111) (0.12800,0.12169) (0.12900,0.12225) (0.13000,0.12280) (0.13100,0.12334) (0.13200,0.12387) (0.13300,0.12438) (0.13400,0.12487) (0.13500,0.12536) (0.13600,0.12583) (0.13700,0.12629) (0.13800,0.12674) (0.13900,0.12718) (0.14000,0.12761) (0.14100,0.12804) (0.14200,0.12845) (0.14300,0.12886) (0.14400,0.12927) (0.14500,0.12967) (0.14600,0.13006) (0.14700,0.13045) (0.14800,0.13084) (0.14900,0.13123) (0.15000,0.13162) (0.15100,0.13201) (0.15200,0.13240) (0.15300,0.13278) (0.15400,0.13317) (0.15500,0.13356) (0.15600,0.13395) (0.15700,0.13433) (0.15800,0.13472) (0.15900,0.13511) (0.16000,0.13550)};
\node[gray,anchor=west] at (1.83,.92) {$\frac{t}{2}$};
\node[geoblue,anchor=west] at (1.83,0.77087) {$f_\varepsilon$};
\node[geoamber,align=center] at (.55,.69) {transition\\near $p$};
\draw[geoamber,->] (.45,.56)--(.12,0.11662);
\node at (1.05,-.19) {warping function};
\end{scope}
\begin{scope}[shift={(6.3cm,1.28cm)},x=2.15cm,y=1.40cm]
\draw[gray,dashed] (0,0)--(1.55885,.9);
\draw[gray,dashed] (0,0)--(1.55885,-.9);
\fill[geoblue!6] (0,0) ellipse[x radius=.018,y radius=.08];
\draw[geoblue!55,dashed] (0,0) ellipse[x radius=.018,y radius=.08];
\draw[gray!60,->] (-.19,0)--(1.74892,0) node[right] {$z$};
\draw[geoblue,thick] plot coordinates {(0.00000,0.00000) (0.00000,0.08000) (0.00000,0.08100) (0.00000,0.08200) (0.00000,0.08300) (0.00000,0.08400) (0.00000,0.08500) (0.00000,0.08600) (0.00001,0.08700) (0.00002,0.08800) (0.00003,0.08900) (0.00006,0.09000) (0.00011,0.09100) (0.00017,0.09200) (0.00025,0.09299) (0.00035,0.09399) (0.00048,0.09498) (0.00063,0.09597) (0.00081,0.09695) (0.00102,0.09793) (0.00126,0.09890) (0.00153,0.09986) (0.00182,0.10082) (0.00214,0.10177) (0.00249,0.10270) (0.00287,0.10363) (0.00327,0.10455) (0.00370,0.10545) (0.00416,0.10634) (0.00464,0.10722) (0.00514,0.10808) (0.00567,0.10893) (0.00621,0.10977) (0.00678,0.11059) (0.00737,0.11139) (0.00798,0.11219) (0.00861,0.11296) (0.00926,0.11373) (0.00993,0.11447) (0.01061,0.11521) (0.01130,0.11592) (0.01202,0.11662) (0.01274,0.11731) (0.01349,0.11798) (0.01424,0.11864) (0.01501,0.11928) (0.01579,0.11990) (0.01658,0.12051) (0.01739,0.12111) (0.01820,0.12169) (0.01902,0.12225) (0.01986,0.12280) (0.02070,0.12334) (0.02155,0.12387) (0.02241,0.12438) (0.02328,0.12487) (0.02416,0.12536) (0.02504,0.12583) (0.02593,0.12629) (0.02682,0.12674) (0.02772,0.12718) (0.02862,0.12761) (0.02952,0.12804) (0.03043,0.12845) (0.03135,0.12886) (0.03226,0.12927) (0.03318,0.12967) (0.03410,0.13006) (0.03502,0.13045) (0.03594,0.13084) (0.03686,0.13123) (0.03778,0.13162) (0.03870,0.13201) (0.03962,0.13240) (0.04054,0.13278) (0.04147,0.13317) (0.04239,0.13356) (0.04331,0.13395) (0.04423,0.13433) (0.04515,0.13472) (0.04608,0.13511) (0.04700,0.13550) (0.82140,0.46093) (1.55892,0.77087)};
\draw[geoblue,thick] plot coordinates {(0.00000,-0.00000) (0.00000,-0.08000) (0.00000,-0.08100) (0.00000,-0.08200) (0.00000,-0.08300) (0.00000,-0.08400) (0.00000,-0.08500) (0.00000,-0.08600) (0.00001,-0.08700) (0.00002,-0.08800) (0.00003,-0.08900) (0.00006,-0.09000) (0.00011,-0.09100) (0.00017,-0.09200) (0.00025,-0.09299) (0.00035,-0.09399) (0.00048,-0.09498) (0.00063,-0.09597) (0.00081,-0.09695) (0.00102,-0.09793) (0.00126,-0.09890) (0.00153,-0.09986) (0.00182,-0.10082) (0.00214,-0.10177) (0.00249,-0.10270) (0.00287,-0.10363) (0.00327,-0.10455) (0.00370,-0.10545) (0.00416,-0.10634) (0.00464,-0.10722) (0.00514,-0.10808) (0.00567,-0.10893) (0.00621,-0.10977) (0.00678,-0.11059) (0.00737,-0.11139) (0.00798,-0.11219) (0.00861,-0.11296) (0.00926,-0.11373) (0.00993,-0.11447) (0.01061,-0.11521) (0.01130,-0.11592) (0.01202,-0.11662) (0.01274,-0.11731) (0.01349,-0.11798) (0.01424,-0.11864) (0.01501,-0.11928) (0.01579,-0.11990) (0.01658,-0.12051) (0.01739,-0.12111) (0.01820,-0.12169) (0.01902,-0.12225) (0.01986,-0.12280) (0.02070,-0.12334) (0.02155,-0.12387) (0.02241,-0.12438) (0.02328,-0.12487) (0.02416,-0.12536) (0.02504,-0.12583) (0.02593,-0.12629) (0.02682,-0.12674) (0.02772,-0.12718) (0.02862,-0.12761) (0.02952,-0.12804) (0.03043,-0.12845) (0.03135,-0.12886) (0.03226,-0.12927) (0.03318,-0.12967) (0.03410,-0.13006) (0.03502,-0.13045) (0.03594,-0.13084) (0.03686,-0.13123) (0.03778,-0.13162) (0.03870,-0.13201) (0.03962,-0.13240) (0.04054,-0.13278) (0.04147,-0.13317) (0.04239,-0.13356) (0.04331,-0.13395) (0.04423,-0.13433) (0.04515,-0.13472) (0.04608,-0.13511) (0.04700,-0.13550) (0.82140,-0.46093) (1.55892,-0.77087)};
\draw[geoamber,line width=2pt] plot coordinates {(0.00000,0.08000) (0.00000,0.08100) (0.00000,0.08200) (0.00000,0.08300) (0.00000,0.08400) (0.00000,0.08500) (0.00000,0.08600) (0.00001,0.08700) (0.00002,0.08800) (0.00003,0.08900) (0.00006,0.09000) (0.00011,0.09100) (0.00017,0.09200) (0.00025,0.09299) (0.00035,0.09399) (0.00048,0.09498) (0.00063,0.09597) (0.00081,0.09695) (0.00102,0.09793) (0.00126,0.09890) (0.00153,0.09986) (0.00182,0.10082) (0.00214,0.10177) (0.00249,0.10270) (0.00287,0.10363) (0.00327,0.10455) (0.00370,0.10545) (0.00416,0.10634) (0.00464,0.10722) (0.00514,0.10808) (0.00567,0.10893) (0.00621,0.10977) (0.00678,0.11059) (0.00737,0.11139) (0.00798,0.11219) (0.00861,0.11296) (0.00926,0.11373) (0.00993,0.11447) (0.01061,0.11521) (0.01130,0.11592) (0.01202,0.11662) (0.01274,0.11731) (0.01349,0.11798) (0.01424,0.11864) (0.01501,0.11928) (0.01579,0.11990) (0.01658,0.12051) (0.01739,0.12111) (0.01820,0.12169) (0.01902,0.12225) (0.01986,0.12280) (0.02070,0.12334) (0.02155,0.12387) (0.02241,0.12438) (0.02328,0.12487) (0.02416,0.12536) (0.02504,0.12583) (0.02593,0.12629) (0.02682,0.12674) (0.02772,0.12718) (0.02862,0.12761) (0.02952,0.12804) (0.03043,0.12845) (0.03135,0.12886) (0.03226,0.12927) (0.03318,0.12967) (0.03410,0.13006) (0.03502,0.13045) (0.03594,0.13084) (0.03686,0.13123) (0.03778,0.13162) (0.03870,0.13201) (0.03962,0.13240) (0.04054,0.13278) (0.04147,0.13317) (0.04239,0.13356) (0.04331,0.13395) (0.04423,0.13433) (0.04515,0.13472) (0.04608,0.13511) (0.04700,0.13550)};
\draw[geoamber,line width=2pt] plot coordinates {(0.00000,-0.08000) (0.00000,-0.08100) (0.00000,-0.08200) (0.00000,-0.08300) (0.00000,-0.08400) (0.00000,-0.08500) (0.00000,-0.08600) (0.00001,-0.08700) (0.00002,-0.08800) (0.00003,-0.08900) (0.00006,-0.09000) (0.00011,-0.09100) (0.00017,-0.09200) (0.00025,-0.09299) (0.00035,-0.09399) (0.00048,-0.09498) (0.00063,-0.09597) (0.00081,-0.09695) (0.00102,-0.09793) (0.00126,-0.09890) (0.00153,-0.09986) (0.00182,-0.10082) (0.00214,-0.10177) (0.00249,-0.10270) (0.00287,-0.10363) (0.00327,-0.10455) (0.00370,-0.10545) (0.00416,-0.10634) (0.00464,-0.10722) (0.00514,-0.10808) (0.00567,-0.10893) (0.00621,-0.10977) (0.00678,-0.11059) (0.00737,-0.11139) (0.00798,-0.11219) (0.00861,-0.11296) (0.00926,-0.11373) (0.00993,-0.11447) (0.01061,-0.11521) (0.01130,-0.11592) (0.01202,-0.11662) (0.01274,-0.11731) (0.01349,-0.11798) (0.01424,-0.11864) (0.01501,-0.11928) (0.01579,-0.11990) (0.01658,-0.12051) (0.01739,-0.12111) (0.01820,-0.12169) (0.01902,-0.12225) (0.01986,-0.12280) (0.02070,-0.12334) (0.02155,-0.12387) (0.02241,-0.12438) (0.02328,-0.12487) (0.02416,-0.12536) (0.02504,-0.12583) (0.02593,-0.12629) (0.02682,-0.12674) (0.02772,-0.12718) (0.02862,-0.12761) (0.02952,-0.12804) (0.03043,-0.12845) (0.03135,-0.12886) (0.03226,-0.12927) (0.03318,-0.12967) (0.03410,-0.13006) (0.03502,-0.13045) (0.03594,-0.13084) (0.03686,-0.13123) (0.03778,-0.13162) (0.03870,-0.13201) (0.03962,-0.13240) (0.04054,-0.13278) (0.04147,-0.13317) (0.04239,-0.13356) (0.04331,-0.13395) (0.04423,-0.13433) (0.04515,-0.13472) (0.04608,-0.13511) (0.04700,-0.13550)};
\draw[geoblue!60,densely dashed] (0.82140,0)
 ellipse[x radius=0.05070,y radius=0.46093];
\node[geoblue,fill=white,inner sep=1pt] at (0.82140,0) {$t=1$};
\draw[geoblue!65] (1.55892,0)
 ellipse[x radius=0.08480,y radius=0.77087];
\fill (0,0) circle[radius=1.4pt] node[below left=2pt] {$p$};
\node[align=center] at (-.12,.75) {flat\\cap};
\draw[gray!70,->] (-.10,.53)--(-.005,.045);
\node at (1.05,1.12) {conical end};
\end{scope}
\node[anchor=west] at (0,3.27)
 {$\nu=v_p(1)=\frac{1}{4}$, $a_\varepsilon\longrightarrow \frac{1}{2}$};
\end{tikzpicture}
\caption{The sharpness construction for the rotational lower bound,
with $n=3$, $r=1$, and fixed local volume ratio $\nu=\frac{1}{4}$.
The flat cap $t\le\varepsilon$ is joined to a conical end through
the orange transition $\varepsilon\le t\le2\varepsilon$.
As $\varepsilon\downarrow0$, the cap and transition shrink to $p$,
and the profiles converge to the dashed cone $f_0(t)=\frac{t}{2}$.
The local ratio is fixed with $v_p(1)=\frac{1}{4}$ and
the asymptotic ratio is $v_\varepsilon=a_\varepsilon^2\to \frac{1}{4}$,
and $I_p^{g_\varepsilon}(1)\to6\pi$.}
\label{fig:rot-lower-quarter}
\end{figure}
\FloatBarrier

The rotational examples now give the sharp first-order coefficients
in Theorem~\ref{thm:near-maximal} and show that a quadratic remainder
is necessary in its upper estimate.

\begin{proof}[Proof of Corollary~\ref{cor:all-dimensional-sharpness}]
The rotational examples have $K\ge0$. Consequently,
\[
 \mathcal C_n(v)\ge2n(n-1)\omega_n
                   \frac{1-v^{\frac{1}{n-1}}}{1-v}\qquad(0\le v<1).
\]
Put $v=1-\delta$, with $0<\delta<1$. Taylor's theorem, including
the sign of the remainder, gives
\[
 (1-\delta)^{\frac{1}{n-1}}
 \le 1-\frac{\delta}{n-1}-\frac{n-2}{2(n-1)^2}\delta^2.
\]
Therefore
\[
 2n(n-1)\omega_n\frac{1-(1-\delta)^{\frac{1}{n-1}}}{\delta}
 \ge2n\omega_n+
       \frac{n(n-2)}{n-1}\omega_n\delta.
\]
Combining this with Theorem~\ref{thm:near-maximal} gives
\eqref{eq:all-dimensional-sandwich-C} and
\eqref{eq:all-dimensional-sharp-limit}.
Thus the leading coefficient in the upper estimate is sharp in
every dimension, and a positive quadratic remainder is necessary.

For the lower estimate, put $\eta=1-v_p(1)$. The sharp rotational
value in \eqref{eq:rot-local-lower} satisfies
\[
 n(n-1)\omega_n
 \left((1-\eta)^{\frac{n-3}{n-1}}-(1-\eta)\right)
 =2n\omega_n\eta+O_n(\eta^2).
\]
For each $\eta$, choose a sufficiently fine smooth approximation from
the lower-bound construction. Since its asymptotic volume ratio tends
to $1-\eta$, we may arrange, as $\eta\downarrow0$, that
$\delta=1-v=\eta+O(\eta^2)$ and
$I_p(1)=2n\omega_n\eta+O_n(\eta^2)$. Thus
$\frac{I_p(1)}{\eta}\to2n\omega_n$ while $\frac{\delta^2}{\eta}\to0$.
No larger first-order coefficient can hold in the lower estimate of
Theorem~\ref{thm:near-maximal}.
\end{proof}

\section{Arbitrary positive volume ratios in dimension three}
\label{sec:three-dimensional}

Throughout this section $(M^3,g)$ has $K\ge0$, $v>0$, and the
center $p$ is arbitrary. We prove Theorem~\ref{thm:main} and give a simpler proof for the
upper estimate of Theorem~\ref{thm:near-maximal} with the explicit
constants $C_3=200\pi$ and $\delta_3=\frac{1}{100}$ by comparing
the local integral of scalar curvature with the asymptotic integral \eqref{eq:asymptotic-input}. The two estimates
differ only in the convex function used in this comparison.

Use the Busemann functions and probability measure from
Section~\ref{sec:busemann-preliminaries}. For an integrable family
$\psi_u$ of functions or tensor fields, write
\begin{equation}\label{eq:weighted-average}
 \overline\psi(x)=\int_{\mathcal R_p}\psi_u(x)\dd\nu_p(u).
\end{equation}
The function $\overline b$ is the function $h$ used in
Section~\ref{hd:sec:first}, and
$\overline{\dd b\otimes\dd b}$ is the almost-everywhere average
of the weak differentials. The convexity and differentiation
properties of these averages follow from Section~\ref{sec:busemann-preliminaries}
as in the proof of Lemma~\ref{hd:lem:busemann-average}.
Put $\delta=1-v$.

\subsection{Spherical inequalities and Busemann averages}
\label{sec:spherical-inequalities}

For $s\in\R$, write $s_+:=\max\{s,0\}$ for its positive part.
Let $\dd\theta$ be the area element of the unit sphere $\Sph^2$,
of total area $4\pi$, and define
\begin{equation}\label{eq:ell}
 \ell(q)=
 \begin{cases}
  2(1-q),&0<q\le\frac{1}{2},\\
  \frac{1}{2q},&\frac{1}{2}<q\le1.
 \end{cases}
\end{equation}

\begin{lemma}\label{lem:spherical-inequalities}
Suppose $f:\Sph^2\to[0,1]$ is measurable and
\[
 \int_{\Sph^2}f\dd\theta=4\pi q,\qquad 0<q\le1.
\]
Then, for every unit vector $w\in\R^3$,
\begin{equation}\label{eq:spherical-inequalities}
 \begin{aligned}
 \frac{1}{4\pi q}\int_{\Sph^2}\ip{u}{w}^2f(u)\dd\theta(u)
     &\ge\frac{q^2}{3},\\
 \frac{2}{4\pi q}\int_{\Sph^2}\ip{u}{w}_+f(u)\dd\theta(u)
     &\le\ell(q).
 \end{aligned}
\end{equation}
\end{lemma}

\begin{proof}
Write $z=\ip{u}{w}$. Since $0\le f\le1$,
\[
 z^2f\ge q^2f-(q^2-z^2)_+.
\]
Integrating and using
$\int_{\Sph^2}\psi(z)\dd\theta=2\pi\int_{-1}^1\psi(z)\dd z$ for any function $\psi\in L^1[-1,1]$
gives
\[
 \int_{\Sph^2}z^2f\dd\theta
 \ge\int_{\Sph^2}(q^2f-(q^2-z^2)_+)\dd\theta
 =4\pi q^3-2\pi\int_{-q}^{q}(q^2-z^2)\dd z
 =\frac{4\pi q^3}{3}.
\]
For the second estimate, if $q\le\frac{1}{2}$, since $0\le f\le1$, $z_+f\le(1-2q) f+(z-(1-2q))_+$. Hence
\[
 \begin{aligned}
 2\int_{\Sph^2}z_+f\dd\theta
 &\le8\pi q(1-2q)+4\pi\int_{1-2q}^1(z-(1-2q))\dd z\\
 &=8\pi q(1-2q)+8\pi q^2=8\pi q(1-q).
 \end{aligned}
\]
If $\frac{1}{2}<q\le1$,
\[
\begin{aligned}
	2\int_{\Sph^2}z_+f\dd\theta\le2\int_{\Sph^2}z_+\dd\theta=4\pi\int_0^1z\dd z=2\pi.
\end{aligned}
\]
Dividing by $4\pi q$ proves \eqref{eq:spherical-inequalities}.
The first estimate is attained by $f=\mathbf1_{\{|z|\le q\}}$,
and the second by $f=\mathbf1_{\{z\ge1-2q\}}$, since
\begin{equation}\label{eq:spherical-extremizer}
 \frac{2}{4\pi q}\int_{\Sph^2}z_+
       \mathbf1_{\{z\ge1-2q\}}\dd\theta
 =\frac{1}{q}\int_{\max\{0,1-2q\}}^1z\dd z=\ell(q).
\end{equation}
\end{proof}

\begin{lemma}[Bounds for the Busemann averages]\label{lem:busemann-average-bounds}
Put
\begin{equation}\label{eq:busemann-average-constants}
 m=\frac{v^2}{3},\qquad D=\frac{1+\delta+\delta^2}{3}.
\end{equation}
Then, almost everywhere on $M$,
\begin{equation}\label{eq:busemann-average-bounds}
 |\nabla\overline b|\le\delta,
 \qquad
 m g\le\overline{\dd b\otimes\dd b}\le Dg.
\end{equation}
Moreover, for every unit vector $w\in T_xM$ at such a point,
\begin{equation}\label{eq:busemann-positive-part}
 2\int\ip{\nabla b_u(x)}{w}_+\dd\nu_p(u)\le\ell(v).
\end{equation}
In particular, $\Lip(\overline b)\le\delta$ and
$|\overline b(x)|\le\delta\rho(x)$.
\end{lemma}

\begin{proof}
Let $f:\Sph^2\to[0,1]$ have integral $4\pi v$, and put
$z=\ip{u}{w}$ for a unit vector $w$. Set $t=1-2v$.
Since $(z-t)f\le(z-t)_+$,
\[
 \int_{\Sph^2}zf\dd\theta
 \le4\pi vt+2\pi\int_t^1(z-t)\dd z
 =4\pi v(1-v).
\]
Applying the same argument to $-z$ gives, for every $0<v\le1$,
\begin{equation}\label{eq:zfbound}
 \left|\int_{\Sph^2}zf\dd\theta\right|\le4\pi v\delta.
\end{equation}
Since $\int_{\Sph^2}z^2\dd\theta=\frac{4\pi}{3}$, applying the
first inequality of Lemma~\ref{lem:spherical-inequalities} to $f$ and $1-f$
gives
\begin{equation}\label{eq:z2fbound}
 \frac{v^2}{3}
 \le\frac{1}{4\pi v}\int_{\Sph^2}z^2f\dd\theta
 \le\frac{1-\delta^3}{3v}
 =\frac{1+\delta+\delta^2}{3}.
\end{equation}
When $\delta=0$, the upper bound follows directly from $f=1$
almost everywhere.

Fix a point $x$ where Lemma~\ref{lem:direction-transport} applies,
and let $w\in S_xM$. The density $j_x$ satisfies
$0\le j_x\le1$ and $\int_{S_xM}j_x\dd\theta_x=4\pi v$,
so the preceding estimates apply with $f=j_x$.

First take $\Psi(\xi)=\ip{\xi}{w}$ in
\eqref{eq:direction-transport}. Differentiating the Busemann average
and using \eqref{eq:zfbound} gives
\[
 \left|\ip{\nabla\overline b(x)}{w}\right|
 =\left|\frac{1}{4\pi v}\int_{S_xM}
              \ip{\xi}{w}j_x(\xi)\dd\theta_x(\xi)\right|
 \le\delta.
\]
Taking the supremum over unit $w$ proves
$|\nabla\overline b|\le\delta$ in \eqref{eq:busemann-average-bounds}.

Next take $\Psi(\xi)=\ip{\xi}{w}^2$. Then
\[
 \overline{\dd b\otimes\dd b}_x(w,w)
 =\int_{\mathcal R_p}\ip{\nabla b_u(x)}{w}^2\dd\nu_p(u)
 =\frac{1}{4\pi v}\int_{S_xM}
              \ip{\xi}{w}^2j_x(\xi)\dd\theta_x(\xi).
\]
By \eqref{eq:z2fbound}, this lies between $m$ and $D$ for every
unit $w$, which proves the tensor bounds in \eqref{eq:busemann-average-bounds}.

Finally, take $\Psi(\xi)=\ip{\xi}{w}_+$. The second inequality
of \eqref{eq:spherical-inequalities}, with $q=v$ and $f=j_x$, gives
\[
 2\int_{\mathcal R_p}\ip{\nabla b_u(x)}{w}_+\dd\nu_p(u)
 =\frac{2}{4\pi v}\int_{S_xM}
             \ip{\xi}{w}_+j_x(\xi)\dd\theta_x(\xi)
 \le\ell(v).
\]
This is \eqref{eq:busemann-positive-part}.
As in Lemma~\ref{hd:lem:busemann-average}, the weak gradient bound and
$\overline b(p)=0$ imply
$\Lip(\overline b)\le\delta$ and $|\overline b|\le\delta\rho$.
\end{proof}

\subsection{A common local-to-asymptotic comparison}
\label{sec:divergence}

In every dimension $n\ge3$, consider
\begin{equation}\label{eq:tensor}
 \mathsf T=\frac{1}{2}Sg-\Ric.
\end{equation}
For a unit vector $e_1$ completed to an orthonormal frame,
\[
 \mathsf T(e_1,e_1)=\sum_{2\le i<j\le n}K(e_i\wedge e_j)\ge0.
\]
This shows that $\mathsf T$ is nonnegative. Since $\Ric\ge0$, we also have $\mathsf T\le \frac{S}{2}g$. With the
divergence convention of \cite[Lemma~7.7, equation~(7.9)]{Lee97},
contracted Bianchi gives
\begin{equation}\label{eq:tensor-properties}
 \diver\mathsf T=0,\qquad
 \tr_g\mathsf T=\frac{n-2}{2}S,\qquad
 |\mathsf T(X,Y)|\le\frac{1}{2}S|X|\,|Y|.
\end{equation}
We record these identities in every dimension because they will also be
used in Section~\ref{sec:sectional}. The asymptotic limit in the
following proposition is specifically three dimensional.

\begin{proposition}[Local-to-asymptotic comparison]\label{prop:transfer}
Suppose $n\ge3$, $K\ge0$, and a globally Lipschitz, geodesically
convex function $F$ satisfies, distributionally,
\[
 \Lip(F)\le L_0,\qquad \Hess F\ge\kappa g\ \text{on }B_p(r),
 \qquad L_0,\kappa>0.
\]
Then, for every $R>r$,
\begin{equation}\label{gv:transfer-dimensional}
 (n-2)\kappa I_p(r)\le L_0\frac{I_p(R)}{R-r}.
\end{equation}
In dimension three, if $v>0$, put
$\Lambda=\lim_{t\to\infty}\frac{I_p(t)}{t}=8\pi(1-v)$. Then
\begin{equation}\label{eq:transfer-general}
 I_p(r)\le\frac{L_0}{\kappa}\Lambda
          =8\pi(1-v)\frac{L_0}{\kappa}.
\end{equation}
\end{proposition}

\begin{proof}
Let $\eta\in C_c^\infty(M)$ be nonnegative. Since $\mathsf T\ge0$,
contraction with $\mathsf T$ preserves the Hessian inequality:
on $B_p(r)$, the nonnegative tensor-valued measure
$\Hess F-\kappa g$ has nonnegative pairing with $\eta\mathsf T$.
Consequently,
\[
 \langle\Hess F,\eta\mathsf T\rangle
 \ge\kappa\int_{B_p(r)}\eta\tr_g\mathsf T\dd\mu_g
 =\frac{(n-2)\kappa}{2}\int_{B_p(r)}\eta S\dd\mu_g.
\]
Integration
by parts and \eqref{eq:tensor-properties}
give
\begin{equation}\label{eq:weak-ibp}
 \begin{aligned}
 \frac{(n-2)\kappa}{2}\int_{B_p(r)}\eta S\dd\mu_g
 &\le\langle\Hess F,\eta\mathsf T\rangle\\
 &=-\int_M\mathsf T(\nabla F,\nabla\eta)\dd\mu_g
 \le\frac{L_0}{2}\int_M S|\nabla\eta|\dd\mu_g.
 \end{aligned}
\end{equation}
Nonnegative smooth $W^{1,1}$
approximation with support in a fixed compact set allows Lipschitz
cutoffs. For
$\eta_R=\min\{1,\frac{(R-\rho)_+}{R-r}\}$ this gives
\begin{equation}\label{eq:cutoff}
 \frac{(n-2)\kappa}{2}I_p(r)
  \le\frac{L_0}{2(R-r)}\bigl(I_p(R)-I_p(r)\bigr)
  \le\frac{L_0}{2(R-r)}I_p(R).
\end{equation}
This proves \eqref{gv:transfer-dimensional}. For $n=3$,
\eqref{eq:asymptotic-input} gives
$\frac{I_p(R)}{R-r}\to8\pi(1-v)$ as $R\to\infty$, proving
\eqref{eq:transfer-general}.
\end{proof}

\subsection{Arbitrary positive asymptotic volume ratios}
\label{sec:convex-function}

We first construct a convex function valid for every $v>0$.

\begin{lemma}\label{lem:convex-function}
For every $p\in M^3$ and $r>0$, there is a globally Lipschitz,
geodesically convex function $F_r:M\to\R$ such that
\begin{equation}\label{eq:F-bounds}
 \Lip(F_r)\le\ell(v)\ \text{on }M,\qquad
 \Hess F_r\ge\frac{v^2}{3r}g\ \text{on }B_p(r).
\end{equation}
The Hessian inequality is distributional, and $0\le F_r\le2r$
on $B_p(r)$.
\end{lemma}

\begin{proof}
Define the convex nondecreasing function
\begin{equation}\label{eq:truncated-quadratic}
 \phi_r(s)=
 \begin{cases}
 0,&s\le-r,\\
 \frac{(s+r)^2}{2r},&-r\le s\le r,\\
 2s,&s\ge r,
 \end{cases}
\end{equation}
and set
\begin{equation}\label{eq:averaging}
 F_r(x)=\overline{\phi_r(b)}(x)
       =\int\phi_r(b_u(x))\dd\nu_p(u).
\end{equation}
Since $|b_u|\le\rho$ and $0\le\phi_r'\le2$, this function is
finite, convex, and globally $2$-Lipschitz, and $0\le F_r\le2r$
on $B_p(r)$. The distributional chain rule there gives
\begin{equation}\label{eq:chain-rule}
 \Hess\phi_r(b_u)
 =\phi_r'(b_u)\Hess b_u+\frac{1}{r}\dd b_u\otimes\dd b_u
 \ge\frac{1}{r}\dd b_u\otimes\dd b_u.
\end{equation}
Averaging and using \eqref{eq:busemann-average-bounds} proves the Hessian
bound. For almost every $x$ and every unit vector $w$, the
nonnegativity and upper bound of $\phi_r'$ also give
\[
 \ip{\nabla F_r}{w}
 \le2\int\ip{\nabla b_u}{w}_+\dd\nu_p(u)\le\ell(v),
\]
where the last step is \eqref{eq:busemann-positive-part}. The weak
gradient bound proves the Lipschitz estimate in \eqref{eq:F-bounds}.
\end{proof}

\begin{proof}[Proof of Theorem~\ref{thm:main}]
Apply Proposition~\ref{prop:transfer} to the function in
Lemma~\ref{lem:convex-function}, with $L_0=\ell(v)$ and
$\kappa=\frac{v^2}{3r}$. It follows that
\begin{equation}\label{eq:transfer}
 I_p(r)\le\frac{3\ell(v)}{v^2}\Lambda r
        =\frac{24\pi\ell(v)(1-v)}{v^2}r.
\end{equation}
By the definition of $\ell$,
\[
 I_p(r)\le
 \begin{cases}
 48\pi(1-v)^2v^{-2}r,&0<v\le \frac{1}{2},\\
 12\pi(1-v)v^{-3}r,&\frac{1}{2}<v\le1.
 \end{cases}
\]
Since $4v(1-v)\le1$, the first expression is also at most
$12\pi(1-v)v^{-3}r$. This proves \eqref{eq:main}.
\end{proof}

\subsection{An explicit quadratic remainder near full volume}
\label{sec:three-dimensional-refinement}

Near $v=1$, we replace $F_r$ by a convex function for which the
ratio $\frac{L_0}{\kappa}$ in Proposition~\ref{prop:transfer} is closer
to $r$. This gives the following numerical version of the
three-dimensional bound in Theorem~\ref{thm:near-maximal}.

\begin{theorem}\label{thm:first-order}
Let $(M^3,g)$ have $K\ge0$. If
$0\le1-v\le \frac{1}{100}$, where $v=V_{M^3}$, then
\begin{equation}\label{eq:first-order}
 I_p(r)\le\bigl[8\pi(1-v)+200\pi(1-v)^2\bigr]r
 \qquad(p\in M,\ r>0).
\end{equation}
The leading coefficient $8\pi$ is sharp and a quadratic remainder
is necessary.
\end{theorem}

For the construction below, assume $0<\delta=1-v\le \frac{1}{100}$ and
work at radius one. We use the averages and bounds already
established in \eqref{eq:weighted-average} and
Lemma~\ref{lem:busemann-average-bounds}.

\begin{lemma}\label{lem:fo-potential}
There is a globally Lipschitz,
geodesically convex function $F:M\to\R$ such that
\begin{equation}\label{eq:fo-potential}
 \Lip(F)\le\frac{1+10\delta}{3}\quad\text{on }M,\qquad
 \Hess F\ge\frac{1-10\delta}{3}g
                   \quad\text{on }B_p(1),
\end{equation}
in the distributional sense.
\end{lemma}

\begin{proof}
Define
\[
 q=\sqrt{\overline{b^2}},\qquad h=\overline b,
 \qquad A=\overline{\dd b\otimes\dd b}.
\]
Since $|b_u(x)-b_u(y)|\le d(x,y)$, the triangle inequality for the weighted $L^2$ integral shows
that $q$ is $1$-Lipschitz.

For $x\ne p$, let $w$ be the initial unit direction of a
minimizing segment from $p$ to $x$. The hinge form of Toponogov
comparison \cite[Theorem~12.2.2]{Petersen16}, followed by the limit
along the ray as in Lemma~\ref{hd:lem:coords-rough}, gives
\[
 b_u(x)\ge\rho(x)\ip{u}{w}.
\]
This comparison uses only $K\ge0$, with no restriction on $\delta$.
Put $z_u=\ip{u}{w}$ and $e_u=b_u(x)-\rho z_u\ge0$.
Since $|z_u|\le1$,
\[
 b_u(x)^2=(\rho z_u+e_u)^2\ge\rho^2z_u^2-2\rho e_u.
\]
Equations \eqref{eq:zfbound} and \eqref{eq:z2fbound} give
$|\overline z|\le\delta$ and $\overline{z^2}\ge m$.
Together with $|h(x)|\le\delta\rho$, this yields
\begin{equation}\label{eq:q-lower}
 \overline e=h(x)-\rho\overline z\le2\delta\rho,
 \qquad
 q^2\ge\rho^2\overline{z^2}-2\rho\overline e\ge(m-4\delta)\rho^2\ge\frac{1}{4}\rho^2.
\end{equation}
The last inequality follows from $\delta\le \frac{1}{100}$.
In particular, $q>0$ away from $p$ and $\frac{|b_u|}{q}\le2$.

On $M\setminus\{p\}$, differentiation and Cauchy--Schwarz give
\begin{equation}\label{eq:q-gradient}
	\dd q=\frac{\overline{b\dd b}}{q},
	\qquad \dd q\otimes\dd q\le A\le Dg.
\end{equation}
It follows that $\Lip(q)\le\sqrt D$.

Set
\begin{equation}\label{eq:corrected-norm}
	G=q+2h,\qquad a=\sqrt D+2\delta.
\end{equation}
All Hessian identities below are understood distributionally on
$M\setminus\{p\}$. On each compact subset, $q$ is bounded away
from zero, the functions $b_u$ are uniformly Lipschitz, and their
nonnegative Hessian measures have uniformly bounded mass.
The distributional product and chain rules therefore apply to
$\overline{b^2}$ and its square root.
Differentiating $q^2=\overline{b^2}$ gives
\[
 2q\,\Hess q+2\,\dd q\otimes\dd q
 =2A+2\,\overline{b\,\Hess b}.
\]
Together with $\Hess h=\overline{\Hess b}$, this yields
\begin{equation}\label{eq:G-hessian}
 \Hess G
 =\frac{A-\dd q\otimes\dd q}{q}
  +\overline{\left(2+\frac{b}{q}\right)\Hess b}.
\end{equation}
By \eqref{eq:q-gradient}, $\Hess b_u\ge0$, and $\frac{|b_u|}{q}\le2$,
both terms are nonnegative. Thus $\Hess G\ge0$ away from $p$.
Moreover,
\[
 G\ge(\frac{1}{2}-2\delta)\rho\ge0,\qquad G(p)=0,
 \qquad\Lip(G)\le a.
\]
Hence $G$ is globally geodesically convex.

For $s\ge0$, let
\[
 H_a(s)=
 \begin{cases}
 \frac{s^2}{2},&0\le s\le a,\\
 as-\frac{a^2}{2},&s\ge a,
 \end{cases}
 \qquad F=H_a(G).
\]
Then $F$ is globally convex and $\Lip(F)\le a^2$.
On $B_p(1)$, we have $G\le a\rho<a$, so $F=\frac{G^2}{2}$.
Using \eqref{eq:G-hessian}, $\frac{h}{q}\ge-2\delta$, and
$A-\dd q\otimes\dd q\ge0$, we obtain on $B_p(1)\setminus\{p\}$
\begin{align}
 \Hess F&=G\,\Hess G+\dd G\otimes\dd G\notag\\
 &=\frac{q+2h}{q}(A-\dd q\otimes\dd q)
       +(\dd q+2\dd h)\otimes(\dd q+2\dd h)\notag\\
 &\quad+G\,\overline{\left(2+\frac{b}{q}\right)\Hess b}\notag\\
 &\ge A+\frac{2h}{q}(A-\dd q\otimes\dd q)
       +2(\dd q\otimes\dd h+\dd h\otimes\dd q)
       +4\dd h\otimes\dd h\notag\\
 &\ge(1-4\delta)A+4\delta\,\dd q\otimes\dd q
       +2(\dd q\otimes\dd h+\dd h\otimes\dd q)
       +4\dd h\otimes\dd h.
 \label{eq:quadratic-hessian}
\end{align}
Since $|\dd h|\le\delta$ and $0<\delta<\frac{1}{4}$, we obtain
\[
 \begin{aligned}
 &4\delta\,\dd q\otimes\dd q
       +2(\dd q\otimes\dd h+\dd h\otimes\dd q)
       +4\dd h\otimes\dd h\\
 &\quad=4\delta
       \left(\dd q+\frac{\dd h}{2\delta}\right)
       \otimes\left(\dd q+\frac{\dd h}{2\delta}\right)
       +\left(4-\frac{1}{\delta}\right)\dd h\otimes\dd h\\
 &\quad\ge-(\delta-4\delta^2)g,
 \end{aligned}
\]
 Consequently,
\[
 \Hess F\ge\bigl((1-4\delta)m-\delta+4\delta^2\bigr)g
 =\frac{1-9\delta+21\delta^2-4\delta^3}{3}g
 \ge\frac{1-10\delta}{3}g.
\]
Since $F$ is globally convex, its Hessian measure has no
negative part at $p$. The lower
bound holds on all of $B_p(1)$ in the distributional sense.

Finally, $\delta\le \frac{1}{100}$ gives $\sqrt D\le \frac{7}{12}$. Hence
\[
 \begin{aligned}
 \Lip(F)\le a^2
 &=D+4\sqrt D\,\delta+4\delta^2\\
 &\le\frac{1+8\delta+13\delta^2}{3}
 \le\frac{1+10\delta}{3}.
 \end{aligned}
\]
This proves \eqref{eq:fo-potential}.
\end{proof}

\begin{proof}[Proof of Theorem~\ref{thm:first-order}]
For $0<\delta\le \frac{1}{100}$, apply Proposition~\ref{prop:transfer}
with $r=1$ and the function in Lemma~\ref{lem:fo-potential}.
Using \eqref{eq:asymptotic-input}, we get
\begin{equation}\label{eq:quadratic-upper}
 \begin{aligned}
 I_p(1)
 &\le8\pi\delta\frac{1+10\delta}{1-10\delta}\\
 &=8\pi\delta+\frac{160\pi\delta^2}{1-10\delta}
 \le8\pi\delta+200\pi\delta^2.
 \end{aligned}
\end{equation}
Rescaling by $r^{-2}$ gives \eqref{eq:first-order} for every
$r>0$. If $\delta=0$, Bishop--Gromov rigidity gives Euclidean
space and the assertion is immediate.
\end{proof}

\section{Arbitrary positive volume ratios in dimensions four and five}
\label{sec:sectional}

Throughout this section $K\ge0$, $v>0$, and the center $p$ is arbitrary.
We prove Theorem~\ref{thm:explicit45}. The convex average from
Section~\ref{sec:three-dimensional} still gives a local Hessian
bound. We use the four-dimensional Gauss--Bonnet formulas to
estimate the integral of scalar curvature. We retain the tensor $\mathsf T$ and its dimensional
identities \eqref{eq:tensor}--\eqref{eq:tensor-properties}.

\subsection{Spherical inequalities and a convex function}

We first establish two elementary inequalities on the unit sphere.
Recall that $s_n=|\Sph^{n-1}|=n\omega_n$, and put
\begin{equation}\label{gv:spherical-constant}
 c_n=\frac{\Gamma\left(\frac{n}{2}\right)}
          {\sqrt\pi\,\Gamma\left(\frac{n-1}{2}\right)}
     =\frac{|\Sph^{n-2}|}{|\Sph^{n-1}|},
 \qquad m_n(v)=\frac{v^2}{12c_n^2}.
\end{equation}

\begin{lemma}[Spherical inequalities]\label{gv:spherical-inequalities}
Suppose $n\ge3$ and $f:\Sph^{n-1}\to[0,1]$ is measurable, with
\[
 \int_{\Sph^{n-1}}f\dd\theta=s_nv,\qquad 0<v\le1.
\]
Here $\dd\theta$ denotes spherical area. Then for every unit vector $w$,
\begin{equation}\label{gv:quadratic-spherical-bound}
 \frac{1}{s_nv}\int_{\Sph^{n-1}}
             \ip{\xi}{w}^2f(\xi)\dd\theta(\xi)\ge m_n(v).
\end{equation}
If $B$ is self-adjoint, $0\le B\le\operatorname{Id}$, and
$\tr B=n-2$, then for every $q\in\R^n$,
\begin{equation}\label{gv:shifted-spherical-bound}
 \frac{1}{s_nv}\int_{\Sph^{n-1}}
       \ip{B(q+\xi)}{q+\xi}f(\xi)\dd\theta(\xi)
       \ge\frac{n-2}{n}v^{\frac{2}{n-2}}.
\end{equation}
\end{lemma}

\begin{proof}
Write $z=\ip{\xi}{w}$.
The parametrization
$\xi=zw+\sqrt{1-z^2}\,\eta$, with
$\eta\in\Sph^{n-2}\subset w^\perp$, gives
\[
 \dd\theta=(1-z^2)^{\frac{n-3}{2}}\dd z\dd\sigma_{n-2}(\eta).
\]
After division by $s_n$, the coordinate density is therefore
$c_n(1-z^2)^{\frac{n-3}{2}}\le c_n$. Its integral is one,
so $2c_n\ge1$ and $\frac{v}{2c_n}\le1$. Since $0\le f\le1$,
\[
 \frac{1}{s_nv}\int_{\{|z|\le t\}}f\dd\theta
       \le\frac{2c_nt}{v}.
\]
Using $z^2=\int_0^{|z|}2t\dd t$ and Fubini gives
\[
 \frac{1}{s_nv}\int_{\Sph^{n-1}}z^2f\dd\theta
 \ge\int_0^{\frac{v}{2c_n}}
             2t\left(1-\frac{2c_nt}{v}\right)\dd t
 =\frac{v^2}{12c_n^2}.
\]

For the second assertion, let $P$ be an orthogonal projection onto
an $(n-2)$-dimensional subspace. 
Use the parametrization
$\Phi(y,\alpha)=(y,\sqrt{1-|y|^2}(\cos\alpha,\sin\alpha))$,
where $|y|<1$ and $0\le\alpha<2\pi$. Its induced metric is block
diagonal: the $y$ block is
$\operatorname{Id}+\frac{yy^{\mathsf T}}{1-|y|^2}$ and the
$\alpha$ block is $1-|y|^2$. Thus its determinant is
\[
 \det\left(\operatorname{Id}+\frac{yy^{\mathsf T}}{1-|y|^2}\right)
       (1-|y|^2)
 =\frac{1}{1-|y|^2}(1-|y|^2)=1.
\]
Hence $\dd\theta=\dd y\dd\alpha$, so the projected area has
constant density. 
\[
\frac{1}{s_n}\int_{\Sph^{n-1}}\psi(P\xi)\dd\theta
=\frac{1}{\omega_{n-2}}\int_{B^{n-2}(1)}\psi(y)\dd y.
\] 
Consequently, for every $t\ge0$,
\[
 \frac{1}{s_nv}\int_{\{|P\xi+Pq|\le t\}}f\dd\theta
 \le\frac{\Vol_{n-2}\bigl(B^{n-2}(1)\cap B_{-Pq}(t)\bigr)}
          {v\omega_{n-2}}
 \le\frac{t^{n-2}}{v}.
\]
The same integral identity for the square now gives
\[
 \begin{aligned}
 \frac{1}{s_nv}\int_{\Sph^{n-1}}|P(q+\xi)|^2f\dd\theta
 &\ge\int_0^{v^{\frac{1}{n-2}}}
           2t\left(1-\frac{t^{n-2}}{v}\right)\dd t\\
 &=\frac{n-2}{n}v^{\frac{2}{n-2}}.
 \end{aligned}
\]
Every $B$ in the statement is a convex combination of orthogonal
projections of rank $n-2$. To see this, diagonalize $B$; its
eigenvalue vector belongs to the polytope
$0\le b_i\le1$, $b_1+\cdots+b_n=n-2$, whose vertices have exactly
$n-2$ entries equal to one and the others zero.
Since the integrand in \eqref{gv:shifted-spherical-bound} is linear in $B$,
the estimate for projections proves the general case.
\end{proof}

\begin{lemma}\label{gv:lem:convex-function}
For every $n\ge3$, $p\in M$, and $r>0$, there is a globally
Lipschitz, geodesically convex function $F_r:M\to\R$ such that
\begin{equation}\label{gv:F-bounds}
 \Lip(F_r)\le2,\qquad
 \Hess F_r\ge\frac{m_n(v)}{r} g\ \text{on }B_p(r),\qquad
 0\le F_r\le2r\ \text{on }B_p(r).
\end{equation}
The Hessian inequality is distributional.
\end{lemma}

\begin{proof}
Use the same function $\phi_r$ as in \eqref{eq:truncated-quadratic}
and the fixed probability measure $\nu_p$ to set
\[
 F_r(x)=\int\phi_r(b_u(x))\dd\nu_p(u).
\]
As in Lemma~\ref{lem:convex-function}, convexity,
$0\le\phi_r'\le2$, and $|b_u|\le\rho$ give the Lipschitz and
value bounds. The chain rule \eqref{eq:chain-rule} is independent
of the dimension. At almost every $x$, use
\eqref{gv:quadratic-spherical-bound} on $S_xM$ with $f=j_x$, whose integral is
$s_nv$. Direction transfer then gives
\[
 \int_{\mathcal R_p}\ip{\nabla b_u(x)}{w}^2\dd\nu_p(u)
 =\frac{1}{s_nv}\int_{S_xM}\ip{\xi}{w}^2j_x(\xi)\dd\theta_x(\xi)
 \ge m_n(v)
\]
for every unit $w$. Averaging the chain rule proves the stated
Hessian bound.
\end{proof}

\subsection{Proper strictly convex levels}
For a bounded function $f$ on a nonempty set $U$, write
\[
 \operatorname{osc}_U f=\sup_U f-\inf_U f
\]
for its oscillation on $U$.

\begin{lemma}[Proper strictly convex approximation]\label{gv:lem:proper-approx}
For every $\lambda>0$ and ray direction $u$, put
$f_u=\lambda F_r+b_u$. There exist smooth proper functions
$f_{u,k}:M\to\R$, $k=1,2,\ldots$, with $\Hess f_{u,k}>0$
everywhere, such that
\[
 \sup_{B_p(r)}|f_{u,k}-f_u|\le\frac{r}{2k},\qquad
 \nabla f_{u,k}\longrightarrow\nabla f_u
       \quad\text{almost everywhere on }B_p(r).
\]
Their Hessians satisfy
\[
 \Hess f_{u,k}\ge
       \left(\frac{\lambda m_n(v)}{r}-\frac{1}{kr}\right)g
       \quad\text{on }B_p\left(\frac{kr}{k+1}\right).
\]
In particular, for each fixed $0<r'<r$, this bound holds uniformly
on $B_p(r')$ for all sufficiently large $k$. Moreover,
\begin{equation}\label{gv:eq:proper-bounds}
 \begin{aligned}
 \Lip f_{u,k}&\le2\lambda+1+\frac{1}{k},\\
 \operatorname{osc}_{B_p(r)}f_{u,k}
       &\le\left(2\lambda+2+\frac{1}{k}\right)r.
 \end{aligned}
\end{equation}
Every nonempty regular sublevel of $f_{u,k}$ is a compact smooth ball.
\end{lemma}

\begin{proof}
Let $\rho=d(p,\cdot)$ and $E_p=\sup_{w\in\mathcal R_p}b_w$.
This is convex and $1$-Lipschitz, and $\frac{E_p}{\rho}\to1$ at
infinity; see \cite[Lemma~7(ii)]{ChowLu04}. The convex, $(2\lambda+1)$-Lipschitz
function
\[
 \widetilde f_u=\max\{\lambda F_r+b_u,E_p-5r\}
\]
tends to $+\infty$ at infinity and equals $f_u$ on $B_p(2r)$:
 the first element is at least $-2r$ and the second at most $-3r$.

Set $R_j=2^jr$ and
\[
 G=\sum_{j=0}^{\infty}2^{-j-1}(F_{R_j}-F_{R_j}(p)).
\]
This series converges locally uniformly, $\Lip G\le2$, and
$|G|\le2\rho$. Each summand is convex. Given a compact set $K\subset M$,
choose $j$ so large that $K\subset B_p(R_j)$.
Lemma~\ref{gv:lem:convex-function} gives
\[
 \Hess G\ge2^{-j-1}\Hess F_{R_j}\ge\frac{m_n(v)}{2^{j+1}R_j}g,\qquad \text{on}\; B_p(R_j) .
\] 
In particular, $G$ has a strictly positive distributional Hessian
lower bound on every compact set.

For $k\ge1$, put
\[
 H_{u,k}=\widetilde f_u+\frac{1}{8k}G.
\]
It is strictly convex, has Lipschitz constant at most
$2\lambda+1+\frac{1}{4k}$, and satisfies
\[
 H_{u,k}\ge E_p-5r-\frac{\rho}{4k}\longrightarrow+\infty
       \quad\text{as }\rho\to\infty.
\]
Apply Greene--Wu fine smoothing \cite[Propositions 2.2--2.3]{GW79}
with a positive continuous Hessian lower bound. Choose the smooth
approximation $f_{u,k}$ so that
\[
 |f_{u,k}-H_{u,k}|\le\frac{r}{4k}\quad\text{on }M,
 \qquad \Lip f_{u,k}\le2\lambda+1+\frac{1}{k},
\]
while preserving a positive definite Hessian everywhere.
Since $\Hess H_{u,k}\ge\frac{\lambda m_n(v)}{r}g$ on $B_p(r)$,
the smoothing can also be chosen sufficiently fine on the compact
ball $\overline{B_p(\frac{kr}{k+1})}$ to lose at most
$\frac{1}{kr}g$ in this lower bound. The global bound on the
approximation error preserves convergence to $+\infty$ at infinity.

On $B_p(r)$, $\widetilde f_u=f_u$ and $|G|\le2r$, so
\[
 |f_{u,k}-f_u|\le\frac{r}{4k}+\frac{2r}{8k}=\frac{r}{2k}.
\]
Since $0\le F_r\le2r$ and $-r\le b_u\le r$ there,
$\operatorname{osc}_{B_p(r)}f_u\le(2\lambda+2)r$; the uniform
error gives \eqref{gv:eq:proper-bounds}. Local uniform convergence
of convex functions implies convergence of gradients at
differentiability points, by upper and lower geodesic difference
quotients. Finally, each $f_{u,k}$ tends to $+\infty$ at infinity
and has positive definite Hessian, so it has one nondegenerate
minimum and no other critical point. Morse's lemma and gradient
flow between compact levels show that its nonempty regular
sublevels are compact smooth balls.
\end{proof}

\subsection{The boundary estimate and coarea}

\begin{lemma}[Total Gauss--Kronecker curvature]\label{gv:lem:det-boundary}
Let $n\ge3$, and let $D\subset M^n$ be a compact domain with smooth
boundary and nonnegative outward shape operator $A$. Then
\begin{equation}\label{gv:eq:det-boundary}
 \int_{\partial D}\det A\dd A\ge s_nv.
\end{equation}
\end{lemma}

\begin{proof}
This is also the hypersurface case of \cite[Theorem~1.2]{PanYi26}.
We include a short comparison proof. Along an outward normal segment
minimizing distance to $D$, use parallel transport to identify the
normal Jacobi matrix with an endomorphism of the initial tangent
space. It satisfies $\mathcal J(0)=\operatorname{Id}$ and
$\mathcal J'(0)=A$. Before the cut time, matrix Riccati comparison
\cite[Proposition~2.3]{Eschenburg87}, with the flat model
$A(\operatorname{Id}+tA)^{-1}$, gives
\[
 \mathcal J'(t)\mathcal J(t)^{-1}
       \le A(\operatorname{Id}+tA)^{-1}.
\]
Taking traces and integrating from zero yields
$0<\det\mathcal J(t)\le\det(\operatorname{Id}+tA)$.
Integrating over the minimizing normal segments parametrizes the
exterior up to a set of volume zero. Since $A\ge0$, the flat upper
bound remains nonnegative after the cut time, so
\[
 \Vol\{x:d(x,D)<R\}
 \le\Vol D+\int_{\partial D}\int_0^R
                 \det(\operatorname{Id}+tA)\dd t\dd A.
\]
Because $D$ is bounded, the left side divided by $R^n$ tends to
$\omega_nv$. The polynomial $\det(\operatorname{Id}+tA)$ has
coefficient $\det A$ at $t^{n-1}$. Divide the inequality by $R^n$
and let $R\to\infty$ to obtain
\[
 \omega_nv\le\frac{1}{n}\int_{\partial D}\det A\dd A,
\]
which proves the claim since $s_n=n\omega_n$.
\end{proof}

Let $D\subset M^n$ be a compact smooth domain diffeomorphic to a ball,
with nonnegative outward shape operator $A$. In an orthonormal
eigenframe of $A$, write its eigenvalues as $\lambda_1,\ldots,\lambda_{n-1}$
and denote ambient tangent sectional curvatures by $K_{ij}$. Put
\begin{equation}\label{gv:eq:Qn}
 Q_n(A)=\sum_{1\le i<j\le n-1}K_{ij}
                    \prod_{\ell\notin\{i,j\}}\lambda_\ell\ge0.
\end{equation}
This expression is independent of the choice of principal frame:
rotations within an eigenspace leave the corresponding curvature
contractions unchanged.

\begin{lemma}[Boundary curvature deficit]\label{gv:lem:boundary}
For $n=4,5$,
\begin{equation}\label{gv:eq:boundary-deficit}
 \int_{\partial D}Q_n(A)\dd A\le(n-2)s_n(1-v).
\end{equation}
\end{lemma}

\begin{proof}
For $n=4$, put $\Sigma=\partial D$, let
$\mathrm{II}=g_\Sigma(A\cdot,\cdot)$, and write $H=\tr A$.
With the outward convention for $A$, the four-dimensional
Chern--Gauss--Bonnet formula \cite{Chern45}, multiplied by
$2\pi^2$, has interior density
$\frac{1}{16}(\norm{\operatorname{Rm}}^2-4\norm{\Ric}^2+S^2)$
and boundary density
\[
 \frac{1}{2}\left(\frac{1}{2}H S_\Sigma
          -\langle\Ric_\Sigma,\mathrm{II}\rangle_{g_\Sigma}
          -\det A\right).
\]
To express the latter using $Q_4$, choose a principal frame on
$\Sigma$. The Gauss equation gives
$K^\Sigma_{ij}=K_{ij}+\lambda_i\lambda_j$, and hence
\[
 \begin{aligned}
 \frac{1}{2}H S_\Sigma
       -\langle\Ric_\Sigma,\mathrm{II}\rangle_{g_\Sigma}
 &=\lambda_1K^\Sigma_{23}
       +\lambda_2K^\Sigma_{13}+\lambda_3K^\Sigma_{12}\\
 &=Q_4(A)+3\det A.
 \end{aligned}
\]
Thus the boundary density is $\det A+\frac{1}{2}Q_4(A)$, and the
formula becomes
\begin{equation}\label{gv:eq:CGB-boundary}
 2\pi^2\chi(D)
   =\int_{\partial D}\left(\det A+\frac{1}{2}Q_4(A)\right)\dd A
        +\frac{1}{16}\int_D
           (\norm{\operatorname{Rm}}^2-4\norm{\Ric}^2+S^2)\dd\mu_g.
\end{equation}
The interior density is nonnegative under $K\ge0$
\cite[Section~4]{BishopGoldberg64}. One can see this directly at
each point by choosing a four-dimensional orthonormal frame
diagonalizing $\Ric$. Writing $K_{ij}$ for its sectional curvatures,
$\norm{\operatorname{Rm}}^2\ge4\sum_{i<j}K_{ij}^2$, and expansion gives
\[
 \norm{\operatorname{Rm}}^2-4\norm{\Ric}^2+S^2
 \ge8(K_{12}K_{34}+K_{13}K_{24}+K_{14}K_{23})\ge0.
\]
Since $\chi(D)=1$ and $s_4=2\pi^2$, subtracting
\eqref{gv:eq:det-boundary} from \eqref{gv:eq:CGB-boundary} proves
$\int Q_4\le4\pi^2(1-v)$.

For $n=5$, $\Sigma=\partial D$ is a four-sphere. Let
$\operatorname{Rm}^\top$ denote the restriction of the ambient
Riemann curvature tensor to $T\Sigma$. Its contractions
$\Ric^\top$ and $S^\top$ are taken only over the four tangent
directions: in a tangent orthonormal frame $e_1,\ldots,e_4$,
\[
 \Ric^\top(X,Y)=R^M(e_i,X,e_i,Y),\qquad
 S^\top=\Ric^\top(e_i,e_i).
\]
The Gauss equation and expansion of the curvature contractions give
\begin{equation}\label{gv:eq:CGB-mixed}
 \begin{aligned}
 &\norm{\operatorname{Rm}_\Sigma}^2-4\norm{\Ric_\Sigma}^2+S_\Sigma^2\\
 &\qquad=\norm{\operatorname{Rm}^\top}^2
       -4\norm{\Ric^\top}^2+(S^\top)^2+8Q_5(A)+24\det A.
 \end{aligned}
\end{equation}
For the mixed coefficient, put $H=\tr A$ and
$\sigma_2=\frac{1}{2}((\tr A)^2-\tr(A^2))$.
For a fixed pair $i<j$, the coefficient of $K_{ij}$ in the
mixed terms is
\[
 8\bigl[\lambda_i\lambda_j-\lambda_i(H-\lambda_i)
       -\lambda_j(H-\lambda_j)+\sigma_2\bigr]
       =8\lambda_k\lambda_l,
 \quad\{i,j,k,l\}=\{1,2,3,4\}.
\]
The terms containing only principal curvatures sum to $24\det A$.
The preceding pointwise calculation, now applied to
$\operatorname{Rm}^\top$ and its tangent contractions, shows that
the first three terms on the right of \eqref{gv:eq:CGB-mixed}
have nonnegative sum, since all ambient tangent sectional
curvatures are nonnegative. The four-dimensional Gauss--Bonnet
formula on the closed Riemannian manifold $(\Sigma,g_\Sigma)$ gives
\[
 \frac{1}{24}\int_\Sigma
       (\norm{\operatorname{Rm}_\Sigma}^2
           -4\norm{\Ric_\Sigma}^2+S_\Sigma^2)\dd A
 =\frac{4\pi^2}{3}\chi(\Sigma)=\frac{8\pi^2}{3}=s_5.
\]
Combining this with \eqref{gv:eq:CGB-mixed} and
\eqref{gv:eq:det-boundary} gives
$\int Q_5\le3s_5(1-v)=8\pi^2(1-v)$.
\end{proof}

\begin{lemma}[Coarea with unnormalized gradients]\label{gv:lem:coarea}
Let $n=4$ or $5$. Suppose $f$ is smooth, satisfies
$f(x)\to+\infty$ as $d(p,x)\to\infty$, has positive definite
Hessian everywhere, $\Lip f\le L$, and
$\Hess f\ge\kappa g$ on $B_p(r)$, with $\kappa>0$. Then
\begin{equation}\label{gv:eq:coarea}
 \int_{B_p(r)}2\mathsf T(\nabla f,\nabla f)\dd\mu_g
 \le2(n-2)s_n(1-v)\kappa^{3-n}L^{n-2}
              \operatorname{osc}_{B_p(r)}f.
\end{equation}
\end{lemma}

\begin{proof}
At a regular point put $w=\norm{\nabla f}$ and $\nu=\frac{\nabla f}{w}$.
On $\Sigma_t=\{f=t\}$ the principal curvatures at points of
$B_p(r)$ are at least $\frac{\kappa}{w}$. Consequently, with
$S_{\rm tan}=S-2\Ric(\nu,\nu)=2\mathsf T(\nu,\nu)=2\sum_{i<j}K_{ij}$,
\[
 Q_n(A)\ge\frac{1}{2}(\frac{\kappa}{w})^{n-3}S_{\rm tan},\qquad
 wS_{\rm tan}\le2\kappa^{3-n}L^{n-2}Q_n(A).
\]
The exhaustion condition and positive Hessian give a unique
nondegenerate minimum and no other critical points. Thus every
nonempty regular sublevel of $f$ is a compact smooth ball, as in
Lemma~\ref{gv:lem:proper-approx}. Coarea,
nonnegativity of $Q_n$, and Lemma~\ref{gv:lem:boundary} give
\begin{align*}
 \int_{B_p(r)}2\mathsf T(\nabla f,\nabla f)\dd\mu_g
 &=\int_{\inf_{B_p(r)}f}^{\sup_{B_p(r)}f}
       \int_{\Sigma_t\cap B_p(r)}wS_{\rm tan}\dd A\dd t\\
 &\le2\kappa^{3-n}L^{n-2}
       \int_{\inf_{B_p(r)}f}^{\sup_{B_p(r)}f}
                  \int_{\Sigma_t}Q_n(A)\dd A\dd t.
\end{align*}
This is \eqref{gv:eq:coarea}. The unique minimum has zero volume.
\end{proof}

\begin{proof}[Proof of Theorem~\ref{thm:explicit45}]
For $v<1$ apply Lemma~\ref{gv:lem:coarea} to the smooth approximations
of $f_u=\lambda F_r+b_u$ from Lemma~\ref{gv:lem:proper-approx}, first
on $B_p(r')$ with $r'<r$ and $k$ sufficiently large. Let $k\to\infty$
and then let $r'\uparrow r$. Gradient convergence, uniform gradient
bounds and dominated convergence on compact balls yield
\begin{align}\label{gv:eq:directional-energy}
 &\int_{B_p(r)}2\mathsf T(\lambda\nabla F_r+\nabla b_u,
                         \lambda\nabla F_r+\nabla b_u)\dd\mu_g\notag\\
 &\qquad\le
 2(n-2)s_n(\lambda m_n(v))^{3-n}
        (2\lambda+1)^{n-2}(2\lambda+2)(1-v)r^{n-2}.
\end{align}
At almost every point $x$ where $S>0$, apply
\eqref{gv:shifted-spherical-bound} on $S_xM$ with $f=j_x$,
$B=\frac{2\mathsf T}{S}$ (viewed as an endomorphism using $g$), and
$q=\lambda\nabla F_r$. The bounds \eqref{eq:tensor-properties}
give $0\le B\le\operatorname{Id}$ and $\tr B=n-2$.
Direction transfer then gives
\[
 \int2\mathsf T(q+\nabla b_u,q+\nabla b_u)\dd\nu_p(u)
       \ge\frac{n-2}{n}v^{\frac{2}{n-2}}S.
\]
At points where $S=0$ this is immediate. Average
\eqref{gv:eq:directional-energy} in $u$ and use Fubini to obtain
\begin{equation}\label{gv:eq:45-parameter}
 I_p(r)\le
 \frac{2ns_n(2\lambda+1)^{n-2}(2\lambda+2)}{\lambda^{n-3}m_n(v)^{n-3}}
 v^{-\frac{2}{n-2}}(1-v)r^{n-2}.
\end{equation}
We use the fixed choices in the following table:
\[
\begin{array}{c|c|c|c|c}
 n&\lambda&c_n&m_n(v)&\text{coefficient and power in \eqref{gv:eq:45-parameter}}\\ \hline
 4&\frac{1}{4}&\frac{2}{\pi}&\frac{\pi^2v^2}{48}&17280\,v^{-3}\\[2pt]
 5&\frac{1}{2}&\frac{3}{4}&\frac{4v^2}{27}&116640\pi^2\,v^{-\frac{14}{3}}
\end{array}
\]
This proves \eqref{eq:explicit4}--\eqref{eq:explicit5}. For $v=1$
volume rigidity gives $S=0$.
\end{proof}

\begin{remark}
These constants are not optimized, and these estimates do not recover
the sharp coefficient as $v\uparrow1$. In that regime one may take
the minimum with Theorem~\ref{thm:near-maximal}. The dimensional restriction
here belongs to the boundary Gauss--Bonnet step: the argument only
uses the four-dimensional Euler polynomial and asserts no analogous
sign for higher Euler polynomials under sectional curvature alone.
\end{remark}

\section{Balls centered at a pole}
\label{sec:pole}

Throughout this section, $p$ is a pole and $\rho=d(p,\cdot)$.
We write all dimensional constants in terms of the ambient dimension
$n$; the distance spheres have dimension $n-1$.
The first two subsections use only $n\ge3$ and $\Ric\ge0$.
We then estimate the intrinsic scalar curvature integral of each
distance sphere, using Gauss--Bonnet when $n=3$ and $K\ge0$
when $n\ge4$. This is the only dimension-dependent input;
the remaining argument proves both cases of Theorem~\ref{thm:pole}
at once.

\subsection{The polar volume density in every dimension}
\label{sec:polar}

In polar coordinates $(t,\theta)$ at $p$, write
\begin{equation}\label{eq:polar-volume}
 \dd\mu_g=J(t,\theta)\dd t\dd\theta,
 \qquad j(t,\theta)=J(t,\theta)^{\frac{1}{n-1}},
 \qquad \mathcal A(t)=\Vol(\partial B_p(t)),
\end{equation}
where $\dd\theta$ is the standard area element on $\Sph^{n-1}$.
Let $A_t$ be the outward shape operator of $\Sigma_t=\partial B_p(t)$,
with principal curvatures $\lambda_1,\ldots,\lambda_{n-1}$ and
$H=\tr A_t$. Primes denote derivatives in $t$ with $\theta$ fixed.
The Riccati equation gives
\begin{equation}\label{eq:riccati}
 \frac{J'}{J}=H,\qquad
 H'=-\norm{A_t}^2-\Ric(\nabla\rho,\nabla\rho).
\end{equation}
For $n=3$, the resulting estimate $0\le H\le\frac{2}{t}$ is
\cite[Lemma~3.1]{Xu24}. The concavity argument applies in
every dimension as follows.

\begin{lemma}\label{lem:polar-concavity}
For every $t>0$ and $\theta\in\Sph^{n-1}$,
\begin{equation}\label{eq:j-bounds}
 j''\le0,\qquad 0\le j'\le1,\qquad
 0<j(t,\theta)\le t,\qquad 0\le H\le\frac{n-1}{t}.
\end{equation}
\end{lemma}

\begin{proof}
From \eqref{eq:riccati},
\[
 \frac{j''}{j}
 =-\frac{1}{n-1}\left(\Ric(\nabla\rho,\nabla\rho)
                 +\norm{A_t}^2-\frac{H^2}{n-1}\right)\le0.
\]
Normal coordinates give $j(0,\theta)=0$ and $j'(0,\theta)=1$.
Since $p$ is a pole, $j(t,\theta)>0$ for every $t>0$.
Concavity therefore implies $j'\ge0$; otherwise $j$ would become
negative in finite time. Thus $0\le j'\le1$ and $j\le t$.
Finally, concavity gives $j(t,\theta)\ge tj'(t,\theta)$, so
$0\le H=\frac{(n-1)j'}{j}\le \frac{n-1}{t}$.
\end{proof}

\begin{lemma}\label{lem:polar-integral}
For every $r>0$,
\begin{equation}\label{eq:polar-integral}
 \begin{split}
 I_p(r)\le{}&\int_0^r\int_{\Sigma_t}S_{\Sigma_t}\dd\sigma_t\dd t\vphantom{\int_{\Sph^{n-1}}}\\
 &+(n-1)(n-2)\int_{\Sph^{n-1}}\int_0^r
             j^{n-3}(j')^2\dd t\dd\theta-2\mathcal A'(r),
 \end{split}
\end{equation}
where $S_{\Sigma_t}$ is the intrinsic scalar curvature of $\Sigma_t$.
\end{lemma}

\begin{proof}
The Gauss equation and \eqref{eq:riccati} give
\[
 \begin{split}
 S&=S_{\Sigma_t}+2\Ric(\nabla\rho,\nabla\rho)-H^2+\norm{A_t}^2,\\
 \frac{J''}{J}&=H^2-\norm{A_t}^2-\Ric(\nabla\rho,\nabla\rho).
 \end{split}
\]
Consequently,
\[
 \begin{aligned}
 S&=S_{\Sigma_t}+H^2-\norm{A_t}^2-2\frac{J''}{J}\\
  &\le S_{\Sigma_t}+(n-1)(n-2)\frac{(j')^2}{j^2}-2\frac{J''}{J}.
 \end{aligned}
\]
Here $\norm{A_t}^2\ge \frac{H^2}{n-1}$ and $H=\frac{(n-1)j'}{j}$ were used.
Under the polar parametrization the induced area element is
$\dd\sigma_t=J(t,\theta)\dd\theta$, and
\[
 \mathcal A(t)=\int_{\Sph^{n-1}}J(t,\theta)\dd\theta
             =\int_{\Sph^{n-1}}j(t,\theta)^{n-1}\dd\theta.
\]
Since $p$ is a pole, differentiation under the integral gives
\[
 \mathcal A'(t)=\int_{\Sph^{n-1}}J'(t,\theta)\dd\theta
             =(n-1)\int_{\Sph^{n-1}}j^{n-2}j'\dd\theta.
\]
Integrate the preceding scalar inequality on
$B_p(r)\setminus\overline{B_p(\varepsilon)}$. The $J''$ term is
$-2\mathcal A'(r)+2\mathcal A'(\varepsilon)$.
By Lemma~\ref{lem:polar-concavity},
$0\le\mathcal A'(\varepsilon)\le n(n-1)\omega_n\varepsilon^{n-2}$.
Since $n\ge3$, this tends to zero. Letting $\varepsilon\downarrow0$
proves \eqref{eq:polar-integral}.
\end{proof}

\subsection{Outward minimality and the area radius}
\label{sec:pole-bound}

We use the following Minkowski inequality in its full-dimensional
form, so that the same input applies to both curvature regimes.

\begin{definition}
For a bounded set $E\subset M$, let $P(E)=|D\chi_E|(M)$ denote
its perimeter. For a set of finite perimeter,
$P(E)=\mathcal H_g^{n-1}(\partial^*E)$, where $\partial^*E$ is
its reduced boundary. In particular, $P(\Omega)=|\partial\Omega|$
when $\Omega$ has smooth boundary. A bounded domain $\Omega$ is
\emph{outward minimizing} if $P(\Omega)\le P(E)$ for every
bounded set $E$ of finite perimeter containing $\Omega$ up to a
set of measure zero.
\end{definition}

\begin{quotedtheorem}[Benatti--Fogagnolo--Mazzieri,
\cite{BFM24}, Corollary~4.6]
Suppose $(N^n,h)$, $n\ge3$, is complete, has $\Ric_h\ge0$,
and satisfies $V_{N^n}>0$. If $\Omega\subset N$ is a bounded
outward minimizing domain with smooth boundary, then
\begin{equation}\label{eq:bfm-minkowski}
 V_{N^n}^{\frac{1}{n-1}}
 \left(\frac{|\partial\Omega|}{n\omega_n}\right)^{\frac{n-2}{n-1}}
 \le\frac{1}{n(n-1)\omega_n}
                  \int_{\partial\Omega}H\dd\sigma.
\end{equation}
Here $H$ is the sum of the principal curvatures with respect to
the outward unit normal.
\end{quotedtheorem}

\begin{lemma}\label{lem:outward}
Under $\Ric\ge0$ and the pole assumption, every ball $B_p(r)$ is
outward minimizing: each bounded set $E$ of finite perimeter
containing $B_p(r)$ up to a null set satisfies
$P(E)\ge P(B_p(r))=\mathcal A(r)$.
\end{lemma}

\begin{proof}
The vector field $\nabla\rho$ has length one and
$\diver(\nabla\rho)=H\ge0$ away from $p$.
For $0<\varepsilon<r$, Gauss--Green gives
\[
 \begin{split}
 P(E)&\ge\int_{\partial^*E}
                 \ip{\nabla\rho}{\nu_E}\dd\sigma\\
     &=\mathcal A(\varepsilon)
                  +\int_{E\setminus B_p(\varepsilon)}H\dd\mu_g\\
     &\ge\mathcal A(\varepsilon)
                  +\int_{B_p(r)\setminus B_p(\varepsilon)}H\dd\mu_g
       =\mathcal A(r).
 \end{split}
\]
The field is smooth outside $B_p(\varepsilon)$; a compactly
supported cutoff equal to one near $E$ justifies its use in
Gauss--Green.
\end{proof}

\begin{lemma}\label{lem:area-radius}
For $v\in[0,1]$, put $a=v^{\frac{1}{n-1}}$ and define
\begin{equation}\label{eq:area-radius}
 \eta(t)=\left(\frac{\mathcal A(t)}{n\omega_n}\right)^{\frac{1}{n-1}}.
\end{equation}
Then, for every $t>0$,
\begin{equation}\label{eq:area-bounds}
 \eta'(t)\ge a,\qquad at\le\eta(t)\le t,
 \qquad \mathcal A'(t)\ge n(n-1)\omega_n a\eta(t)^{n-2}.
\end{equation}
\end{lemma}

\begin{proof}
If $v=0$, then $a=0$ and Lemma~\ref{lem:polar-concavity} gives
$\mathcal A'\ge0$, hence $\eta'\ge0=a$.
If $v>0$, apply \eqref{eq:bfm-minkowski} to $B_p(t)$, using
Lemma~\ref{lem:outward} and
$\mathcal A'(t)=\int_{\Sigma_t}H\dd\sigma_t$. This gives
\[
 a\eta(t)^{n-2}\le\frac{\mathcal A'(t)}{n(n-1)\omega_n}
                         =\eta(t)^{n-2}\eta'(t).
\]
Since $\eta(t)>0$, we again obtain $\eta'\ge a$.
In both cases, integration from $\eta(0)=0$ yields $\eta(t)\ge at$,
while $J=j^{n-1}\le t^{n-1}$ gives $\eta(t)\le t$.
Finally, $\mathcal A'=n(n-1)\omega_n\eta^{n-2}\eta'$ proves
the last inequality in \eqref{eq:area-bounds}.
\end{proof}

\subsection{Intrinsic curvature of distance spheres}
\label{polehd:sec:sphere}

We now impose the hypotheses of Theorem~\ref{thm:pole}:
either $n=3$ and $\Ric\ge0$, or $n\ge4$ and $K\ge0$.

\begin{lemma}\label{polehd:lem:sphere-scalar}
For every $t>0$,
\begin{equation}\label{polehd:eq:sphere-scalar}
 \int_{\Sigma_t}S_{\Sigma_t}\dd\sigma_t
              \le n(n-1)(n-2)\omega_n t^{n-3}.
\end{equation}
When $n=3$, equality holds and both sides equal $8\pi$.
\end{lemma}

\begin{proof}
If $n=3$, the pole assumption makes $\Sigma_t$ diffeomorphic
to $\Sph^2$. Since scalar curvature is twice Gaussian curvature,
Gauss--Bonnet gives
\[
 \int_{\Sigma_t}S_{\Sigma_t}\dd\sigma_t
 =4\pi\chi(\Sigma_t)=8\pi.
\]
This proves \eqref{polehd:eq:sphere-scalar} for $n=3$.

Suppose now that $n\ge4$ and $K\ge0$.
Hessian comparison gives
\[
 \Hess\rho\le\frac{1}{\rho}
       \bigl(g-\dd\rho\otimes\dd\rho\bigr).
\]
For $X\in T\Sigma_t$, we have $\dd\rho(X)=0$ and
$\ip{A_tX}{X}=\Hess\rho(X,X)$. Hence
$A_t\le\frac{1}{t}\operatorname{Id}$.
For a radial-parallel unit vector field $e\perp\nabla\rho$,
the Riccati equation gives
\[
 z'(t)=-\norm{A_te}^2-K(e\wedge\nabla\rho)\le-z(t)^2,
 \qquad z(t)=\ip{A_te}{e}.
\]
If $z(t_0)<0$, then $z'\le0$ implies
$z(t)\le z(t_0)<0$ for every $t\ge t_0$.
Integrating $\left(-\frac{1}{z}\right)'\le-1$ would then make
$-\frac{1}{z(t)}$ nonpositive in finite time, contrary to
the smoothness of polar coordinates from a pole.
Thus $A_t\ge0$.

By the Gauss equation, the induced metric $g_t$ on $\Sigma_t$
has nonnegative sectional curvature. Pull it back to $\Sph^{n-1}$
by the polar parametrization and set $h_t=\frac{1}{t^2}g_t$.
Then $K_{h_t}\ge0$ and
\[
 \partial_t h_t
 =2h_t\left(\left(A_t-\frac{1}{t}\operatorname{Id}\right)
                  \cdot,\cdot\right)\le0.
\]
The first variation formula for total scalar curvature on a
closed manifold \cite[Proposition~4.17, p.~120]{Besse87} gives
\begin{equation}\label{polehd:eq:variation}
 \frac{d}{dt}\int_{\Sph^{n-1}}S_{h_t}\dd\mu_{h_t}
 =\int_{\Sph^{n-1}}
   \left\langle\frac{1}{2}S_{h_t}h_t-\Ric_{h_t},
                 \partial_t h_t\right\rangle_{h_t}\dd\mu_{h_t}.
\end{equation}
Applying the pointwise argument following \eqref{eq:tensor}
to $(\Sph^{n-1},h_t)$ shows that
$\frac{1}{2}S_{h_t}h_t-\Ric_{h_t}\ge0$.
Its inner product with $\partial_t h_t\le0$ is therefore
nonpositive. Since $h_t$ converges smoothly to the unit round
metric as $t\downarrow0$,
\[
 \int_{\Sph^{n-1}}S_{h_t}\dd\mu_{h_t}
 \le (n-1)(n-2)|\Sph^{n-1}|=n(n-1)(n-2)\omega_n.
\]
Finally,
\[
 \int_{\Sph^{n-1}}S_{h_t}\dd\mu_{h_t}
 =t^{3-n}\int_{\Sigma_t}S_{\Sigma_t}\dd\sigma_t,
\]
which proves \eqref{polehd:eq:sphere-scalar} for $n\ge4$.
\end{proof}

\subsection{The proof of the pole estimate}
\label{polehd:sec:all}

\begin{proof}[Proof of Theorem~\ref{thm:pole}]
Put $a=v^{\frac{1}{n-1}}$ and use the area radius
$\eta$ from \eqref{eq:area-radius}, which is defined also when $v=0$.
By Lemma~\ref{polehd:lem:sphere-scalar} and $0\le j'\le1$,
the common polar estimate \eqref{eq:polar-integral} gives
\[
 I_p(r)\le n(n-1)\omega_n r^{n-2}
       +(n-1)\int_{\Sph^{n-1}}j(r,\theta)^{n-2}\dd\theta
       -2\mathcal A'(r).
\]
Here the radial integration uses
\[
 \int_0^r j^{n-3}(j')^2\dd t
 \le\int_0^r j^{n-3}j'\dd t
 =\frac{j(r,\theta)^{n-2}}{n-2}.
\]
H\"older's inequality yields
\[
 \begin{aligned}
 \int_{\Sph^{n-1}}j(r,\theta)^{n-2}\dd\theta
 &\le(n\omega_n)^{\frac{1}{n-1}}
       \left(\int_{\Sph^{n-1}}j(r,\theta)^{n-1}\dd\theta
       \right)^{\frac{n-2}{n-1}}\\
 &=n\omega_n\eta(r)^{n-2}.
 \end{aligned}
\]
Consequently,
\[
 I_p(r)\le n(n-1)\omega_n
       \bigl[r^{n-2}+\eta(r)^{n-2}\bigr]-2\mathcal A'(r).
\]
Since $\eta(r)\le r$ and $\mathcal A'(r)\ge0$, this already gives
the uniform estimate \eqref{eq:pole-higher-linear}:
\begin{equation}\label{eq:pole-uniform}
 I_p(r)\le2n(n-1)\omega_n r^{n-2}.
\end{equation}
Using \eqref{eq:area-bounds} in the preceding estimate also gives
\begin{equation}\label{polehd:eq:linear-area}
 I_p(r)\le n(n-1)\omega_n
       \bigl[r^{n-2}+(1-2a)\eta(r)^{n-2}\bigr].
\end{equation}

If $0\le a\le\frac{1}{2}$, use $\eta(r)\le r$ to obtain
\[
 I_p(r)\le2n(n-1)\omega_n(1-a)r^{n-2}.
\]
This includes $v=0$ and proves \eqref{eq:piecewise}.
If $\frac{1}{2}<a\le1$, the coefficient $1-2a$ is negative,
so $\eta(r)\ge ar$ gives
\[
 I_p(r)\le n(n-1)\omega_n
       \bigl(1+a^{n-2}-2a^{n-1}\bigr)r^{n-2},
\]
which is \eqref{eq:pole-higher}.

For each exact $v=a^{n-1}\in[0,2^{-(n-1)}]$, the rotational examples
\eqref{eq:rot-example} in Theorem~\ref{thm:rotational} have
$K\ge0$, have a pole at their center, and have integrals
approaching the first bound. Thus the first branch is sharp
in every dimension under consideration.
At $v=1$, equality holds in Euclidean space.
No sharpness is asserted for the second branch when
$2^{-(n-1)}<v<1$.
\end{proof}

\begin{corollary}\label{polehd:cor:linear}
Under the hypotheses of Theorem~\ref{thm:pole},
\begin{equation}\label{polehd:eq:linear}
 I_p(r)\le2n(n-1)\omega_n(1-v)r^{n-2}.
\end{equation}
\end{corollary}

\begin{proof}
Put $a=v^{\frac{1}{n-1}}$. For $a\le \frac{1}{2}$, use
$1-a\le1-a^{n-1}$. For $a>\frac{1}{2}$, use
$1+a^{n-2}-2a^{n-1}\le2(1-a^{n-1})$.
\end{proof}

\begin{remark}
The condition $K\ge0$ is used only in
Lemma~\ref{polehd:lem:sphere-scalar}: Gauss--Bonnet gives the exact
distance-sphere curvature integral in dimension three, whereas
in higher dimensions, we use $K\ge0$ to obtain its upper bound.
\end{remark}

\appendix

\section{Convex hypersurfaces with large unit-ball integrals}
\label{sec:centers}
\label{sec:double}

We construct smooth complete convex hypersurfaces in $\R^4$
whose unit-ball integrals approach
\[
 \mathcal E(\theta)=16\pi+8\pi\theta\cos\theta,
 \qquad 0<\theta<\frac{\pi}{2}.
\]
The construction joins two flat copies of a smoothed convex
cap--cylinder domain by an explicit convex band in $\R^4$.
The calculation separates the contributions of the spherical
cap, cylinder, and smoothed junction in the thin-double limit.

Fix $a=\cos\theta$ and $b=\sin\theta$, and consider
\begin{equation}\label{eq:cap-cylinder}
 \Omega_\theta=
 \{(y,z)\in\R^2\times\R:
       |y|\le a,\ z\ge-\sqrt{1-|y|^2}\}.
\end{equation}
Its boundary is a unit spherical cap joined to a cylinder of
radius $a$ at $z=-b$. Smooth this junction by replacing
\begin{equation}\label{eq:cap-profile}
 R_\theta(z)=
 \begin{cases}
 \sqrt{1-z^2},&-1\le z\le-b,\\
 a,&z\ge-b
 \end{cases}
\end{equation}
with a smooth nondecreasing concave profile $R_{\theta,\eta}$
away from $z=-1$, equal to $R_\theta$ outside an
$\eta$-neighborhood of $-b$. Such a profile is obtained from a smooth concave minimum
of $\sqrt{1-z^2}$ and $a$ near the junction, equal to the
ordinary minimum outside a smaller neighborhood. Denote the resulting smooth convex domain by
$\Omega_{\theta,\eta}$.

Choose a smooth increasing function $\beta:[0,1]\to[0,\pi]$
such that $\beta(0)=0$, $\beta(1)=\pi$,
$\beta(1-t)=\pi-\beta(t)$, and all positive-order
derivatives vanish at the endpoints. Put
\[
 u(t)=\int_0^t\cos\beta(s)\dd s,
 \qquad w(t)=\int_t^{\frac{1}{2}}\sin\beta(s)\dd s,
 \qquad h=w(0)>0.
\]
Then $u\ge0$, $u(0)=u(1)=0$, and $(u,w)$ is a unit-speed
convex arc from $(0,h)$ to $(0,-h)$, tangent to horizontal
lines to infinite order at its endpoints.
For $\varepsilon>0$, join the two flat faces
$\Omega_{\theta,\eta}\times\{\pm\varepsilon h\}$ by
\begin{equation}\label{eq:explicit-hypersurface}
 X_\varepsilon(q,t)
   =\bigl(q+\varepsilon u(t)\nu(q),\varepsilon w(t)\bigr),
 \qquad q\in\partial\Omega_{\theta,\eta},\quad 0\le t\le1,
\end{equation}
where $\nu$ is the outward unit normal.
This defines a smooth convex hypersurface
$M_{\theta,\eta,\varepsilon}\subset\R^4$.
Indeed, the horizontal sections of the enclosed set are parallel
convex domains, and their offset is a concave function of the
vertical coordinate. The endpoint flatness of $\beta$ gives
smooth attachment to both faces. The hypersurface is complete
and has a cylindrical end, so $K\ge0$ and $V_{M^3}=0$.
We take
\begin{equation}\label{eq:double-center}
 p_\varepsilon=(0,0,0,\varepsilon h).
\end{equation}
As $\varepsilon\downarrow0$ and then $\eta\downarrow0$, the
two faces approach two copies of $\Omega_\theta$, glued along
their common boundary; this is the thin double in
Figure~\ref{fig:convex-double}.

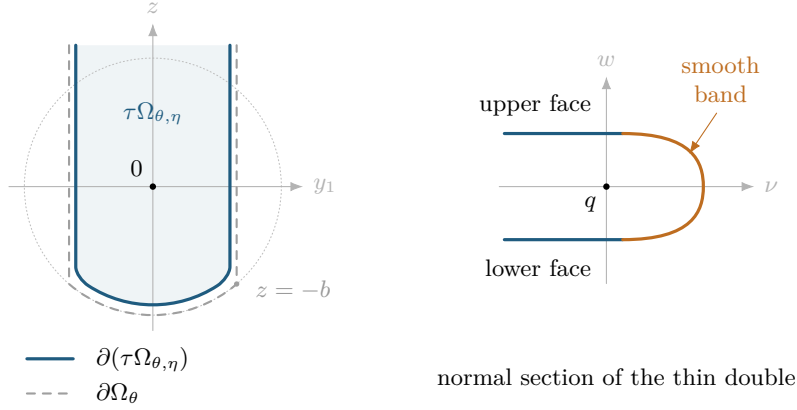
\begin{figure}[htbp]
	\centering
	\begin{tikzpicture}[font=\footnotesize,>=Latex,
		line cap=round,line join=round]
		
		\begin{scope}[x=1.7cm,y=1.7cm]
			\def\figTheta{49.3}
			\def\figTau{0.92}
			\def\figEta{0.08}
			
			\pgfmathsetmacro{\figA}{cos(\figTheta)}
			\pgfmathsetmacro{\figB}{sin(\figTheta)}
			\pgfmathsetmacro{\figZlo}{-\figB-\figEta}
			\pgfmathsetmacro{\figZhi}{-\figB+\figEta}
			\pgfmathsetmacro{\figRlo}{sqrt(1-\figZlo*\figZlo)}
			\pgfmathsetmacro{\figCr}{(\figRlo+\figA)/2}
			\pgfmathsetmacro{\figCz}{\figZlo+
				(\figA-\figRlo)*\figRlo/(-2*\figZlo)}
			\pgfmathsetmacro{\figVz}{\figZhi-2*\figEta/3}
			\pgfmathsetmacro{\figStart}{180+asin(-\figZlo)}
			\pgfmathsetmacro{\figEnd}{360-asin(-\figZlo)}
			\pgfmathsetmacro{\figTop}{1.1/\figTau}
			
			\def\figSmoothPath{
				(-\figA,\figTop)--(-\figA,\figZhi)
				.. controls (-\figA,\figVz) and (-\figCr,\figCz)
				.. (-\figRlo,\figZlo)
				arc[start angle=\figStart,end angle=\figEnd,radius=1]
				.. controls (\figCr,\figCz) and (\figA,\figVz)
				.. (\figA,\figZhi)--(\figA,\figTop)
			}
			
			\begin{scope}[scale=\figTau]
				\path[fill=geoblue!7] \figSmoothPath --cycle;
			\end{scope}
			
			\draw[gray!50,densely dotted] (0,0) circle[radius=1];
			\draw[gray!60,->] (-1.12,0)--(1.17,0)
			node[right] {$y_1$};
			\draw[gray!60,->] (0,-1.12)--(0,1.27)
			node[above] {$z$};
			
			\draw[gray!75,thick,dashed]
			(-\figA,1.1)--(-\figA,-\figB)
			arc[start angle=180+\figTheta,
			end angle=360-\figTheta,radius=1]
			--(\figA,1.1);
			
			\begin{scope}[scale=\figTau]
				\draw[geoblue,very thick] \figSmoothPath;
			\end{scope}
			
			\fill (0,0) circle[radius=1.2pt]
			node[above left] {$0$};
			\node[geoblue] at (0,.57) {$\tau\Omega_{\theta,\eta}$};
			\fill[gray!75] (\figA,-\figB) circle[radius=1pt];
			\node[anchor=west,gray!80] at (.72,-.8) {$z=-b$};
			
			\draw[geoblue,very thick] (-.98,-1.34)--(-.62,-1.34);
			\node[anchor=west] at (-.52,-1.34)
			{$\partial(\tau\Omega_{\theta,\eta})$};
			\draw[gray!75,thick,dashed] (-.98,-1.61)--(-.62,-1.61);
			\node[anchor=west] at (-.52,-1.61) {$\partial\Omega_\theta$};
		\end{scope}
		
		\begin{scope}[shift={(6.0cm,0cm)},x=1.35cm,y=1.35cm]
			\draw[gray!60,->] (-1.05,0)--(1.45,0)
			node[right] {$\nu$};
			\draw[gray!60,->] (0,-.95)--(0,1.08)
			node[above] {$w$};
			\draw[geoblue,very thick] (-1,.52)--(.16,.52);
			\draw[geoblue,very thick] (-1,-.52)--(.16,-.52);
			\draw[geoamber,very thick] (.16,.52)
			.. controls (.85,.52) and (.95,.24) .. (.95,0)
			.. controls (.95,-.24) and (.85,-.52) .. (.16,-.52);
			\fill (0,0) circle[radius=1.2pt]
			node[below left] {$q$};
			\node[anchor=east] at (-.05,.8) {upper face};
			\node[anchor=east] at (-.05,-.8) {lower face};
			\node[geoamber,align=center] at (1.12,1.05) {smooth\\band};
			\draw[geoamber,->] (1.08,.72)--(.81,.32);
			\node at (.1,-1.86) {normal section of the thin double};
		\end{scope}
		
	\end{tikzpicture}
	\caption{Left: the smoothed profile after a homothety of factor
		$\tau<1$ (solid), given by
		$y_1=\pm\tau R_{\theta,\eta}(\frac{z}{\tau})$, and the limiting
		cap--cylinder profile as $\eta\downarrow0$ and $\tau\uparrow1$
		(dashed). The dotted circle is the unit circle.
		The local smoothing is schematic.
		Right: the connecting band in $\R^4$.}
	\label{fig:convex-double}
\end{figure}

We compute in the thin-double limit, taking $\varepsilon\downarrow0$
first with $\eta$ fixed. Write $H=\kappa_1+\kappa_2$ on
$\partial\Omega_{\theta,\eta}$. Both flat faces have zero scalar
curvature. On the band, the transverse principal curvature is
$\frac{\beta'}{\varepsilon}$, the tangential principal curvatures tend to
$\kappa_i\sin\beta$, and $\varepsilon^{-1}\dd\mu_g$ tends to
$\dd t\dd\sigma$. Thus $S=2\sum_{i<j}\lambda_i\lambda_j$ gives,
for every relatively compact patch $A$,
\begin{equation}\label{eq:band-integral}
 \lim_{\varepsilon\downarrow0}
 \int_{X_\varepsilon(A\times[0,1])}S\dd\mu_g
 =2\int_0^1\beta'\sin\beta\dd t\int_AH\dd\sigma
 =4\int_AH\dd\sigma.
\end{equation}
The product of the two tangential curvatures has vanishing
integral because they remain bounded for fixed $\eta$ and the
band thickness tends to zero.

In the limiting model, the spherical cap has $H=2$ and area
$2\pi(1-b)$, while the cylinder between $z=-b$ and $z=b$
has $H=\frac{1}{a}$ and area $4\pi ab$. Their contributions to the
scalar curvature integral are therefore
\[
 16\pi(1-b)\qquad\text{and}\qquad16\pi b.
\]
At the junction circle, the outward normal turns through
$\theta$ and the circumference is $2\pi a$. To compute the
contribution of its smoothing $\Gamma_\eta$, write $s$ for
meridian arclength, $r_\eta$ for the distance to the axis, and
$\alpha$ for the signed outward normal angle measured from the radial
direction, increasing along the meridian toward the cylindrical end. The principal curvatures are $\frac{\dd\alpha}{\dd s}$ and
$\frac{\cos\alpha}{r_\eta}$, and $\dd\sigma=2\pi r_\eta\dd s$. Hence
\[
 \int_{\Gamma_\eta}H\dd\sigma
 =2\pi\int r_\eta\dd\alpha+2\pi\int\cos\alpha\dd s
 \longrightarrow2\pi a\theta.
\]
Indeed, $r_\eta\to a$ uniformly, the total normal rotation tends
to $\theta$, and the meridian length tends to zero. Thus the
junction contributes $8\pi a\theta$ to the thin-double scalar
curvature integral: the shrinking width does not remove the
integrated meridian curvature, which records the normal rotation.

To include the cap and junction in an open unit geodesic ball,
first scale the hypersurface by $0<\tau<1$. A band point over
$q$ has distance from $p_\varepsilon$ between
$|q|-\varepsilon\max u$ and $|q|+\varepsilon$, by projection to
$\R^3$ and by a path along the upper face and then the band.
For small $\eta$, the sphere of radius $\frac{1}{\tau}$ meets only the
cylindrical part, transversely. Consequently,
\begin{equation}\label{eq:double-ball-limit}
 \lim_{\varepsilon\downarrow0}
 I_{\tau p_\varepsilon}^{\tau M_{\theta,\eta,\varepsilon}}(1)
 =4\tau\int_{\partial\Omega_{\theta,\eta}
                  \cap B^{\R^3}_{\frac{1}{\tau}}(0)}H\dd\sigma.
\end{equation}
Here the superscript specifies the induced metric on the scaled
hypersurface. The included cylinder has limiting height
$b+\sqrt{\tau^{-2}-a^2}$. Letting $\eta\downarrow0$ and then
$\tau\uparrow1$ therefore gives
\begin{equation}\label{eq:double-value}
 \begin{aligned}
 &\lim_{\tau\uparrow1}\lim_{\eta\downarrow0}
   \lim_{\varepsilon\downarrow0}
 I_{\tau p_\varepsilon}^{\tau M_{\theta,\eta,\varepsilon}}(1)\\
 &\quad=4\lim_{\tau\uparrow1}\tau
 \left[4\pi(1-b)+2\pi\bigl(b+\sqrt{\tau^{-2}-a^2}\bigr)
                  +2\pi a\theta\right]\\
 &\quad=16\pi(1-b)+16\pi b+8\pi a\theta
 =16\pi+8\pi\theta\cos\theta.
 \end{aligned}
\end{equation}

Maximizing $\theta\cos\theta$ gives the following proposition.

\begin{proposition}\label{prop:double-lower}
For every $\epsilon>0$, there exist a smooth complete
noncompact convex hypersurface $M^3\subset\R^4$ and a point
$p\in M$ such that
\begin{equation}\label{eq:convex-lower}
 K\ge0,\qquad V_{M^3}=0,\qquad
 I_p(1)>16\pi+8\pi\theta_0\cos\theta_0-\epsilon,
\end{equation}
where $\theta_0\in(0,\frac{\pi}{2})$ satisfies
\begin{equation}\label{eq:Cstar}
    \theta_0\tan\theta_0=1.
\end{equation}
Numerically, $\theta_0\approx0.860$ and
$16\pi+8\pi\theta_0\cos\theta_0\approx64.367$.
\end{proposition}

Since this limiting value is larger than $16\pi$, the uniform
estimate \eqref{eq:pole-uniform} fails when the pole assumption
on the center is omitted.

\end{document}